\documentclass{article}

\newcommand{\textlineskip}{\baselineskip=11.5dd plus0.5dd minus0.2dd}

\def\qed{\hbox{${\vcenter{\vbox{
   \hrule height 0.4pt\hbox{\vrule width 0.4pt height 6pt
   \kern5pt\vrule width 0.4pt}\hrule height 0.4pt}}}$}}

\catcode`\@=11

\font\twrm=cmr12

\font\eightrm=cmr8

\newcounter{subsect}

\def\section#1{\par \bigskip \setcounter{subsect}{0}
\setcounter{theorem}{0} 
\setcounter{equation}{0}
\addtocounter{section}{1}\begin{center}\thesection.
\uppercase{#1} \end{center} \par \smallskip }

\def\subsection#1{\par \medskip
\addtocounter{subsect}{1}\begin{center}{ \eightrm
\thesection.\thesubsect. \uppercase{#1}} \end{center} \par
\smallskip }



\usepackage{amsmath}
\usepackage{amsthm}
\usepackage{amssymb}
\usepackage{hyperref}
\usepackage{caption}
\usepackage{subcaption}
\usepackage{graphicx}
\usepackage{lineno}

\newtheorem{theorem}{Theorem}

\newtheorem{proposition}[theorem]{Proposition}
\newtheorem{hypothesis}[theorem]{Hypothesis}

\newtheorem{corollary}[theorem]{Corollary}
\newtheorem{definition}[theorem]{Definition}

\newtheorem{algorithm}[theorem]{Algorithm}

\newtheorem{remark}[theorem]{Remark}

\numberwithin{equation}{section}
\numberwithin{theorem}{section}

\newcommand{\abs}[1]{|#1|}
\newcommand{\set}[1]{\left\{#1\right\}}

\DeclareMathOperator{\V}{V}
\DeclareMathOperator{\E}{E}

\begin{document}

\normalsize\textlineskip
\thispagestyle{empty}

\begin{center} {\twrm
ANNUAIRE DE L'UNIVERSIT\'E DE  SOFIA ``St.~Kl.~OHRIDSKI" \\[3pt]
FACULT\'E DE MATH\'EMATIQUES ET  INFORMATIQUE

\vspace*{70pt}
\uppercase{Small minimal (3, 3)-Ramsey graphs}\\
\vspace*{0.035truein}
}
\vspace*{25pt}
{\eightrm\uppercase{Aleksandar Bikov}}
\end{center}

\vspace*{30pt}

{
\parindent0pt \footnotesize \leftskip20pt
\rightskip20pt  \baselineskip10pt
We say that $G$ is a $(3, 3)$-Ramsey graph if every $2$-coloring of the edges of $G$ forces a monochromatic triangle. The $(3, 3)$-Ramsey graph $G$ is minimal if $G$ does not contain a proper $(3, 3)$-Ramsey subgraph. In this work we find all minimal $(3, 3)$-Ramsey graphs with up to 13 vertices with the help of a computer, and we obtain some new results for these graphs. We also obtain new upper bounds on the independence number and new lower bounds on the minimum degree of arbitrary $(3, 3)$-Ramsey graphs.
\smallskip

{\bf 2000 Math.\ Subject Classification}. 05C55
\par
}

\vspace*{20pt}


\section{Introduction}

In this work only finite, non-oriented graphs without loops and multiple edges are considered. The following notations are used:

$\V(G)$ - the vertex set of $G$;

$\E(G)$ - the edge set of $G$;

$\overline{G}$ - the complement of $G$;

$\omega(G)$ - the clique number of $G$;

$\alpha(G)$ - the independence number of $G$;

$\chi(G)$ - the chromatic number of $G$;

$N_G(v), v \in \V(G)$ - the set of all vertices of G adjacent to $v$;

$d(v), v \in \V(G)$ - the degree of the vertex $v$, i.e. $d(v) = \abs{N_G(v)}$;

$G(v), v \in \V(G)$ - subgraph of $G$ induced by $N_G(v)$;

$G-v, v \in \V(G)$ - subgraph of $G$ obtained from $G$ by deleting the vertex $v$ and all edges incident to $v$;

$G-e, e \in \E(G)$ - subgraph of $G$ obtained from $G$ by deleting the edge $e$;

$\Delta(G)$ - the maximum degree of $G$;

$\delta(G)$ - the minimum degree of $G$;

$K_n$ - complete graph on $n$ vertices;

$C_n$ - simple cycle on $n$ vertices;

$G_1 + G_2$ - graph $G$ for which $\V(G) = \V(G_1) \cup \V(G_2)$ and $\E(G) = \E(G_1) \cup \E(G_2) \cup E'$, where $E' = \set{[x, y] : x \in \V(G_1), y \in \V(G_2)}$, i.e. $G$ is obtained by connecting every vertex of $G_1$ to every vertex of $G_2$.

All undefined terms can be found in \cite{Har69}.\\

Each partition
\begin{equation}
\label{equation: r-coloring}
\E(G) = E_1 \cup ... \cup E_r, E_i \cap E_j = \emptyset, i \neq j
\end{equation}
is called an $r$-coloring of the edges of $G$. We say that $H \subseteq G$ is a monochromatic subgraph of color $i$ in the $r$-coloring (\ref{equation: r-coloring}), if $\E(H) \subseteq E_i$.

Let $p$ and $q$ be positive integers, $p \geq 2$ and $q \geq 2$. The expression $G \rightarrow (p, q)$ means that for every $2$-coloring of $\E(G)$ there exists a $p$-clique of the first color or a $q$-clique of the second color. If $G \rightarrow (p, q)$, we say that $G$ is a $(p, q)$-Ramsey graph. Similarly, the expression $G \rightarrow (p_1, ..., p_r)$ is defined for the $r$-colorings of $\E(G)$.

The smallest possible integer $n$ for which $K_n \rightarrow (p, q)$ is called a Ramsey number and is denoted by $R(p, q)$. The Ramsey numbers $R(p_1, p_2, ..., p_r)$ are defined similarly.

The existence of Ramsey numbers was proved by Ramsey in \cite{Ram30}. Only a few exact values of Ramsey numbers are known (see \cite{Rad14}). In this work we will use the equality $R(3, 3) = 6$. This equality means that $K_6 \rightarrow (3, 3)$ and $K_5 \not\rightarrow (3, 3)$. It is clear, that if $\omega(G) \geq 6$, then $G \rightarrow (3, 3)$. In \cite{EH67} Erd\"os and Hajnal posed the following problem:
\begin{center}
\label{question: Erdos and Hajnal}
\emph{Is there a graph $G \rightarrow (3, 3)$ with $\omega(G) < 6$ ?}
\end{center}

\begin{figure}[h]
	\centering
	\includegraphics[height=160px,width=160px]{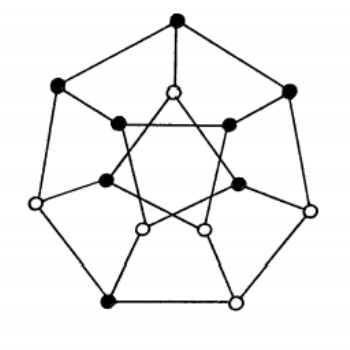}
	\caption{The complement of the P\'osa graph from \cite{GS71}}\label{figure: Posa}
\end{figure}

The first example of a graph which gives a positive answer to this question was given by P\'osa. The complement of this graph is presented on Figure \ref{figure: Posa}. P\'osa did not publish this result himself, but the graph was included in \cite{GS71}. Later, Graham \cite{Gra68} constructed the smallest possible example of such a graph, namely $K_3+C_5$. It is easy to see that the P\'osa graph contains $K_3+C_5$ (it is the subgraph induced by the black vertices on Figure \ref{figure: Posa}).

There exist $(3, 3)$-Ramsey graphs which do not contain $K_5$. These graphs have at least 15 vertices \cite{PRU99}. The first 15-vertex $(3, 3)$-Ramsey graph which does not contain $K_5$ was constructed by Nenov \cite{Nen81a}. This graph is obtained from the graph $\Gamma$ presented on Figure \ref{figure: Nenov_14} by adding a new vertex which is adjacent to all vertices of $\Gamma$.

\begin{figure}[h]
	\centering
	\includegraphics[height=210px,width=210px]{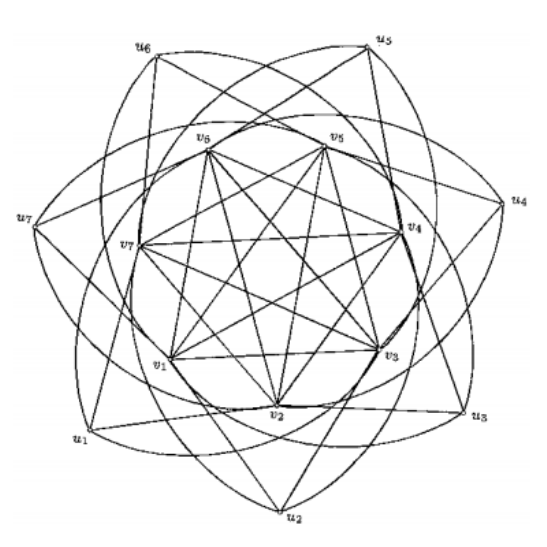}
	\caption{The Nenov graph $\Gamma$ from \cite{Nen81a}}
	\label{figure: Nenov_14}
\end{figure}

Folkman constructed a graph $G \rightarrow (3, 3)$ with $\omega(G) = 3$ \cite{Fol70}. The minimum number of vertices of such graphs is not known. To date, we know only that this minimum is between 19 and 786, \cite{RX07} and \cite{LRX12}.

Obviously, if $H$ is a $(p, q)$-Ramsey graph, then its every supergraph $G$ is also a $(p, q)$-Ramsey graph.
\begin{definition}
\label{definition: minimal (p, q)-Ramsey graph}
We say that $G$ is a minimal $(p, q)$-Ramsey graph if $G \rightarrow (p, q)$ and $H \not\rightarrow (p, q)$ for each proper subgraph $H$ of $G$.
\end{definition}

It is easy to see that $K_6$ is a minimal $(3, 3)$-Ramsey graph and there are no minimal $(3, 3)$-Ramsey graphs with 7 vertices. The only such 8-vertex graph is the Graham graph $K_3+C_5$, and there is only one such 9-vertex graph, Nenov \cite{Nen79} (see Figure \ref{figure: Nenov_9}).

\begin{figure}[h]
	\begin{minipage}{.45\textwidth}
		\centering
		\includegraphics[trim={0 470 0 0},clip,height=160px,width=160px]{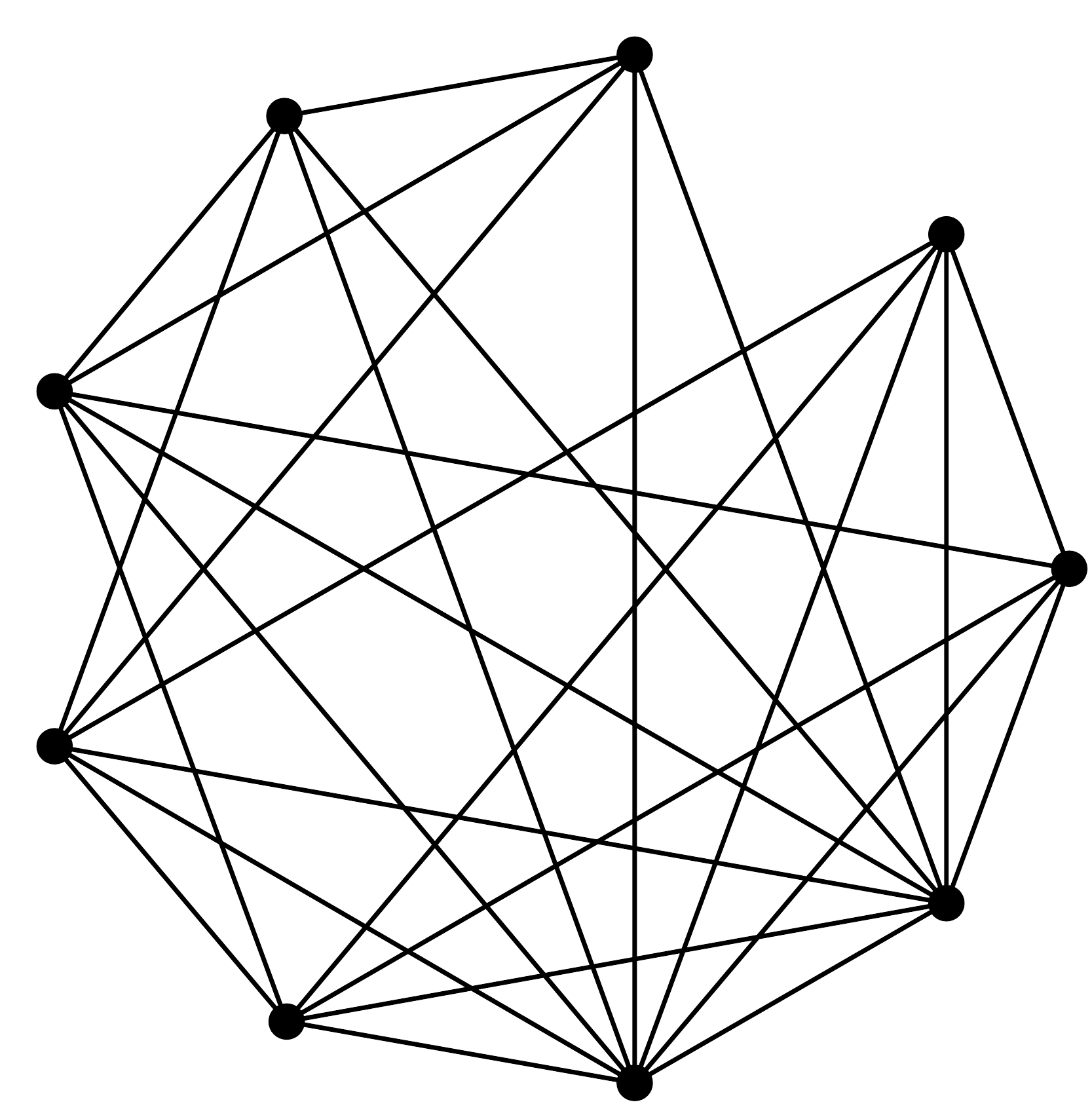}
		\caption{9-vertex minimal\hfill\break $(3, 3)$-Ramsey graph}
		\label{figure: Nenov_9}
	\end{minipage}\hfill
	\begin{minipage}{.45\textwidth}
		\centering
		\includegraphics[trim={0 470 0 0},clip,height=160px,width=160px]{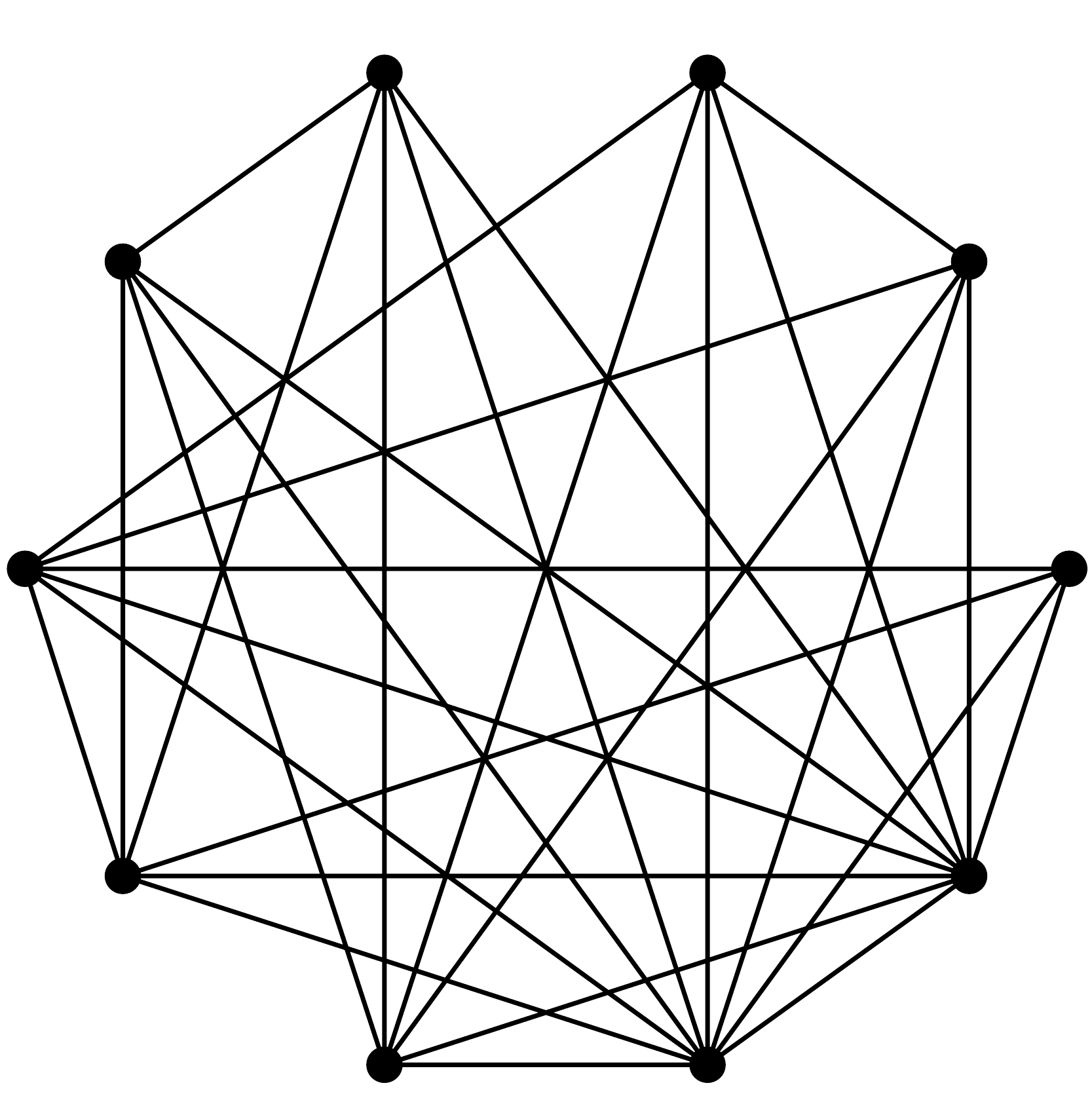}
		\caption{10-vertex minimal\hfill\break $(3, 3)$-Ramsey graph}
		\label{figure: Nenov_10}
	\end{minipage}
\end{figure}

For each pair of positive integers $p \geq 3$, $q \geq 3$ there exist infinitely many minimal $(p, q)$-Ramsey graphs \cite{BEL76}, \cite{FL06}. The simplest infinite sequence of minimal $(3, 3)$-Ramsey graphs are the graphs $K_3+C_{2r+1}, r \geq 1$. This sequence contains the already mentioned graphs $K_6$ and $K_3+C_5$. It was obtained by Nenov and Khadzhiivanov in \cite{NK79}. Later, this sequence was reobtained in \cite{BR80}, \cite{GSS95}, \cite{Sza77}.

Three 10-vertex minimal $(3, 3)$-Ramsey graphs are known. One of them is $K_3+C_7$ from the sequence $K_3+C_{2r+1}, r \geq 1$. The other two were obtained by Nenov in \cite{Nen80c} (the second graph is presented on Figure \ref{figure: Nenov_10} and the third is a subgraph of $K_1+\overline{C_9}$).

\vspace{1em}

The main goal of this work is to find new minimal $(3, 3)$-Ramsey graphs. To achieve this, we develop computer algorithms which are presented in Section 3. Using Algorithm \ref{algorithm: finding all minimal (3, 3)-Ramsey graphs on n vertices}, in Section 4 we find all minimal $(3, 3)$-Ramsey graphs with up to 12 vertices. In the next Section 5 we find all 13-vertex minimal $(3, 3)$-Ramsey graphs using Algorithm \ref{algorithm: finding n-vertex minimal (3, 3)-Ramsey graphs G for which alpha(G) geq n - k geq 1}. From the graphs found in Section 4 and Section 5 we obtain interesting corollaries, which are presented in Section 6. In Section 7 and Section 8, with the help of Algorithm \ref{algorithm: finding minimal (3, 3)-Ramsey graphs G for which alpha(G) geq abs(V(G)) - k geq 1} we obtain, accordingly, new upper bounds on the independence number and new lower bounds on the minimum degree of minimal $(3, 3)$-Ramsey graphs with an arbitrary number of vertices.

Similar computer aided research is made in \cite{JR95}, \cite{PRU99}, \cite{Col05}, \cite{CR06}, \cite{RX07}, \cite{XLS10}, \cite{LRX12} and \cite{SXP12}. We shall note that the algorithms from \cite{PRU99} were very useful to us.\\

This work is an extended version of my Master's thesis under the supervision of prof. Nedyalko Nenov. The most essential new element is Algorithm \ref{algorithm: finding minimal (3, 3)-Ramsey graphs G for which alpha(G) geq abs(V(G)) - k geq 1}, which is obtained jointly with prof. Nenov.

\section{Auxiliary results}

We will need the following results:

\begin{theorem}
\label{theorem: delta(G) geq (p-1)^2}
\cite{BEL76}\cite{FL06}
Let $G$ be a minimal $(p, p)$-Ramsey graph. Then, $\delta(G) \geq (p-1)^2$. In particular, when $p = 3$, we have $\delta(G) \geq 4$.
\end{theorem}

\begin{definition}
\label{definition: Sperner graph}
We say that $G$ is a Sperner graph if $N_G(u) \subseteq N_G(v)$ for some pair of vertices $u, v \in \V(G)$.
\end{definition}

\begin{proposition}
\label{proposition: minimal (p, q)-Ramsey graph is not Sperner}
If $G$ is a minimal $(p, q)$-Ramsey graph, then $G$ is not a Sperner graph.
\end{proposition}

\begin{proof}
Suppose the opposite is true, and let $u, v \in \V(G)$ be such that $N_G(u) \subseteq N_G(v)$. We color the edges of $G - u$ with two colors in such a way that there is no monochromatic $p$-clique of the first color and no monochromatic $q$-clique of the second color. After that, for each vertex $w \in N_G(u)$ we color the edge $[u, w]$ with the same color as the edge $[v, w]$. We obtain a 2-coloring of the edges of $G$ with no monochromatic $p$-cliques of the first color and no monochromatic $q$-cliques of the second color.
\end{proof}

\begin{theorem}
\label{theorem: F_e(3, 3; 5) = 15}
\cite{PRU99}
Let $G$ be a $(3, 3)$-Ramsey graph and $G \neq K_6$. If $|V(G)| \leq 14$, then $\omega(G) = 5$.
\end{theorem}

\begin{figure}[h]
	\centering
	\includegraphics[trim={0 470 0 0},clip,height=160px,width=160px]{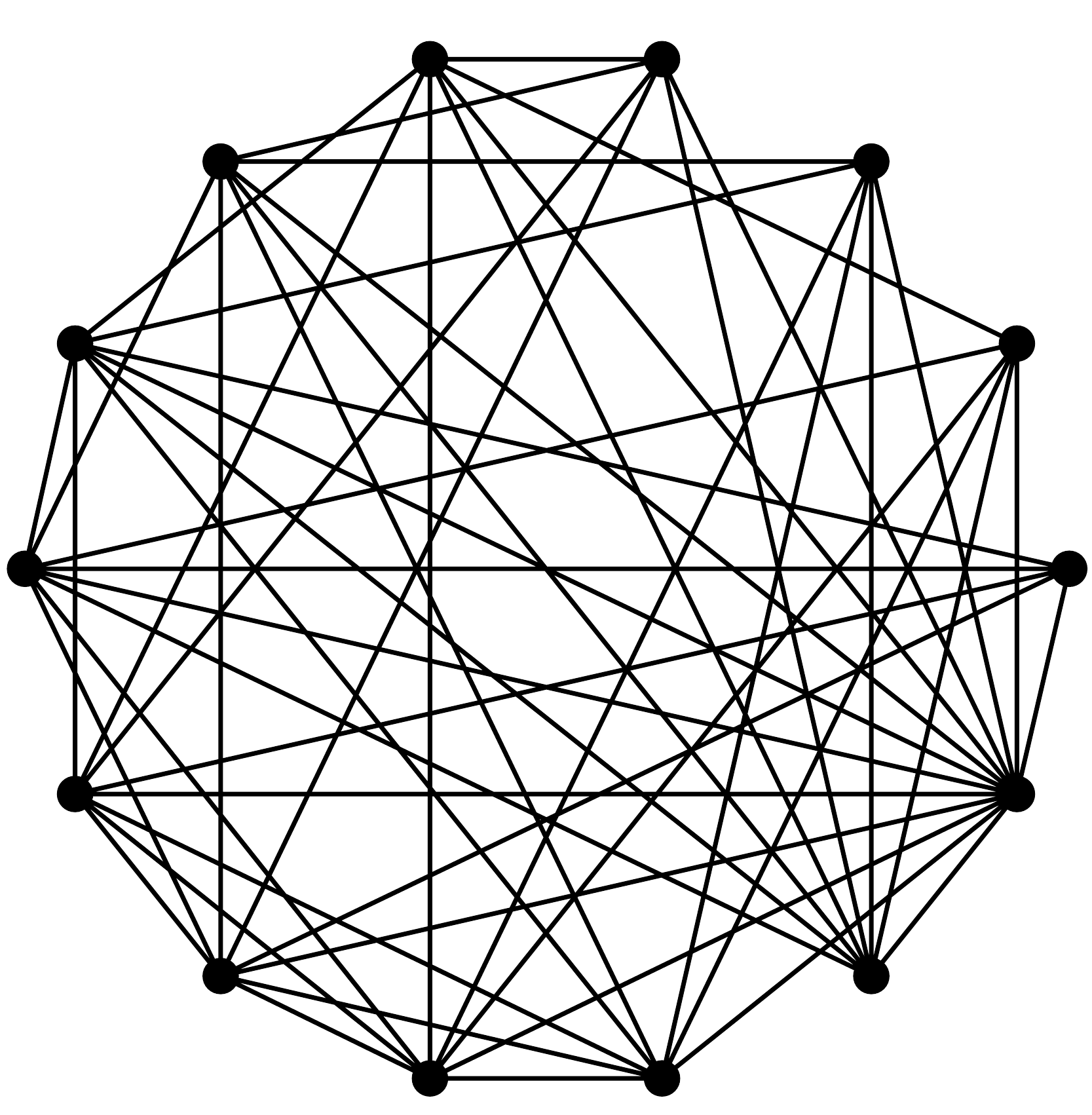}
	\caption{14-vertex minimal $(3, 3)$-Ramsey graph with a single 5 clique}
	\label{figure: 14_1xk5}
\end{figure}

According to Theorem \ref{theorem: F_e(3, 3; 5) = 15}, every $(3, 3)$-Ramsey graph $G$ with no more than 14 vertices contains a 5-clique. There exist 14-vertex $(3, 3)$-Ramsey graphs containing only a single 5-clique, an example of such a graph is presented on Figure \ref{figure: 14_1xk5}. The graph on Figure \ref{figure: 14_1xk5} is obtained with the help of the only 15-vertex bicritical $(3, 3)$-Ramsey graph with clique number 4 from \cite{PRU99}. First, by removing a vertex from the bicritical graph, we obtain 14-vertex graphs without 5 cliques. After that, by adding edges to the obtained graphs, we find a 14-vertex $(3, 3)$-Ramsey graph with a single 5-clique whose subgraph is the minimal $(3, 3)$-Ramsey graph on Figure \ref{figure: 14_1xk5}. Let us note that in \cite{PRU99} they obtain all 15-vertex $(3, 3)$-Ramsey graphs with clique number 4, and with the help of these graphs, one can find more examples of 14-vertex $(3, 3)$-Ramsey graphs.

\begin{theorem}
\label{theorem: chi(G) geq R(p, q)}
\cite{Lin72}
Let $G$ be a graph and $G \rightarrow (p, q)$. Then, $\chi(G) \geq R(p, q)$. In particular, if $G \rightarrow (3, 3)$, then $\chi(G) \geq 6$.
\end{theorem}

\begin{corollary}
\label{corollary: chi(H) geq 5}
Let $G \rightarrow (3, 3)$, let $v_1, ..., v_s$ be independent vertices of $G$ and $H = G - \set{v_1, ..., v_s}$. Then, $\chi(H) \geq 5$.
\end{corollary}

\begin{theorem}
\label{theorem: alpha(G(v)) leq d(v) - 3}
Let $G$ be a minimal $(3, 3)$-Ramsey graph. Then, for each vertex $v \in \V(G)$ we have $\alpha(G(v)) \leq d(v) - 3$.
\end{theorem}

\begin{proof}
Suppose the opposite is true, and let $A \subseteq N_G(v)$ be an independent set in $G(v)$ such that $|A| = d(v) - 2$. Let $a, b \in N_G(v) \setminus A$. Consider a 2-coloring of the edges of $G - v$ in which there are no monochromatic triangles. We color the edges $[v, a]$ and $[v, b]$ with the same color in such a way that there is no monochromatic triangle (if $a$ and $b$ are adjacent, we chose the color of $[v, a]$ and $[v, b]$ to be different from the color of $[a, b]$, and if $a$ and $b$ are not adjacent, then we chose an arbitrary color for $[v, a]$ and $[v, b]$). We color the remaining edges incident to $v$ with the other color, which is different from the color of $[v, a]$ and $[v, b]$. Since $N_G(v) \setminus \{a, b\} = A$ is and independent set, we obtain a 2-coloring of the edges of $G$ without monochromatic triangles, which is a contradiction.
\end{proof}

\begin{corollary}
\label{corollary: if d(v) = 4, then G(v) = K_4} 
Let $G$ be a minimal $(3, 3)$-Ramsey graph and $d(v) = 4$ for some vertex $v \in \V(G)$. Then, $G(v) = K_4$.
\end{corollary}

\section{Algorithms}

In this section, the computer algorithms used in this work are presented.

The first algorithm is appropriate for finding all minimal $(3, 3)$-Ramsey graphs with a small number of vertices. 

\begin{algorithm}
\label{algorithm: finding all minimal (3, 3)-Ramsey graphs on n vertices}
Finding all minimal $(3, 3)$-Ramsey graphs with $n$ vertices, where $n$ is fixed and $7 \leq n \leq 14$.

1. Generate all n-vertex non-isomorphic graphs with minimum degree at least 4, and denote the obtained set by $\mathcal{B}$.

2. Remove from $\mathcal{B}$ all Sperner graphs.

3. Remove from $\mathcal{B}$ all graphs with clique number not equal to 5.

4. Remove from $\mathcal{B}$ all graphs with chromatic number less than 6.

5. Remove from $\mathcal{B}$ all graphs which are not $(3, 3)$-Ramsey graphs.

6. Remove from $\mathcal{B}$ all graphs which are not minimal $(3, 3)$-Ramsey graphs.
\end{algorithm}

\begin{theorem}
\label{theorem: finding all minimal (3, 3)-Ramsey graphs on n vertices}
Fix $n \in \set{7, ..., 14}$. Then, after executing Algorithm \ref{algorithm: finding all minimal (3, 3)-Ramsey graphs on n vertices}, $\mathcal{B}$ consists of all $n$-vertex minimal $(3, 3)$-Ramsey graphs.
\end{theorem}

\begin{proof}
Step 6 guaranties that $\mathcal{B}$ contains only minimal $(3, 3)$-Ramsey graphs with $n$ vertices. Let $G$ be an arbitrary $n$-vertex minimal $(3, 3)$-Ramsey graph. We will prove that $G \in \mathcal{B}$. By Theorem \ref{theorem: delta(G) geq (p-1)^2}, $\delta(G) \geq 4$, and by Theorem \ref{proposition: minimal (p, q)-Ramsey graph is not Sperner}, $G$ is not a Sperner graph. Since $\abs{V(G)} \leq 14$, by Theorem \ref{theorem: F_e(3, 3; 5) = 15} we have $\omega(G) = 5$. By Theorem \ref{theorem: chi(G) geq R(p, q)}, $\chi(G) \geq 6$. Therefore, after step 4, $G \in \mathcal{B}$.
\end{proof}

In section 4 of this work we use Algorithm \ref{algorithm: finding all minimal (3, 3)-Ramsey graphs on n vertices} to obtain all $(3, 3)$-Ramsey graphs with up to 12 vertices. Algorithm \ref{algorithm: finding all minimal (3, 3)-Ramsey graphs on n vertices} is not appropriate in the cases $n \geq 13$, because the number of graphs generated in step 1 is too big. To find the 13-vertex minimal $(3, 3)$-Ramsey graphs, we will use Algorithm \ref{algorithm: finding n-vertex minimal (3, 3)-Ramsey graphs G for which alpha(G) geq n - k geq 1}, which is defined below.

In order to present the next algorithms we will need the following definitions and auxiliary propositions:

We say that a 2-coloring of the edges of a graph is $(3, 3)$-free if it has no monochromatic triangles.

\begin{definition}
\label{definition: marked vertex set}
Let $G$ be a graph and $M \subseteq \V(G)$. Let $G_1$ be a graph which is obtained by adding a new vertex $v$ to $G$ such that $N_{G_1}(v) = M$. We say that $M$ is a marked vertex set in $G$ if there exists a $(3, 3)$-free $2$-coloring of the edges of $G$ which cannot be extended to a $(3, 3)$-free $2$-coloring of the edges of $G_1$.
\end{definition}

It is clear that if $G \rightarrow (3, 3)$, then there are no marked vertex sets in $G$. The following proposition is true:

\begin{proposition}
\label{proposition: marked vertex sets in (3, 3)-Ramsey graphs}
Let $G$ be a minimal $(3, 3)$-Ramsey graph, let $v_1, ..., v_s$ be independent vertices of $G$ and $H = G - \set{v_1, ..., v_s}$. Then, $N_G(v_i), i = 1, ..., s$, are marked vertex sets in $H$.
\end{proposition}

\begin{proof}
Suppose the opposite is true, i.e. $N_G(v_i)$ is not a marked vertex set in $H$ for some $i \in \set{1, ..., s}$. Since $G$ is a minimal $(3, 3)$-Ramsey graph, there exists a $(3, 3)$-free $2$-coloring of the edges of $G - v_i$, which induces a $(3, 3)$-free $2$-coloring of the edges of $H$. By supposition, we can extend this $2$-coloring to a $(3, 3)$-free $2$-coloring of the edges of the graph $H_i = G - \set{v_1, ..., v_{i - 1}, v_{i + 1}, ..., v_s}$. Thus, we obtain a $(3, 3)$-free $2$-coloring of the edges of $G$, which is a contradiction.
\end{proof}

\begin{definition}
\label{definition: complete family of marked vertex sets}
Let $\set{M_1, ..., M_s}$ be a family of marked vertex sets in the graph $G$. Let $G_i$ be a graph which is obtained by adding a new vertex $v_i$ to $G$ such that $N_{G_i}(v_i) = M_i, i = 1, ..., s$. We say that $\set{M_1, ..., M_s}$ is a complete family of marked vertex sets in $G$, if for each $(3, 3)$-free $2$-coloring of the edges of $G$ there exists $i \in \set{1, ..., s}$ such that this $2$-coloring can not be extended to a $(3, 3)$-free $2$-coloring of the edges of $G_i$.
\end{definition}

\begin{proposition}
\label{proposition: if set(N_G(v_1), ..., N_G(v_s)) is a complete family of marked vertex sets in H, then G rightarrow (3, 3)}
Let $v_1, ..., v_s$ be independent vertices of the graph $G$ and $H = G - \set{v_1, ..., v_s}$. If $\set{N_G(v_1), ..., N_G(v_s)}$ is a complete family of marked vertex sets in $H$, then $G \rightarrow (3, 3)$.
\end{proposition}

\begin{proof}
Consider a 2-coloring of the edges of $G$ which induces a 2-coloring with no monochromatic triangles in $H$. According to Definition \ref{definition: complete family of marked vertex sets}, this 2-coloring of the edges of $H$ can not be extended in $G$ without forming a monochromatic triangle.
\end{proof}

It is easy to prove the following strengthening of Proposition \ref{proposition: marked vertex sets in (3, 3)-Ramsey graphs}:
\begin{proposition}
\label{proposition: if G is a minimal (3, 3)-Ramsey graph, then set(N_G(v_1), ..., N_G(v_s)) is a complete family of marked vertex sets}
Let $G$ be a minimal $(3, 3)$-Ramsey graph, let $v_1, ..., v_s$ be independent vertices of $G$ and $H = G - \set{v_1, ..., v_s}$. Then, $\set{N_G(v_1), ..., N_G(v_s)}$ is a complete family of marked vertex sets in $H$. What is more, this family is a minimal complete family, in the sense that it does not contain a proper complete subfamily.
\end{proposition}

Let $G$ be a minimal $(3, 3)$-Ramsey graph and $\alpha(G) \geq \abs{\V(G)} - k \geq 1$. Let $A$ be an independent set in $G$ such that $\abs{A} = \abs{\V(G)} - k$. Then, $\abs{\V(G - A)} = k$, and therefore the graph $G$ is obtained by adding and independent set of vertices to the $k$-vertex graph $G - A$. From Proposition \ref{proposition: minimal (p, q)-Ramsey graph is not Sperner} it is easy to see that for a fixed $k$ there are a finite number of minimal $(3, 3)$-Ramsey graphs $G$ for which $\alpha(G) \geq \abs{\V(G)} - k \geq 1$. We define an algorithm for finding all minimal $(3, 3)$-Ramsey graphs $G$ for which $\alpha(G) \geq \abs{\V(G)} - k \geq 1$, where $k$ is fixed (but $\V(G)$ is not fixed).

\begin{algorithm}
\label{algorithm: finding minimal (3, 3)-Ramsey graphs G for which alpha(G) geq abs(V(G)) - k geq 1}
(A. Bikov and N. Nenov)
Finding all minimal $(3, 3)$-Ramsey graphs $G$ for which $\omega(G) < q$ and $\alpha(G) \geq \abs{\V(G)} - k \geq 1$, where $q$ and $k$ are fixed positive integers.

1. Denote by $\mathcal{A}$ the set of all $k$-vertex graphs $H$ for which $\omega(H) < q$ and $\chi(H) \geq 5$. The obtained minimal $(3, 3)$-Ramsey graphs will be output in the set $\mathcal{B}$, let $\mathcal{B} = \emptyset$.

2. For each graph $H \in \mathcal{A}$:

2.1. Find all subsets $M$ of $\V(H)$ which have the properties:

(a) $K_{q - 1} \not\subseteq H[M]$, i.e. $M$ is a $K_{(q-1)}$-free subset.

(b) $M \not\subseteq N_H(v), \forall v \in \V(H)$.

(c) $M$ is a marked vertex set in $H$(see Definition \ref{definition: marked vertex set}).

Denote by $\mathcal{M}(H)$ the family of subsets of $\V(H)$ which have the properties (a), (b) and (c). Enumerate the elements of $\mathcal{M}(H)$: $\mathcal{M}(H) = \set{M_1, ..., M_t}$.
	
2.2. Find all minimal complete subfamilies of $\mathcal{M}(H)$ (see Definition \ref{definition: complete family of marked vertex sets}). For each such found subfamily $\set{M_{i_1}, ..., M_{i_s}}$ construct the graph $G = G(M_{i_1}, ..., M_{i_s})$ by adding new independent vertices $v_1, v_2, ..., v_s$ to $\V(H)$ such that $N_G(v_j) = M_{i_j}, j = 1, ..., s$. Add $G$ to $\mathcal{B}$. If there are no complete subfamilies of $\mathcal{M}(H)$, then no supergraphs of $H$ are added to $\mathcal{B}$.

3. Remove isomorph copies of graphs from $\mathcal{B}$.

4. Remove from $\mathcal{B}$ all non-minimal $(3, 3)$-Ramsey graphs.
\end{algorithm}

\begin{remark}
\label{remark: q in set(4, 5, 6)}
It is clear, that if $G$ is a minimal $(3, 3)$-Ramsey graph and $\omega(G) \geq 6$, then $G = K_6$. Obviously there are no $(3, 3)$-Ramsey graphs with clique number less than 3. Therefore, we shall use Algorithm \ref{algorithm: finding minimal (3, 3)-Ramsey graphs G for which alpha(G) geq abs(V(G)) - k geq 1} only for $q \in \set{4, 5, 6}$.
\end{remark}

\begin{theorem}
\label{theorem: finding minimal (3, 3)-Ramsey graphs G for which alpha(G) geq abs(V(G)) - k geq 1}
After executing Algorithm \ref{algorithm: finding minimal (3, 3)-Ramsey graphs G for which alpha(G) geq abs(V(G)) - k geq 1}, the set $\mathcal{B}$ coincides with the set of all minimal $(3, 3)$-Ramsey graphs $G$ for which $\omega(G) < q$ and $\alpha(G) \geq \abs{\V(G)} - k \geq 1$.
\end{theorem}

\begin{proof}
From step 2.2 it becomes clear that every graph $G$ which is added to $\mathcal{B}$ is obtained by adding independent vertices $v_1, ..., v_s$ to a graph $H \in \mathcal{A}$. Therefore, $\alpha(G) \geq s = \abs{\V(G)} - \abs{\V(H)} = \abs{\V(G)} - k$. From $\omega(H) < q$ and $K_{q - 1} \not\subseteq H[N_G(v_i)], i = 1, ..., s$, it follows $\omega(G) < q$. According to Proposition \ref{proposition: if set(N_G(v_1), ..., N_G(v_s)) is a complete family of marked vertex sets in H, then G rightarrow (3, 3)}, after step 2.2 $\mathcal{B}$ contains only $(3, 3)$-Ramsey graphs, and after step 4 $\mathcal{B}$ contains only minimal $(3, 3)$-Ramsey graphs.

In order to prove that $\mathcal{B}$ contains all minimal $(3, 3)$-Ramsey graphs which fulfill the conditions, consider an arbitrary minimal $(3, 3)$-Ramsey graph $G$ for which $\omega(G) < q$ and $\alpha(G) \geq \abs{\V(G)} - k \geq 1$. We will prove that $G \in \mathcal{B}$.
	
Denote $s = \abs{\V(G)} - k \geq 1$. Let $v_1, ..., v_s$ be independent vertices of $G$ and $H = G - \set{v_1, ..., v_s}$. By \ref{corollary: chi(H) geq 5}, $\chi(H) \geq 5$. Therefore, after executing step 1, $H \in \mathcal{A}$.

From $\omega(G) < q$ it follows $\omega(G(v_i)) < q - 1$. By Proposition \ref{proposition: minimal (p, q)-Ramsey graph is not Sperner}, $G$ is not a Sperner graph, and therefore $N_G(v_i) \not\subseteq N_H(v), \forall v \in \V(H)$. According to Proposition \ref{proposition: marked vertex sets in (3, 3)-Ramsey graphs}, $N_G(v_i)$ are marked vertex sets in $H$. Therefore, after executing step 2.1, $N_G(v_i) \in \mathcal{M}(H), i = 1, ..., s$.

From Proposition \ref{proposition: if G is a minimal (3, 3)-Ramsey graph, then set(N_G(v_1), ..., N_G(v_s)) is a complete family of marked vertex sets} it becomes clear that $\set{N_G(v_1), ..., N_G(v_s)}$ is a minimal complete subfamily of $\mathcal{M}(H)$. Therefore, in step 2.2 the graph $G$ is added to $\mathcal{B}$.

Thus, the theorem is proved.
\end{proof}

In order to find the 13-vertex minimal $(3, 3)$-Ramsey graphs we will use the following modification of Algorithm \ref{algorithm: finding minimal (3, 3)-Ramsey graphs G for which alpha(G) geq abs(V(G)) - k geq 1} in which $n = \abs{\V(G)}$ is fixed:

\begin{algorithm}
\label{algorithm: finding n-vertex minimal (3, 3)-Ramsey graphs G for which alpha(G) geq n - k geq 1}
Modification of Algorithm \ref{algorithm: finding minimal (3, 3)-Ramsey graphs G for which alpha(G) geq abs(V(G)) - k geq 1} for finding all $n$-vertex minimal $(3, 3)$-Ramsey graphs $G$ for which $\omega(G) < q$ and $\alpha(G) \geq n - k \geq 1$, where $q$, $k$ and $n$ are fixed positive integers.

In step 2.2 of Algorithm \ref{algorithm: finding minimal (3, 3)-Ramsey graphs G for which alpha(G) geq abs(V(G)) - k geq 1} add the condition to consider only minimal complete subfamilies $\set{M_{i_1}, ..., M_{i_s}}$ of $\mathcal{M}(H)$ in which $s = n - k$.
\end{algorithm}

\section{Minimal $(3, 3)$-Ramsey graphs with up to 12 vertices}

We execute Algorithm \ref{algorithm: finding all minimal (3, 3)-Ramsey graphs on n vertices} for $n = 7, 8, 9, 10, 11, 12$, and we find all minimal $(3, 3)$-Ramsey graphs with up to 12 vertices except $K_6$. In this way, we obtain the known results: there is no minimal $(3, 3)$-Ramsey graph with 7 vertices, the Graham graph $K_3+C_5$ is the only such 8-vertex graph, and there exists only one such 9-vertex graph, the Nenov graph from \cite{Nen79} (see Figure \ref{figure: Nenov_9}). We also obtain the following new results:

\begin{theorem}
\label{theorem: 10-vertex minimal (3, 3)-Ramsey graphs}
There are exactly 6 minimal 10-vertex $(3, 3)$-Ramsey graphs. These graphs are given on Figure \ref{figure: 10}, and some of their properties are listed in Table \ref{table: 10-vertex graphs properties}.
\end{theorem}

\begin{theorem}
\label{theorem: 11-vertex minimal (3, 3)-Ramsey graphs}
There are exactly 73 minimal 11-vertex $(3, 3)$-Ramsey graphs. Some of their properties are listed in Table \ref{table: 11-vertex graphs properties}. Examples of 11-vertex minimal $(3, 3)$-Ramsey graphs are given on Figure \ref{figure: 11_a4} and Figure \ref{figure: 11_a2}.
\end{theorem}

\begin{theorem}
\label{theorem: 12-vertex minimal (3, 3)-Ramsey graphs}
There are exactly 3041 minimal 12-vertex $(3, 3)$-Ramsey graphs. Some of their properties are listed in Table \ref{table: 12-vertex graphs properties}. Examples of 12-vertex minimal $(3, 3)$-Ramsey graphs are given on Figure \ref{figure: 12_a5} and Figure \ref{figure: 12_aut}.
\end{theorem}

We will use the following enumeration for the obtained minimal $(3, 3)$-Ramsey graphs:\\
- $G_{10.1}$, ..., $G_{10.6}$ are the 10-vertex graphs;\\
- $G_{11.1}$, ..., $G_{11.73}$ are the 11-vertex graphs;\\
- $G_{12.1}$, ..., $G_{12.3041}$ are the 12-vertex graphs;

The indexes correspond to the order of the graphs' canonical labels defined in \emph{nauty} \cite{MP13}.

Detailed data for the number of graphs obtained at each step of the execution of Algorithm \ref{algorithm: finding all minimal (3, 3)-Ramsey graphs on n vertices} is given in Table \ref{table: steps in finding all minimal (3, 3)-Ramsey graphs with up to 12 vertices}.

\begin{table}[h]
	\centering
	\resizebox{\textwidth}{!}{
		\begin{tabular}{ | l | r | r | r | r | r | }
			\hline
			Step of		& $n = 8$		& $n = 9$		& $n = 10$		& $n = 11$		& $n = 12$		\\
			Algorithm \ref{algorithm: finding all minimal (3, 3)-Ramsey graphs on n vertices}&&&&&\\
			\hline
			1				&  424			& 15 471		& 1 249 973		& 187 095 840	& 48 211 096 031\\
			2				&  59			& 2 365			& 206 288		& 33 128 053	& 9 148 907 379	\\
			3				&  9			& 380			& 41 296		& 8 093 890		& 2 763 460 021	\\
			4				&  1			& 7				& 356			& 78 738		& 44 904 195	\\
			5				&  1			& 3				& 126			& 23 429		& 11 670 079	\\
			6				&  1			& 1				& 6				& 73			& 3041			\\
			\hline
		\end{tabular}
	}
	\caption{Steps in finding all minimal $(3, 3)$-Ramsey graphs with up to 12 vertices}
	\label{table: steps in finding all minimal (3, 3)-Ramsey graphs with up to 12 vertices}
\end{table}

\begin{table}
	\small
	\centering
	\resizebox{\textwidth}{!}{
		\begin{tabular}{ | p{2.0cm} | p{2.0cm} | p{2.0cm} | p{2.0cm} | p{2.0cm} | p{2.0cm} | p{2.0cm} | }
			\hline
			$|\E(G)|$	\hfill $\#$		& $\delta(G)$ \hfill $\#$	& $\Delta(G)$ \hfill $\#$	& $\alpha(G)$ \hfill $\#$	& $\chi(G)$ \hfill $\#$		& $|Aut(G)|$ \hfill $\#$\\
			\hline
			30	\hfill 1				& 4	\hfill 1				& 9	\hfill 6				& 2	\hfill 3				& 6	\hfill 6				& 4		\hfill 2		\\
			31	\hfill 1				& 5	\hfill 4				& 	\hfill 					& 3	\hfill 3				& 	\hfill 					& 8		\hfill 2		\\
			32	\hfill 2				& 6	\hfill 1				& 	\hfill 					& 	\hfill 					& 	\hfill 					& 16	\hfill 1		\\
			33	\hfill 1				& 	\hfill 					& 	\hfill 					& 	\hfill 					& 	\hfill 					& 84	\hfill 1		\\
			34	\hfill 1				& 	\hfill 					& 	\hfill 					& 	\hfill 					& 	\hfill 					& 		\hfill			\\
			\hline
		\end{tabular}
	}
	\caption{Some properties of the 10-vertex minimal $(3, 3)$-Ramsey graphs}
	\label{table: 10-vertex graphs properties}
	\vspace{1em}
	\resizebox{\textwidth}{!}{
		\begin{tabular}{ | p{2.0cm} | p{2.0cm} | p{2.0cm} | p{2.0cm} | p{2.0cm} | p{2.0cm} | p{2.0cm} | }
			\hline
			$|\E(G)|$	\hfill $\#$		& $\delta(G)$	\hfill $\#$	& $\Delta(G)$	\hfill $\#$	& $\alpha(G)$	\hfill $\#$	& $\chi(G)$	\hfill $\#$		& $|Aut(G)|$ \hfill $\#$\\
			\hline
			35	\hfill 6				& 4	\hfill 5				& 8	\hfill 1				& 2	\hfill 4				& 6	\hfill 73				& 1		\hfill 20		\\
			36	\hfill 13				& 5	\hfill 58				& 10\hfill 72				& 3	\hfill 66				& 	\hfill 					& 2		\hfill 29		\\
			37	\hfill 23				& 6	\hfill 10				& 	\hfill 					& 4	\hfill 3				& 	\hfill 					& 4		\hfill 14		\\
			38	\hfill 25				& 	\hfill 					& 	\hfill 					& 	\hfill 					& 	\hfill 					& 6		\hfill 1		\\
			39	\hfill 5				& 	\hfill 					& 	\hfill 					& 	\hfill 					& 	\hfill 					& 8		\hfill 4		\\
			41	\hfill 1				& 	\hfill 					& 	\hfill 					& 	\hfill 					& 	\hfill 					& 12	\hfill 1		\\
			\hfill 					& 	\hfill 					& 	\hfill 					& 	\hfill 					& 	\hfill 					& 16	\hfill 3		\\
			\hfill 					& 	\hfill 					& 	\hfill 					& 	\hfill 					& 	\hfill 					& 24	\hfill 1		\\
			\hline
		\end{tabular}
	}
	\caption{Some properties of the 11-vertex minimal $(3, 3)$-Ramsey graphs}
	\label{table: 11-vertex graphs properties}
	\vspace{1em}
	\resizebox{\textwidth}{!}{
		\begin{tabular}{ | p{2.0cm} | p{2.0cm} | p{2.0cm} | p{2.0cm} | p{2.0cm} | p{2.0cm} | p{2.0cm} | }
			\hline
			$|\E(G)|$	\hfill $\#$		& $\delta(G)$	\hfill $\#$	& $\Delta(G)$	\hfill $\#$	& $\alpha(G)$	\hfill $\#$	& $\chi(G)$	\hfill $\#$		& $|Aut(G)|$ \hfill $\#$\\
			\hline
			38	\hfill 5				& 4	\hfill 129				& 8	\hfill 43				& 2	\hfill 124				& 6	\hfill 3 041			& 1		\hfill 1 792	\\
			39	\hfill 27				& 5	\hfill 2 178			& 9	\hfill 1 196			& 3	\hfill 2 431			& 	\hfill 					& 2		\hfill 851		\\
			40	\hfill 144				& 6	\hfill 611				& 11\hfill 1 802			& 4	\hfill 485				& 	\hfill 					& 4		\hfill 286		\\
			41	\hfill 418				& 7	\hfill 123				& 	\hfill 					& 5	\hfill 1				& 	\hfill 					& 6		\hfill 1		\\
			42	\hfill 1 014			& 	\hfill 					& 	\hfill 					& 	\hfill 					& 	\hfill 					& 8		\hfill 67		\\
			43	\hfill 459				& 	\hfill 					& 	\hfill 					& 	\hfill 					& 	\hfill 					& 12	\hfill 16		\\
			44	\hfill 224				& 	\hfill 					& 	\hfill 					& 	\hfill 					& 	\hfill 					& 16	\hfill 18		\\
			45	\hfill 351				& 	\hfill 					& 	\hfill 					& 	\hfill 					& 	\hfill 					& 24	\hfill 6		\\
			46	\hfill 299				& 	\hfill 					& 	\hfill 					& 	\hfill 					& 	\hfill 					& 32	\hfill 1		\\
			47	\hfill 84				& 	\hfill 					& 	\hfill 					& 	\hfill 					& 	\hfill 					& 36	\hfill 1		\\
			48	\hfill 16				& 	\hfill 					& 	\hfill 					& 	\hfill 					& 	\hfill 					& 96	\hfill 1		\\
			\hfill 					& 	\hfill 					& 	\hfill 					& 	\hfill 					& 	\hfill 					& 108	\hfill 1		\\
			\hline
		\end{tabular}
	}
	\caption{Some properties of the 12-vertex minimal $(3, 3)$-Ramsey graphs}
	\label{table: 12-vertex graphs properties}
	\vspace{1em}
	\resizebox{\textwidth}{!}{
		\begin{tabular}{ | p{2.0cm} | p{2.0cm} | p{2.0cm} | p{2.0cm} | p{2.0cm} | p{2.0cm} | p{2.0cm} | }
			\hline
			$|\E(G)|$	\hfill $\#$		& $\delta(G)$	\hfill $\#$	& $\Delta(G)$	\hfill $\#$	& $\alpha(G)$	\hfill $\#$	& $\chi(G)$	\hfill $\#$		& $|Aut(G)|$ \hfill $\#$\\
			\hline
			41	\hfill 4				& 4	\hfill 13 725			& 8	\hfill 16				& 2	\hfill 13				& 6	\hfill 306 622			& 1		\hfill 251 976	\\
			42	\hfill 44				& 5	\hfill 191 504			& 9	\hfill 61 678			& 3	\hfill 218 802			& 7	\hfill 13				& 2		\hfill 46 487	\\
			43	\hfill 220				& 6	\hfill 85 932			& 10\hfill 175 108			& 4	\hfill 86 721			& 	\hfill 					& 3		\hfill 10		\\
			44	\hfill 1 475			& 7	\hfill 15 391			& 12\hfill 69 833			& 5	\hfill 1 097			& 	\hfill 					& 4		\hfill 6 851	\\
			45	\hfill 7 838			& 8	\hfill 83				& 	\hfill 					& 6	\hfill 2				& 	\hfill 					& 6		\hfill 83		\\
			46	\hfill 28 805			& 	\hfill 					& 	\hfill 					& 	\hfill 					& 	\hfill 					& 8		\hfill 916		\\
			47	\hfill 33 810			& 	\hfill 					& 	\hfill 					& 	\hfill 					& 	\hfill 					& 12	\hfill 129		\\
			48	\hfill 26 262			& 	\hfill 					& 	\hfill 					& 	\hfill 					& 	\hfill 					& 16	\hfill 106		\\
			49	\hfill 39 718			& 	\hfill 					& 	\hfill 					& 	\hfill 					& 	\hfill 					& 24	\hfill 44		\\
			50	\hfill 62 390			& 	\hfill 					& 	\hfill 					& 	\hfill 					& 	\hfill 					& 32	\hfill 12		\\
			51	\hfill 59 291			& 	\hfill 					& 	\hfill 					& 	\hfill 					& 	\hfill 					& 36	\hfill 3		\\
			52	\hfill 34 132			& 	\hfill 					& 	\hfill 					& 	\hfill 					& 	\hfill 					& 40	\hfill 1		\\
			53	\hfill 10 878			& 	\hfill 					& 	\hfill 					& 	\hfill 					& 	\hfill 					& 48	\hfill 11		\\
			54	\hfill 1 680			& 	\hfill 					& 	\hfill 					& 	\hfill 					& 	\hfill 					& 72	\hfill 3		\\
			55	\hfill 86				& 	\hfill 					& 	\hfill 					& 	\hfill 					& 	\hfill 					& 96	\hfill 2		\\
			56	\hfill 2				& 	\hfill 					& 	\hfill 					&	\hfill 					& 	\hfill 					& 144	\hfill 1		\\
			\hline
		\end{tabular}
	}
	\caption{Some properties of the 13-vertex minimal $(3, 3)$-Ramsey graphs}
	\label{table: 13-vertex graphs properties}
\end{table}

\section{Minimal $(3, 3)$-Ramsey graphs with 13 vertices}

The method with which we find all 13-vertex minimal $(3, 3)$-Ramsey graphs consists of two parts:

1. First, we find the 13-vertex minimal $(3, 3)$-Ramsey graphs with independence number 2. We use $R(3, 6) = 18$ \cite{Rad14}, and that all graphs $G$ for which $\alpha(G) < 3$ and $\omega(G) < 6$ are known \cite{McK_r}. Among them, the 13-vertex graphs are 275 086. By computer check, we find that exactly 13 of these graphs are minimal $(3, 3)$-Ramsey graphs.

2. It remains to find the 13-vertex minimal $(3, 3)$-Ramsey graphs with independence number at least 3. To do this, we execute Algorithm \ref{algorithm: finding n-vertex minimal (3, 3)-Ramsey graphs G for which alpha(G) geq n - k geq 1}($n = 13; k = 10; q = 6$). First, in step 1 of Algorithm \ref{algorithm: finding n-vertex minimal (3, 3)-Ramsey graphs G for which alpha(G) geq n - k geq 1} we find all 1 923 103 graphs $H$ with 10 vertices for which $\omega(H) \leq 5$ and $\chi(H) \geq 5$. After that, in step 2 of Algorithm \ref{algorithm: finding n-vertex minimal (3, 3)-Ramsey graphs G for which alpha(G) geq n - k geq 1} we add 3 independent vertices to the obtained 10-vertex graphs, and thus, we obtain all 306 622 minimal $(3, 3)$-Ramsey graphs with 13-vertices and independence number at least 3.

Finally, we obtain
\begin{theorem}
\label{theorem: 13-vertex minimal (3, 3)-Ramsey graphs}
There are exactly 306 635 minimal 13-vertex $(3, 3)$-Ramsey graphs. Some of their properties are listed in \ref{table: 13-vertex graphs properties}. Examples of 13-vertex minimal $(3, 3)$-Ramsey graphs are given on Figure \ref{figure: 13_regular}, Figure \ref{figure: 13_aut} and Figure \ref{figure: 13_a2}.
\end{theorem}

We enumerate the obtained 13-vertex $(3, 3)$-Ramsey graphs: $G_{13.1}$, ..., $G_{13.306635}$.\\

As noted, all graphs $G$ for which $\alpha(G) < 3$ and $\omega(G) < 6$ are known and from $R(3, 6) = 18$ it follows that these graphs have at most 17 vertices. By computer check we find that there are no minimal $(3, 3)$-Ramsey graphs with independence number 2 and more than 13 vertices. Thus, we prove

\begin{theorem}
\label{theorem: minimal (3, 3)-Ramsey graphs G for which alpha(G) = 2}
Let $G$ be a minimal $(3, 3)$-Ramsey graph and $\alpha(G) = 2$. Then, $\abs{\V(G)} \leq 13$. There are exactly 145 minimal $(3, 3)$-Ramsey graphs for which $\alpha(G) = 2$:

- 8-vertex: 1 ($K_3+C_5$);

- 9-vertex: 1 (see Figure \ref{figure: Nenov_9});

- 10-vertex: 3 ($G_{10.3}$, $G_{10.5}$, $G_{10.6}$, see Figure \ref{figure: 10});

- 11-vertex: 4 ($G_{11.46}$, $G_{11.47}$, $G_{11.54}$, $G_{11.69}$, see Figure \ref{figure: 11_a2});

- 12-vertex: 124;

- 13-vertex: 13 (see Figure \ref{figure: 13_a2});
\end{theorem}

By executing Algorithm \ref{algorithm: finding n-vertex minimal (3, 3)-Ramsey graphs G for which alpha(G) geq n - k geq 1}($n = 10, 11, 12; k = 7, 8, 9; q = 6$), we find all minimal $(3, 3)$-Ramsey graphs with 10, 11 and 12 vertices and independence number greater than 2. In this way, with the help of Theorem \ref{theorem: minimal (3, 3)-Ramsey graphs G for which alpha(G) = 2}, we obtain a new proof of Theorem \ref{theorem: 10-vertex minimal (3, 3)-Ramsey graphs}, Theorem \ref{theorem: 11-vertex minimal (3, 3)-Ramsey graphs} and Theorem \ref{theorem: 12-vertex minimal (3, 3)-Ramsey graphs}.

\section{Corollaries from the obtained results}

\subsection{Minimum and maximum degree}

\begin{figure}[h]
	\centering
	\includegraphics[trim={0 470 0 0},clip,height=160px,width=160px]{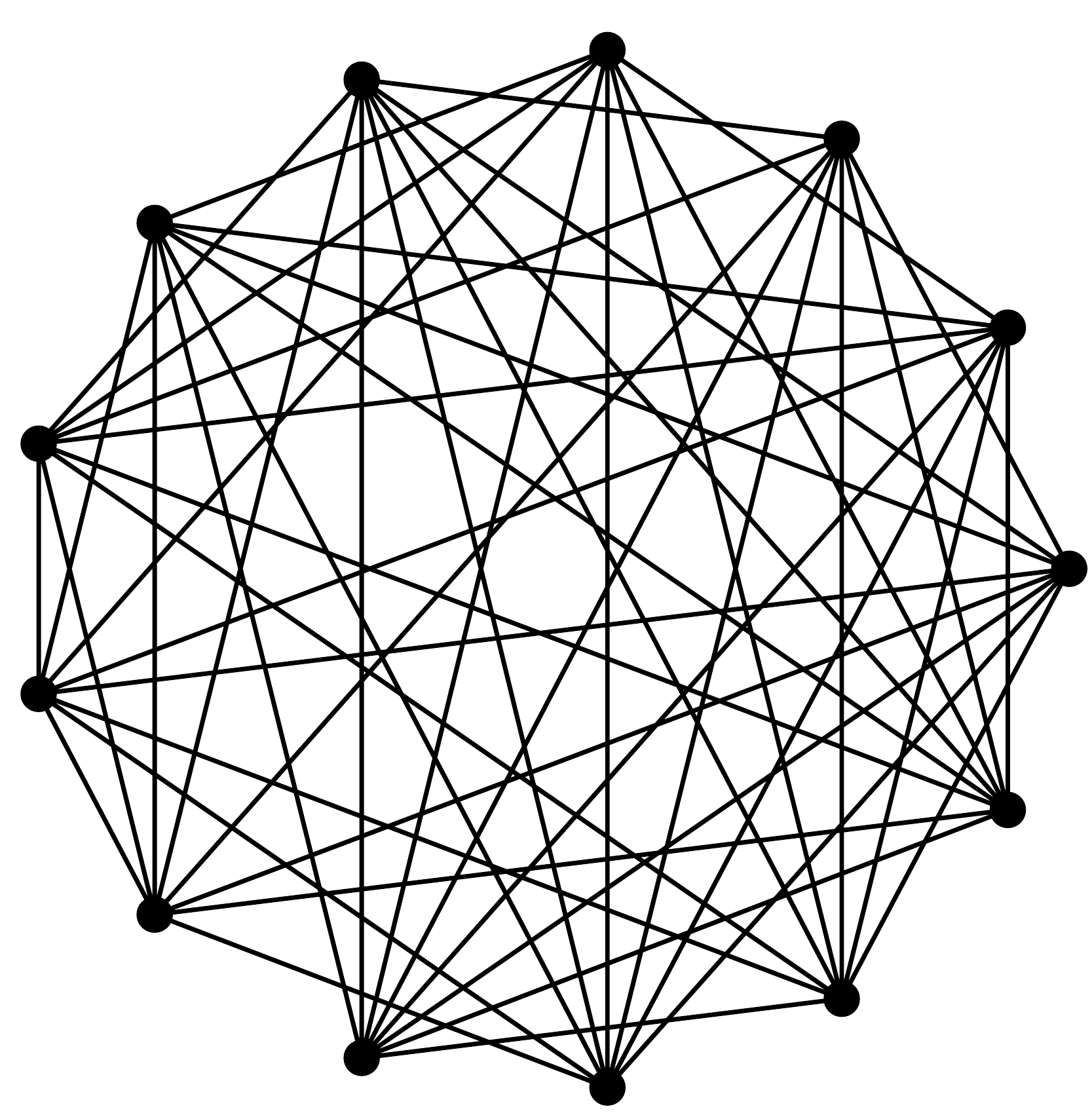}
	\caption{8-regular 13-vertex minimal $(3, 3)$-Ramsey graph}
	\label{figure: 13_regular}
\end{figure}

By Theorem \ref{theorem: delta(G) geq (p-1)^2}, if $G$ is a minimal $(3, 3)$-Ramsey graph, then $\delta(G) \geq 4$. Via very elegant constructions, in \cite{BEL76} and \cite{FL06} it is proved that the bound $\delta(G) \geq (p - 1)^2$ from Theorem \ref{theorem: delta(G) geq (p-1)^2} is exact. However, these constructions are not very economical in the case $p = 3$. For example, the minimal $(3, 3)$-Ramsey graph $G$ from \cite{FL06} with $\delta(G) = 4$ is not presented explicitly, but it is proved that it is a subgraph of a graph with 17577 vertices. From the next theorem we see that the smallest minimal $(3, 3)$-Ramsey graph $G$ with $\delta(G) = 4$ has 10 vertices:

\begin{theorem}
\label{theorem: 10-vertex minimal (3, 3)-Ramsey graph G for which delta(G) = 4}
Let $G$ be a minimal $(3, 3)$-Ramsey graph and $\delta(G) = 4$. Then, $\abs{\V(G)} \geq 10$. There is only one 10-vertex minimal $(3, 3)$-Ramsey graph $G$ with $\delta(G) = 4$, namely $G_{10.2}$ (see Figure \ref{figure: 10}). What is more, $G$ has only a single vertex of degree 4. For all other 10-vertex minimal $(3, 3)$-Ramsey graphs $G$, $\delta(G) = 5$.
\end{theorem}

Let $G$ be a $(3, 3)$-Ramsey graph. By Theorem \ref{theorem: chi(G) geq R(p, q)}, $\chi(G) \geq 6$ and from the inequality $\chi(G) \leq \Delta(G) + 1$ (see \cite{Har69}) we obtain $\Delta(G) \geq 5$. From the Brooks' Theorem (see \cite{Har69}) it follows that if $G \neq K_6$, then $\Delta(G) \geq 6$. The following related question arises naturally:
\begin{align*}
\label{question: 6-regular minimal (3, 3)-Ramsey graph}
&\emph{Are there minimal $(3, 3)$-Ramsey graphs which are 6-regular?}\\
&\emph{(i.e. $d(v) = 6, \forall v \in \V(G)$)}
\end{align*}
From the obtained minimal $(3, 3)$-Ramsey graphs we see that the following theorem is true:
\begin{theorem}
\label{theorem: regular minimal (3, 3)-Ramsey graphs}
Let $G$ be a regular minimal $(3, 3)$-Ramsey graph and $G \neq K_6$. Then, $\abs{\V(G)} \geq 13$. There is only one regular minimal $(3, 3)$-Ramsey with 13 vertices, and this is the graph presented on Figure \ref{figure: 13_regular}, which is 8-regular.
\end{theorem}	

In regard to the maximum degree of the minimal $(3, 3)$-Ramsey graphs we obtain the following result:
\begin{theorem}
\label{theorem: maximum degree of minimal (3, 3)-Ramsey graphs}
Let $G$ be a minimal $(3, 3)$-Ramsey graph. Then:

(a) $\Delta(G) = \abs{\V(G)} - 1$, if $\abs{\V(G)} \leq 10$.

(b) $\Delta(G) \geq 8$, if $\abs{\V(G)} = 11$, $12$ or $13$.
\end{theorem}

\subsection{Chromatic number}

By Theorem \ref{theorem: chi(G) geq R(p, q)}, if $G$ is a $(3, 3)$-Ramsey graph, then $\chi(G) \geq 6$.

From the obtained minimal $(3, 3)$-Ramsey graphs we derive the following results:

\begin{theorem}
\label{theorem: if abs(V(G)) leq 12, then chi(G) leq 7}
Let $G$ be a minimal $(3, 3)$-Ramsey graph and $\abs{\V(G)} \leq 12$. Then $\chi(G) = 6$.
\end{theorem}

\begin{theorem}
\label{theorem: 7-chromatic (3, 3)-Ramsey graphs}
Let G be a minimal $(3,3)$-Ramsey graph and $|V(G)| \leq 14$. Then $\chi(G) \leq 7$. The smallest 7-chromatic minimal $(3,3)$-Ramsey graphs are the 13 minimal $(3,3)$-Ramsey graph with 13 vertices and independence number 2, given on Figure \ref{figure: 13_a2}.
\end{theorem}

\begin{proof}
Suppose the opposite is true, i.e. $\chi(G) \geq 8$. Then, according to \cite{Nen10}, $G = K_1 + Q$, where $\overline{Q}$ is the graph presented on Figure \ref{figure: Q_13}. The graph $K_1 + Q$ is a $(3, 3)$-Ramsey graph, but it is not minimal. By Theorem \ref{theorem: if abs(V(G)) leq 12, then chi(G) leq 7}, there are no 7-chromatic minimal $(3, 3)$-Ramsey graphs with less than 13 vertices. The graphs on Figure \ref{figure: 13_a2} are 13-vertex minimal $(3, 3)$-Ramsey graphs with independence number 2, and therefore these graphs are 7-chromatic. By computer check, we find that among the 13-vertex $(3, 3)$-Ramsey graphs with independence number greater than 2 there are no 7-chromatic graphs.
\end{proof}

\begin{figure}[h]
	\centering
	\includegraphics[height=160px,width=160px]{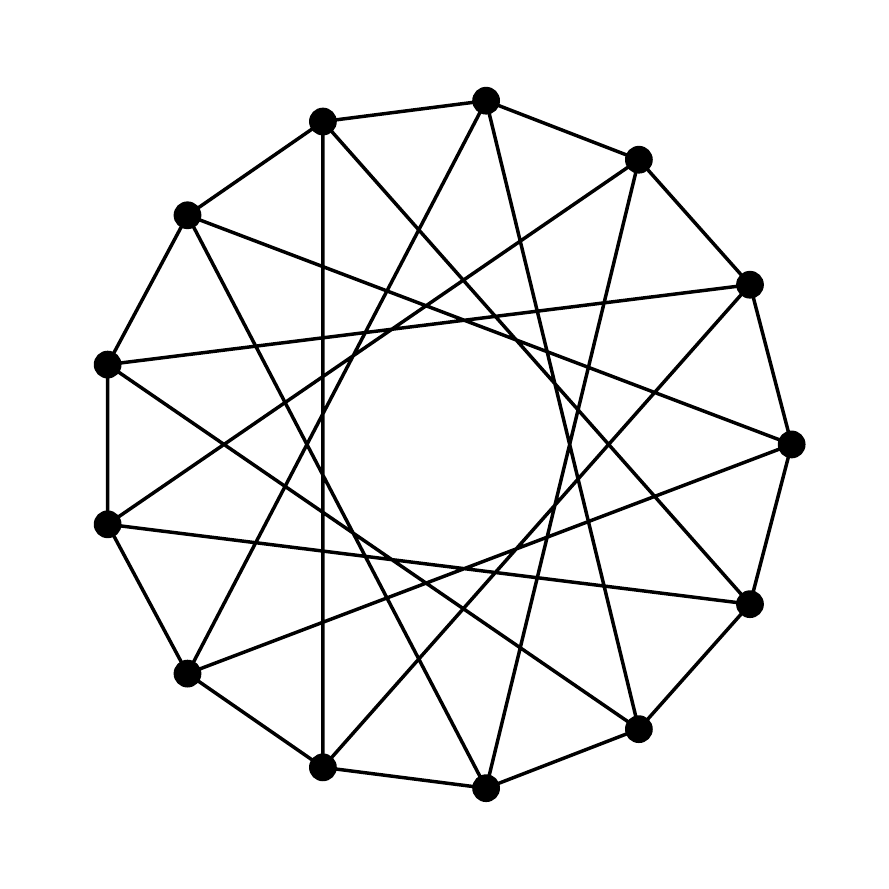}
	\caption{Graph $\overline{Q}$}
	\label{figure: Q_13}
\end{figure}

\subsection{Multiplicities}

\begin{definition}
\label{definition: K_3-multiplicity}
Denote by $M(G)$ the minimum number of monochromatic triangles in all 2-colorings of $\E(G)$. The number $M(G)$ is called a $K_3$-multiplicity of the graph $G$.
\end{definition}

In \cite{Goo59} the $K_3$-multiplicities of all complete graphs are computed, i.e. $M(K_n)$ is computed for all positive integers $n$. Similarly, the $K_p$-multiplicity of a graph is defined \cite{HP74}. The following works are dedicated to the computation of the multiplicities of some concrete graphs: \cite{Jac80}, \cite{Jac82}, \cite{RS77}, \cite{BR90}, \cite{PR01}.

With the help of a computer, we check the $K_3$-multiplicities of the obtained minimal $(3, 3)$-Ramsey graphs and we derive the following results:

\begin{theorem}
\label{theorem: if \abs(V(G)) leq 13 and G neq K_6, then M(G) = 1}
If $G$ is a minimal $(3, 3)$-Ramsey graph, $\abs{\V(G)} \leq 13$ and $G \neq K_6$, then $M(G) = 1$.
\end{theorem}

We suppose the following hypothesis is true:
\begin{hypothesis}
\label{hypothesis: if G is a minimal (3, 3)-Ramsey graph and G neq K_6, then M(G) = 1}
If $G$ is a minimal $(3, 3)$-Ramsey graph and $G \neq K_6$, then $M(G) = 1$.
\end{hypothesis}

In support to this hypothesis we prove the following:

\begin{proposition}
\label{proposition: if G is a minimal (3, 3)-Ramsey graph, G neq K_6 and delta(G) leq 5, then M(G) = 1}
If $G$ is a minimal $(3, 3)$-Ramsey graph, $G \neq K_6$ and $\delta(G) \leq 5$, then $M(G) = 1$.
\end{proposition}

\begin{proof}
Let $v \in \V(G)$ and $d(v) \leq 5$. Consider a 2-coloring of $\E(G-v)$ without monochromatic triangles. We will color the edges incident to $v$ with two colors in such a way that we will obtain a 2-coloring of $\E(G)$ with exactly one monochromatic triangle. To achieve this, we consider the following two cases:

Case 1: $d(v) = 4$. By Corollary \ref{corollary: if d(v) = 4, then G(v) = K_4}, $G(v) = K_4$. Let $N_v = \{a, b, c, d\}$ and suppose that $[a, b]$ is colored with the first color. Then, $[c, d]$ is also colored with the first color (otherwise, by coloring $[v, a]$ and $[v, b]$ with the second color and $[v, c]$ and $[v, d]$ with the fist color, we obtain a 2-coloring of $\E(G)$ without monochromatic triangles). Thus, $[a, b]$ and $[c, d]$ are colored in the first color. We color $[v, a]$ and $[v, b]$ with the first color and $[v, c]$ and $[v, d]$ with the second color. We obtain a 2-coloring of $\E(G)$ with exactly one monochromatic triangle $[v, a, b]$.

Case 2: $d(v) = 5$. Since $\omega(G) \leq 5$, in $N_G(v)$ there are two non-adjacent vertices $a$ and $b$. From $G \rightarrow (3, 3)$ it follows easily that in $G(v) - \{a, b\}$ there is an edge of the first color and an edge of the second color. Therefore, we can suppose that in $G(v) - \{a, b\}$ there is exactly one edge of one of the colors, say the first color. We color $[v, a]$ and $[v, b]$ with the second color and the other three edges incident to $v$ with the first color. We obtain a 2-coloring of $\E(G)$ with exactly one monochromatic triangle.
\end{proof}

In the end, also in support to the hypothesis, we shall note that $M(K_3 + C_{2r + 1}) = 1, \ r \geq 2$ \cite{NK79}.

\subsection{Automorphism groups}

Denote by $Aut(G)$ the automorphism group of the graph $G$. We use the \emph{nauty} programs \cite{MP13} to find the number of automorphisms of the obtained minimal $(3, 3)$-Ramsey graphs with 10, 11, 12 and 13 vertices. Most of the obtained graphs have small automorphism groups (see Table \ref{table: 10-vertex graphs properties}, Table \ref{table: 11-vertex graphs properties}, Table \ref{table: 12-vertex graphs properties} and Table \ref{table: 13-vertex graphs properties}). We list the graphs with at least 60 automorphisms:

- The graphs in the form $K_3+C_{2r+1}$: $\abs{Aut(K_3+C_5)} = 60$. $\abs{Aut(K_3+C_7)} = 84$, $\abs{Aut(K_3+C_9)} = 108$;

- $\abs{Aut(G_{12.2240})} = 96$ (see Figure \ref{figure: 12_aut});

- $\abs{Aut(G_{13.255653})} = 144$, $\abs{Aut(G_{13.248305})} = 96$, $\abs{Aut(G_{13.304826})} = 96$, $\abs{Aut(G_{13.113198})} = 72$, $\abs{Aut(G_{13.175639})} = 72$, $\abs{Aut(G_{13.302168})} = 72$ (see Figure \ref{figure: 13_aut});

\section{Upper bounds on the independence number of the\\ minimal $(3, 3)$-Ramsey graphs}

\begin{figure}
	\captionsetup{justification=centering}
	\vspace{-4em}
	\begin{minipage}[t]{.35\textwidth}
		\centering
		\includegraphics[trim={0 0 0 490},clip,height=105px,width=105px]{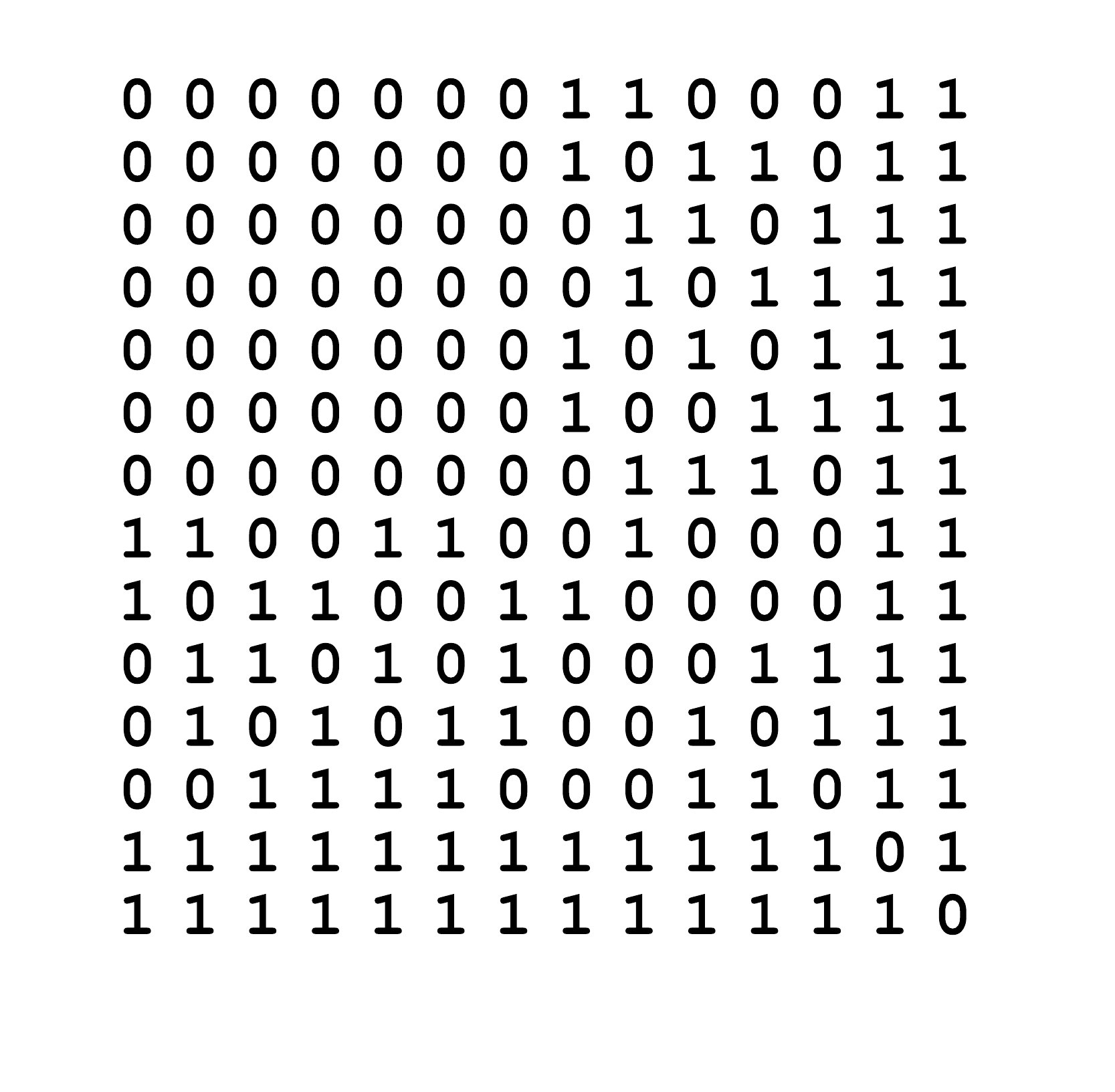}
		\caption{14-vertex minimal\\ $(3, 3)$-Ramsey graph\\ with independence number 7}
		\label{figure: 14_a7}
	\end{minipage}\hfill
	\begin{minipage}[t]{.65\textwidth}
		\centering
		\includegraphics[trim={0 0 0 490},clip,height=195px,width=195px]{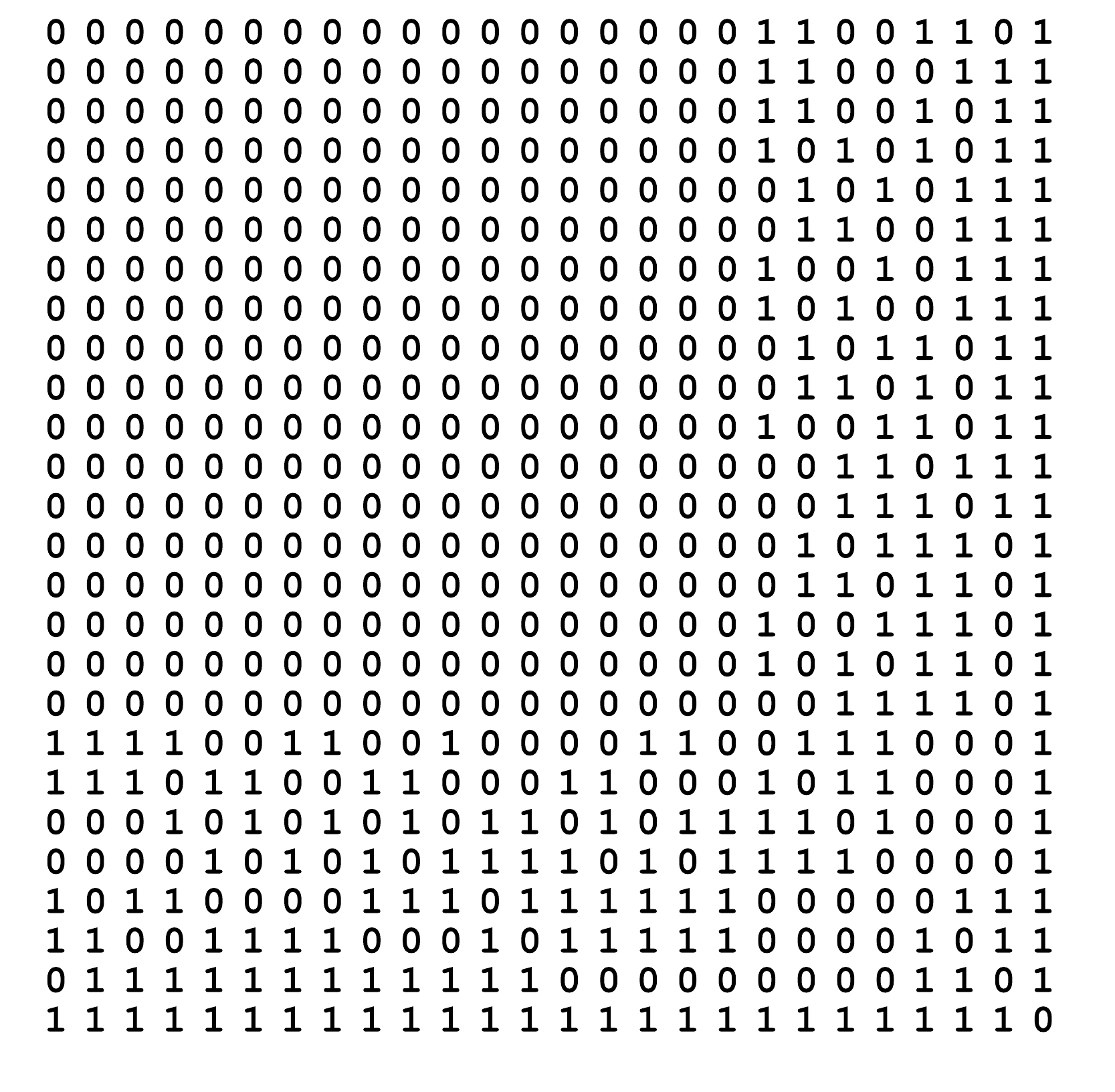}
		\caption{26-vertex minimal\\ $(3, 3)$-Ramsey graph\\ with independence number 18}
		\label{figure: 26_a18}
	\end{minipage}

	\vspace{2em}
	\centering
	\includegraphics[trim={0 0 0 490},clip,height=217.5px,width=217.5px]{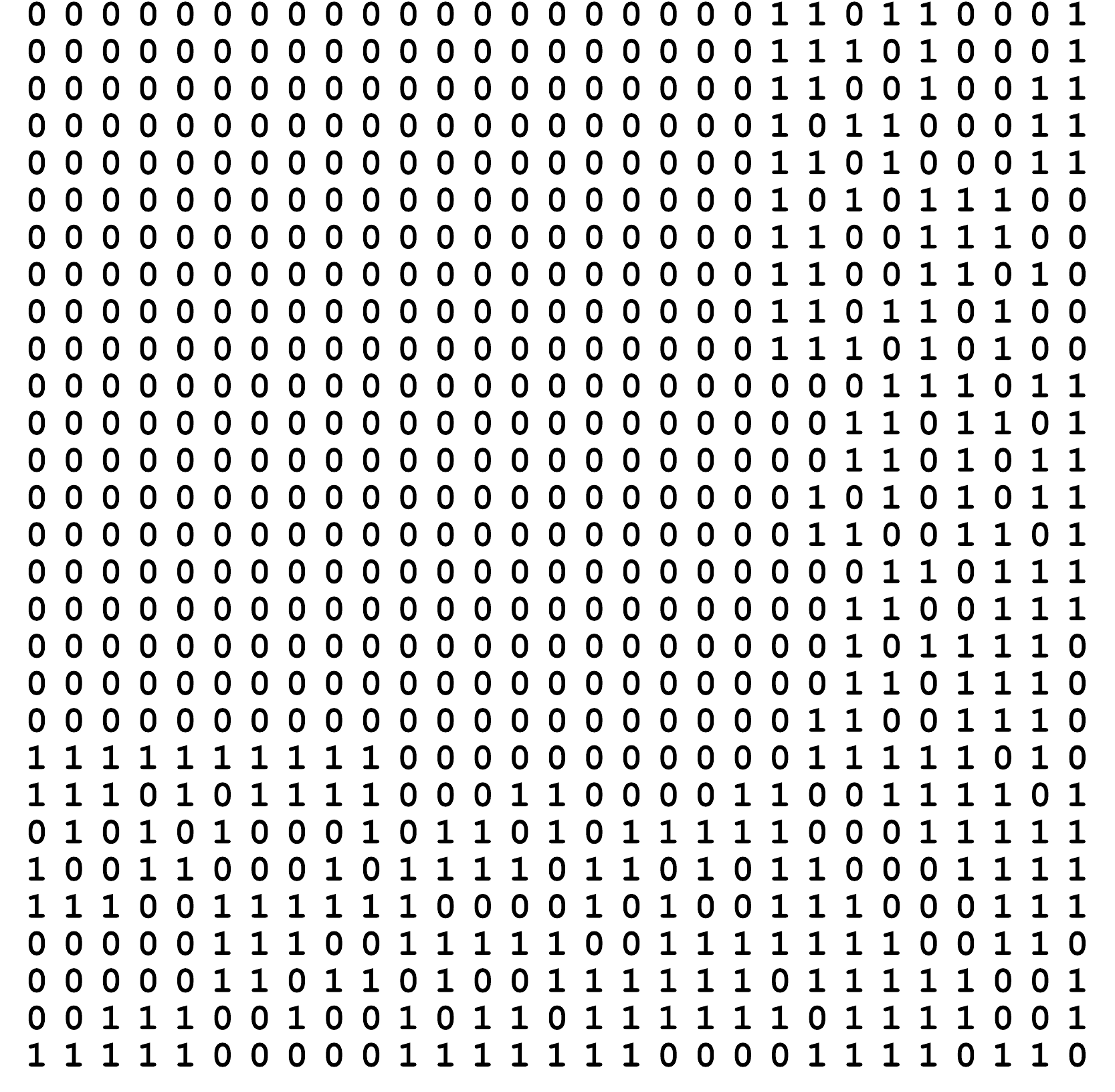}
	\caption{29-vertex minimal $(3, 3)$-Ramsey graph\\ with clique number 4 and independence number 20}
	\label{figure: 29_w4_a20}
\end{figure}

In regard to the maximal possible value of the independence number of the minimal $(3, 3)$-Ramsey graphs, the following theorem holds:

\begin{theorem}
\label{theorem: alpha(G) leq abs(V(G)) - 7}
\cite{Nen80b}
If $G$ is a minimal $(3, 3)$-Ramsey graph, $G \neq K_6$ and $G \neq K_3+C_5$, then $\alpha(G) \leq |V(G)| - 7$. There is a finite number of graphs for which equality is reached.
\end{theorem}

From Theorem \ref{theorem: alpha(G) leq abs(V(G)) - 7} it follows that by executing Algorithm \ref{algorithm: finding minimal (3, 3)-Ramsey graphs G for which alpha(G) geq abs(V(G)) - k geq 1}($q = 6; k = 8$) we obtain all minimal $(3, 3)$-Ramsey graphs $G$ for which $\alpha(G) = \abs{\V(G)} - 7$ or $\alpha(G) = \abs{\V(G)} - 8$. As a result of the execution of this algorithm we derive the following additions to Theorem \ref{theorem: alpha(G) leq abs(V(G)) - 7}:

\begin{theorem}
\label{theorem: (3, 3)-Ramsey graphs G for which alpha(G) = abs(V(G)) - 7}
There are exactly 11 minimal $(3, 3)$-Ramsey graphs $G$, for which $\alpha(G) = \abs{\V(G)} - 7$:

-  9-vertex: 1 (Figure \ref{figure: Nenov_9});

- 10-vertex: 3 ($G_{10.1}$, $G_{10.2}$, $G_{10.4}$, see Figure \ref{figure: 10});

- 11-vertex: 3 ($G_{11.1}$, $G_{11.2}$, $G_{11.21}$, see Figure \ref{figure: 11_a4});

- 12-vertex: 1 ($G_{12.163}$, see Figure \ref{figure: 12_a5});

- 13-vertex: 2 ($G_{13.}$, $G_{13.}$, see Figure \ref{figure: 13_a6}) ;

- 14-vertex: 1 (see Figure \ref{figure: 14_a7});
\end{theorem}

\begin{theorem}
\label{theorem: (3, 3)-Ramsey graphs G for which alpha(G) = abs(V(G)) - 8}
There are exactly 8633 minimal $(3, 3)$-Ramsey graphs $G$ for which $\alpha(G) = \abs{\V(G)} - 8$. The largest of these graphs has 26 vertices, and it is given on Figure \ref{figure: 26_a18}. There is only one minimal $(3, 3)$-Ramsey graph $G$ for which $\alpha(G) = \abs{\V(G)} - 8$ and $\omega(G) < 5$, and it is the 15-vertex graph $K_1 + \Gamma$ from \cite{Nen81a} (see Figure \ref{figure: Nenov_14}).
\end{theorem}

\begin{corollary}
Let $G$ be a minimal $(3, 3)$-Ramsey graph and $\abs{\V(G)} \geq 27$. Then, $\alpha(G) \leq \abs{\V(G)} - 9$.
\end{corollary}

According to Theorem \ref{theorem: (3, 3)-Ramsey graphs G for which alpha(G) = abs(V(G)) - 8}, if $G$ is a minimal $(3, 3)$-Ramsey graph, $\omega(G) < 5$, and $G \neq K_1 + \Gamma$, then $\alpha(G) \leq \abs{\V(G)} - 9$. From Theorem \ref{theorem: F_e(3, 3; 5) = 15} it follows that by executing Algorithm \ref{algorithm: finding minimal (3, 3)-Ramsey graphs G for which alpha(G) geq abs(V(G)) - k geq 1}($q = 5; k = 9$) we obtain all minimal $(3, 3)$-Ramsey graphs $G$ for which $\omega(G) < 5$ and $\alpha(G) = \abs{\V(G)} - 9$, and the graph $K_1 + \Gamma$. As a result of the execution of this algorithm we derive:

\begin{theorem}
\label{theorem: (3, 3)-Ramsey graphs G for which omega (G) < 5 and alpha(G) = abs(V(G)) - 9}
There are exactly 8903 minimal $(3, 3)$-Ramsey graphs $G$ for which $\omega(G) < 5$ and $\alpha(G) = \abs{\V(G)} - 9$. The largest of these graphs has 29 vertices, and it is given на Figure \ref{figure: 29_w4_a20}.
\end{theorem}

\begin{corollary}
Let $G$ be a minimal $(3, 3)$-Ramsey graph such that $\omega(G) < 5$ and $\abs{\V(G)} \geq 30$. Then, $\alpha(G) \leq \abs{\V(G)} - 10$.
\end{corollary}

\section{Lower bounds on the minimum degree of the\\ minimal $(3, 3)$-Ramsey graphs}

According to Proposition \ref{proposition: marked vertex sets in (3, 3)-Ramsey graphs}, if $G$ is a minimal $(3, 3)$-Ramsey graph, then for each vertex $v$ of $G$, $N_G(v)$ is a marked vertex set in $G - v$, and therefore $N_G(v)$ is a marked vertex set in $G(v)$.

\begin{figure}[h]
	\centering
	\includegraphics[trim={0 470 0 0},clip,height=120px,width=120px]{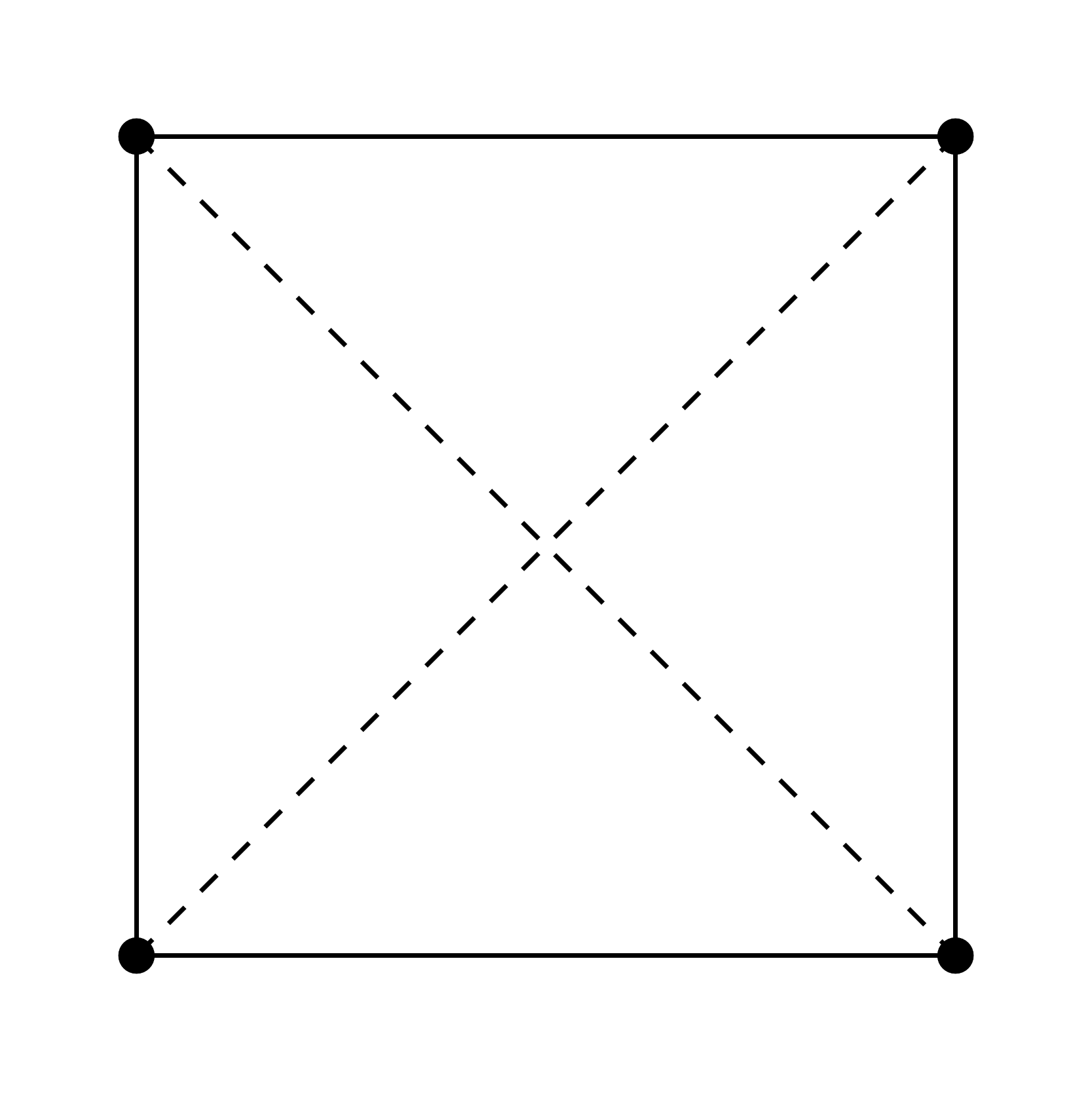}
	\caption{$(3, 3)$-free 2-coloring of the edges of $K_4$}
	\label{figure: N_4_coloring}
\end{figure}

It is easy to see that if $W \subseteq \V(G)$ and $\abs{W} \leq 3$, or $\abs{W} = 4$ and $G[W] \neq K_4$, then $W$ is not a marked vertex set in $G$. A $(3, 3)$-free 2-coloring of $K_4$ which cannot be extended to a $(3, 3)$-free 2-coloring of $K_5$ is shown on Figure \ref{figure: N_4_coloring}. Therefore, the only 4-vertex graph $N$ such that $\V(N)$ is a marked vertex set in $N$ is $K_4$.

\begin{figure}[h]
	\begin{subfigure}{.3\textwidth}
		\centering
		\includegraphics[trim={0 470 0 0},clip,height=100px,width=100px]{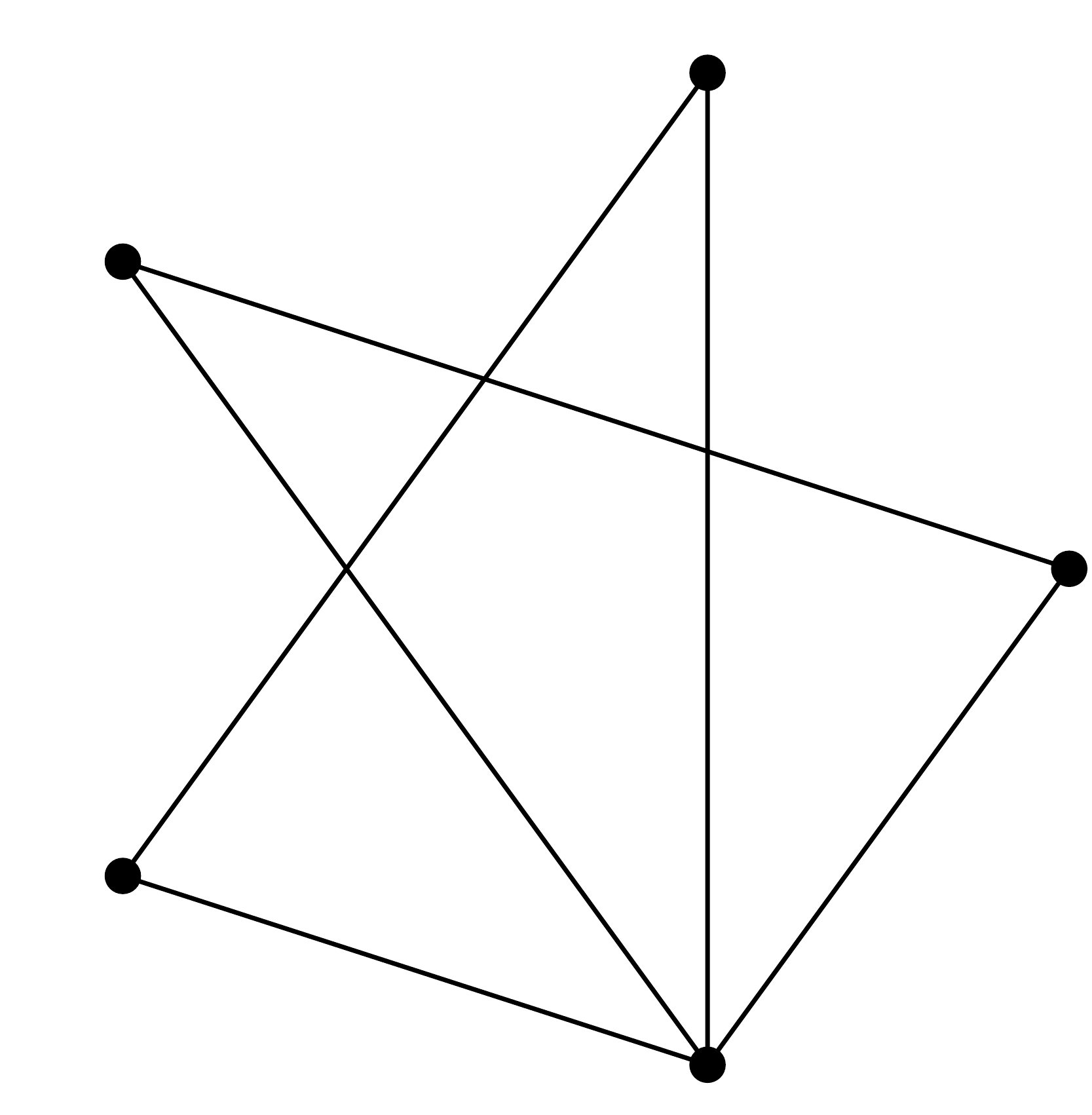}
		\caption*{$N_{5.1}$}
		\label{figure: N_5_1}
	\end{subfigure}\hfill
	\begin{subfigure}{.3\textwidth}
		\centering
		\includegraphics[trim={0 470 0 0},clip,height=100px,width=100px]{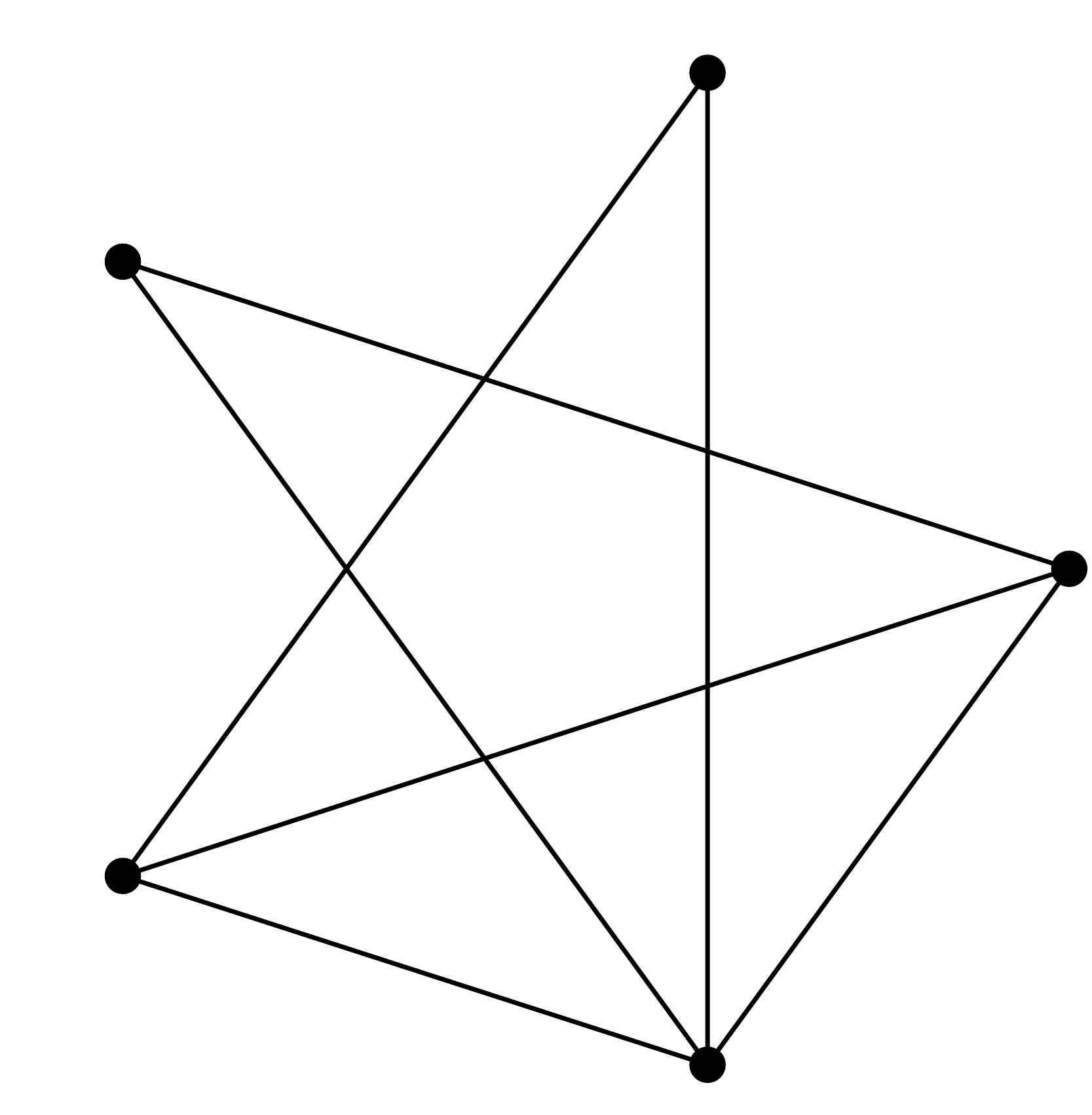}
		\caption*{$N_{5.2}$}
		\label{figure: N_5_2}
	\end{subfigure}\hfill
	\begin{subfigure}{.3\textwidth}
		\centering
		\includegraphics[trim={0 470 0 0},clip,height=100px,width=100px]{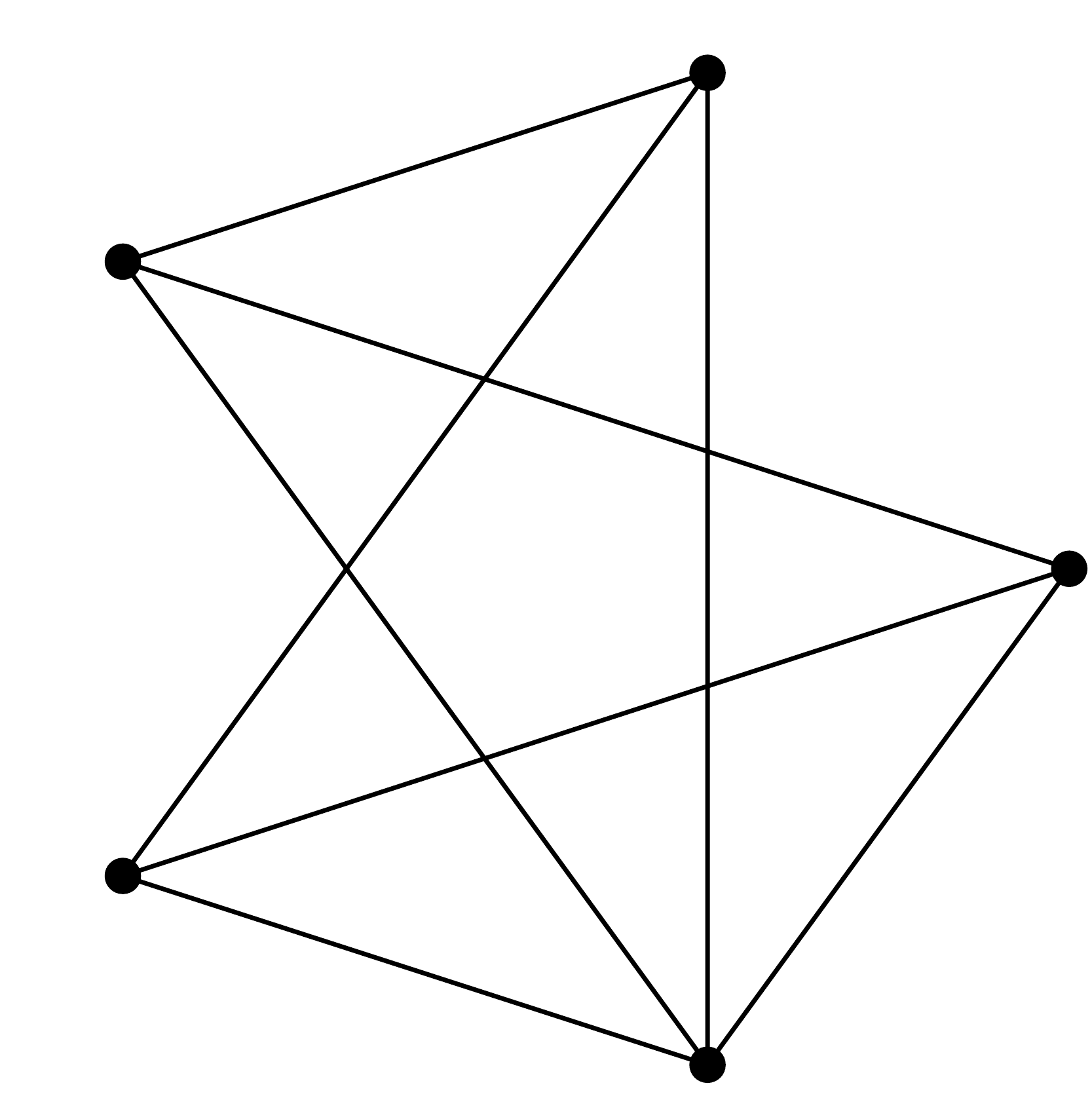}
		\caption*{$N_{5.3}$}
		\label{figure: N_5_3}
	\end{subfigure}
	\caption{The graphs $N_{5.1}$, $N_{5.2}$, $N_{5.3}$}
	\label{figure: N_5_1 N_5_2 N_5_3}
\end{figure}

With the help of a computer, we obtain that there are exactly 3 graphs $N$ with 5 vertices such that $K_4 \not\subset N$ and $\V(N)$ is a marked vertex set in $N$. Namely, they are the graphs $N_{5.1}$, $N_{5.2}$ and $N_{5.3}$ presented on Figure \ref{figure: N_5_1 N_5_2 N_5_3}. We shall note that $N_{5.1} \subset N_{5.2} \subset N_{5.3}$. From these results we derive
\begin{theorem}
\label{theorem: delta(G) geq 5}
Let $G$ be a minimal $(3, 3)$-Ramsey graph and $\omega(G) \leq 4$. Then, $\delta(G) \geq 5$. If $v \in \V(G)$ and $d(v) = 5$, then $G(v) = N_{5.i}$ for some $i \in \set{1, 2, 3}$ (see Figure \ref{figure: N_5_1 N_5_2 N_5_3}).
\end{theorem}
The bound $\delta(G) \geq 5$ from Theorem \ref{theorem: delta(G) geq 5} is exact. For example, the graph $G = K_1 + \Gamma$ from \cite{Nen81a} (see Figure \ref{figure: Nenov_14}) has 7 vertices $v$ such that $d(v) = 5$ and $G(v) = N_{5.3}$.

\begin{figure}[h]
	\begin{subfigure}{.25\textwidth}
		\centering
		\includegraphics[trim={0 470 0 0},clip,height=80px,width=80px]{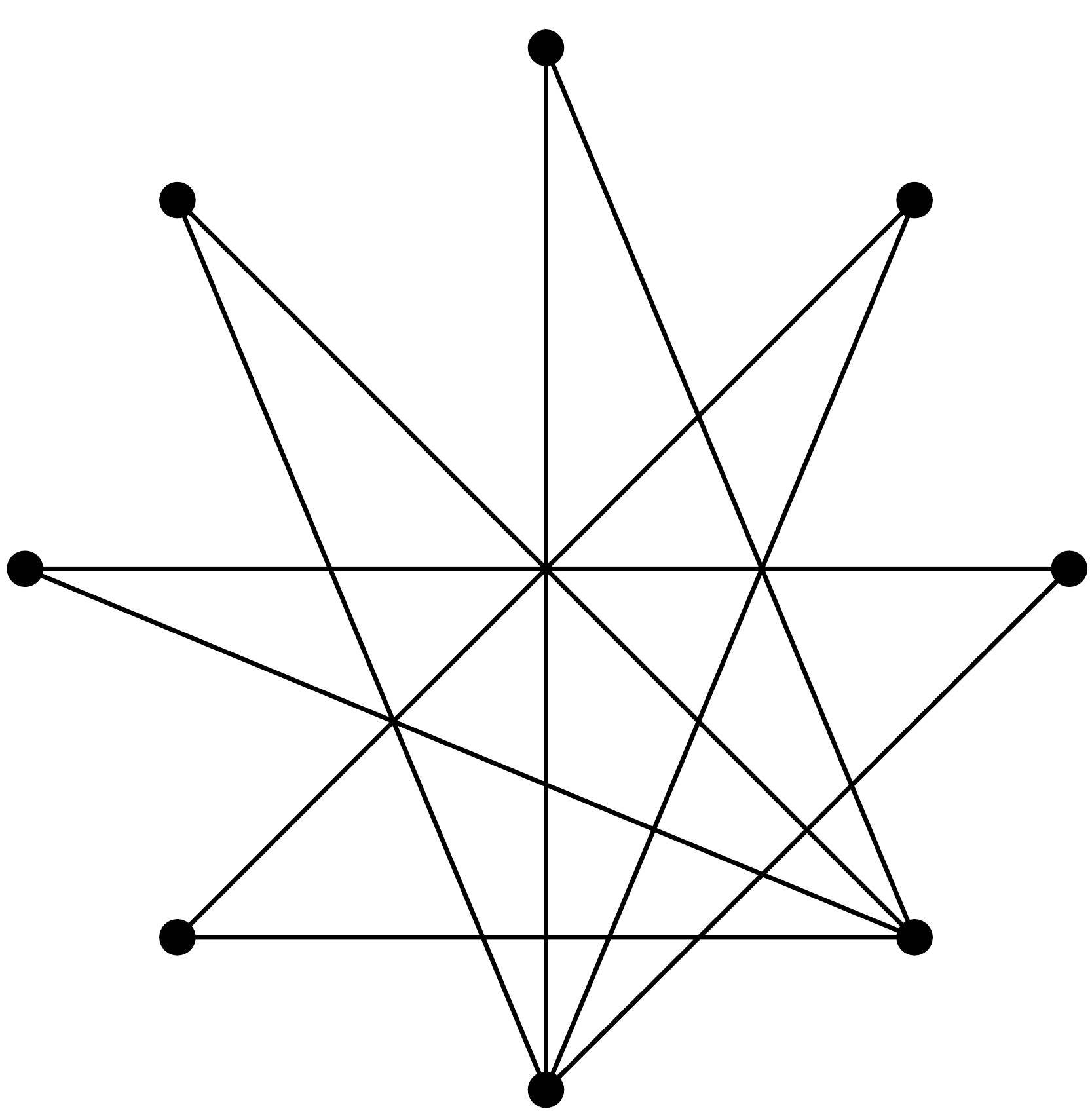}
		\caption*{$N_{8.1}$}
		\label{figure: N_8_1}
	\end{subfigure}\hfill
	\begin{subfigure}{.25\textwidth}
		\centering
		\includegraphics[trim={0 470 0 0},clip,height=80px,width=80px]{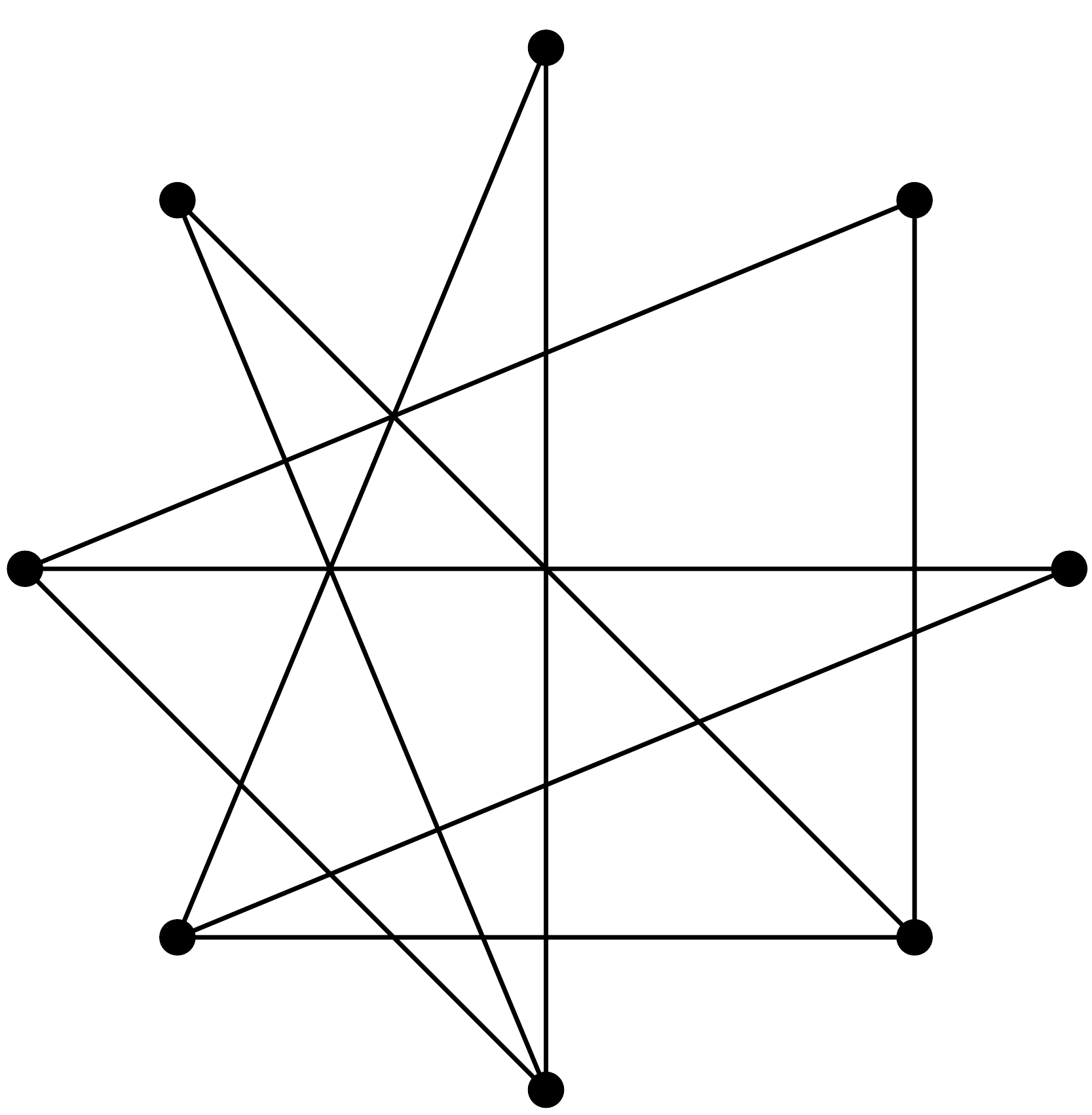}
		\caption*{$N_{8.2}$}
		\label{figure: N_8_2}
	\end{subfigure}\hfill
	\begin{subfigure}{.25\textwidth}
		\centering
		\includegraphics[trim={0 470 0 0},clip,height=80px,width=80px]{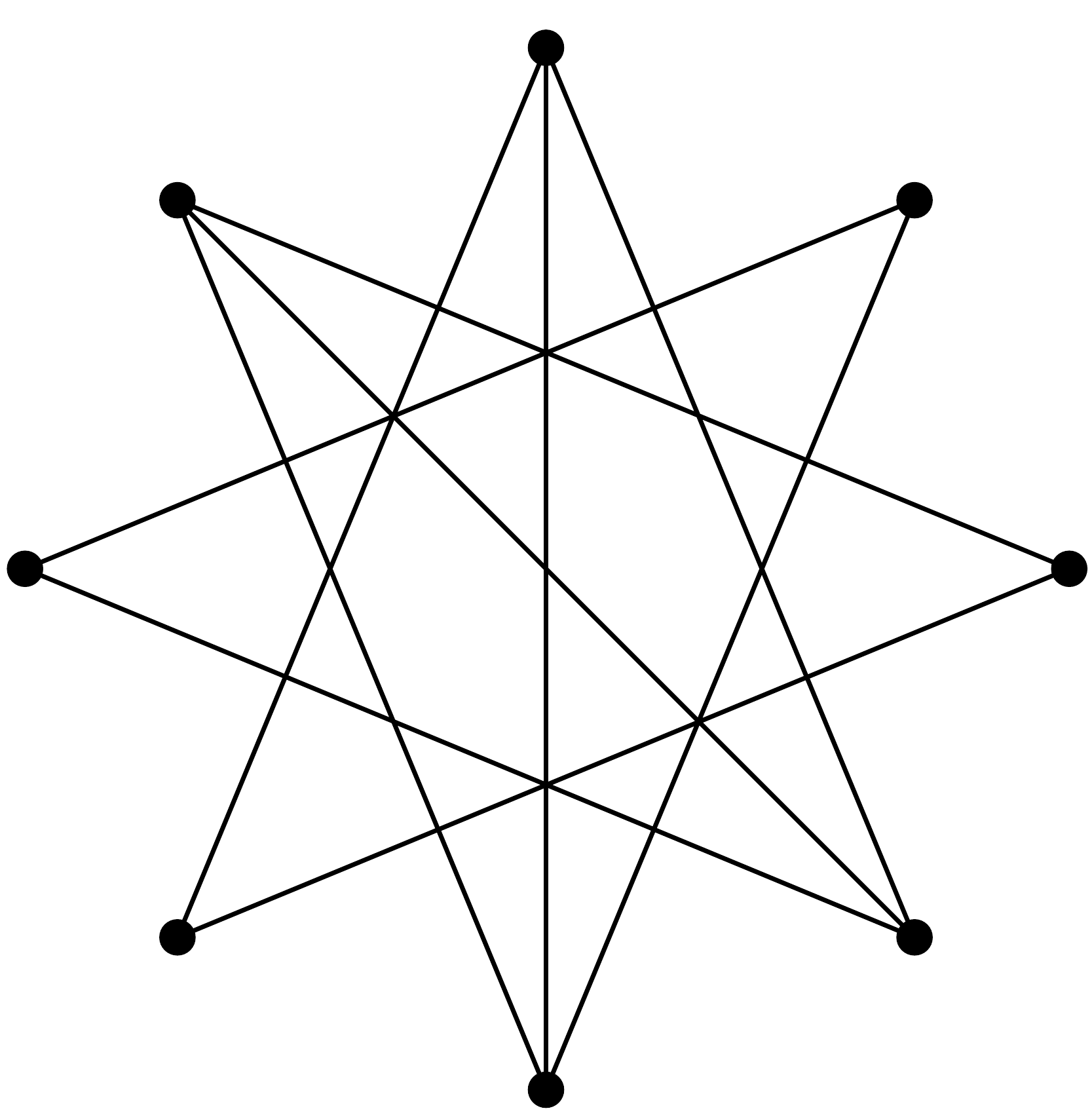}
		\caption*{$N_{8.3}$}
		\label{figure: N_8_3}
	\end{subfigure}\\
	\vspace{1em}
	\begin{subfigure}{.22\textwidth}
		\centering
		\includegraphics[trim={0 470 0 0},clip,height=75px,width=75px]{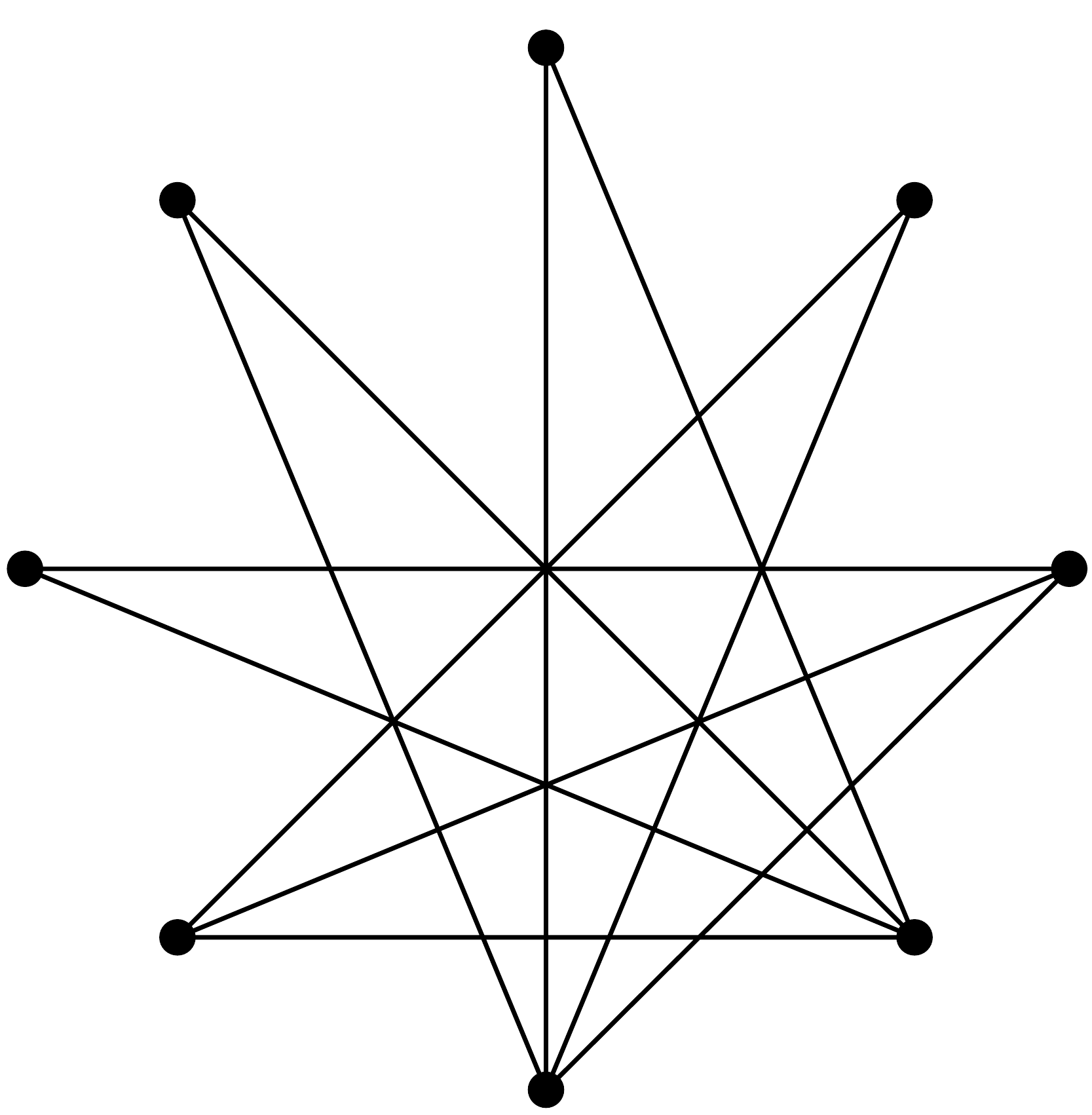}
		\caption*{$N_{8.4}$}
		\label{figure: N_8_4}
	\end{subfigure}\hfill
	\begin{subfigure}{.22\textwidth}
		\centering
		\includegraphics[trim={0 470 0 0},clip,height=75px,width=75px]{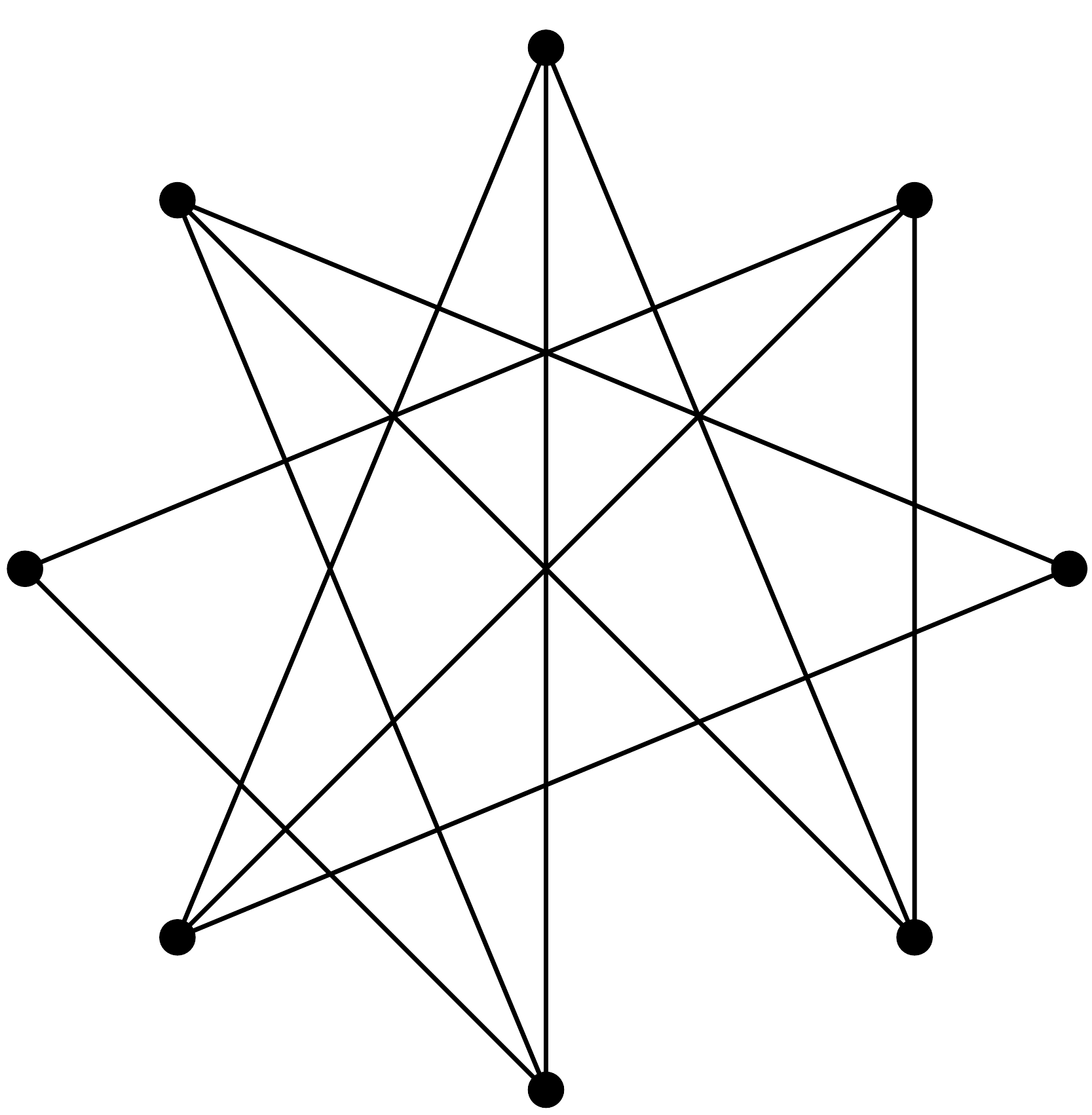}
		\caption*{$N_{8.5}$}
		\label{figure: N_8_5}
	\end{subfigure}\hfill
	\begin{subfigure}{.22\textwidth}
		\centering
		\includegraphics[trim={0 470 0 0},clip,height=75px,width=75px]{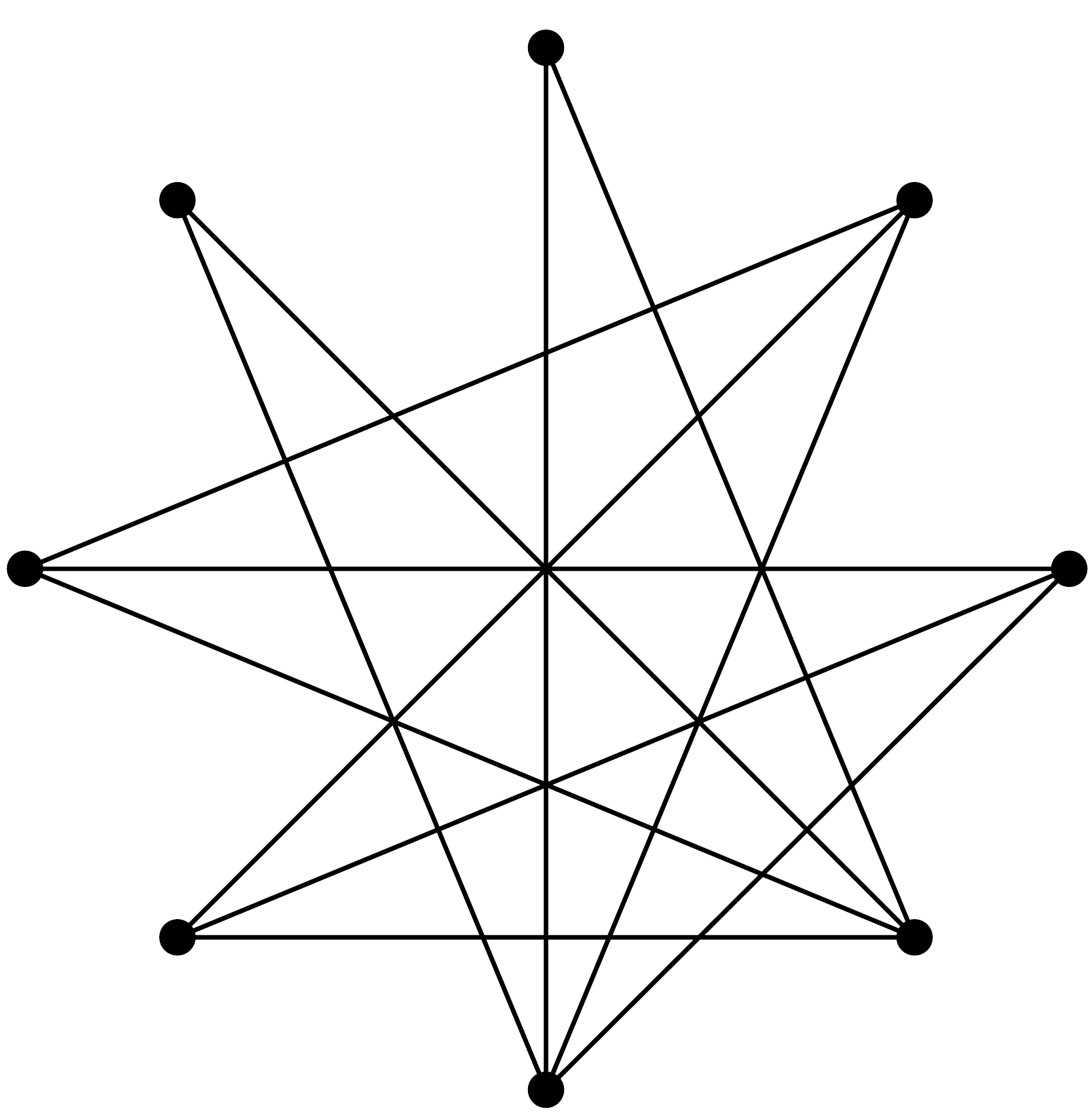}
		\caption*{$N_{8.6}$}
		\label{figure: N_8_6}
	\end{subfigure}\hfill
	\begin{subfigure}{.22\textwidth}
		\centering
		\includegraphics[trim={0 470 0 0},clip,height=75px,width=75px]{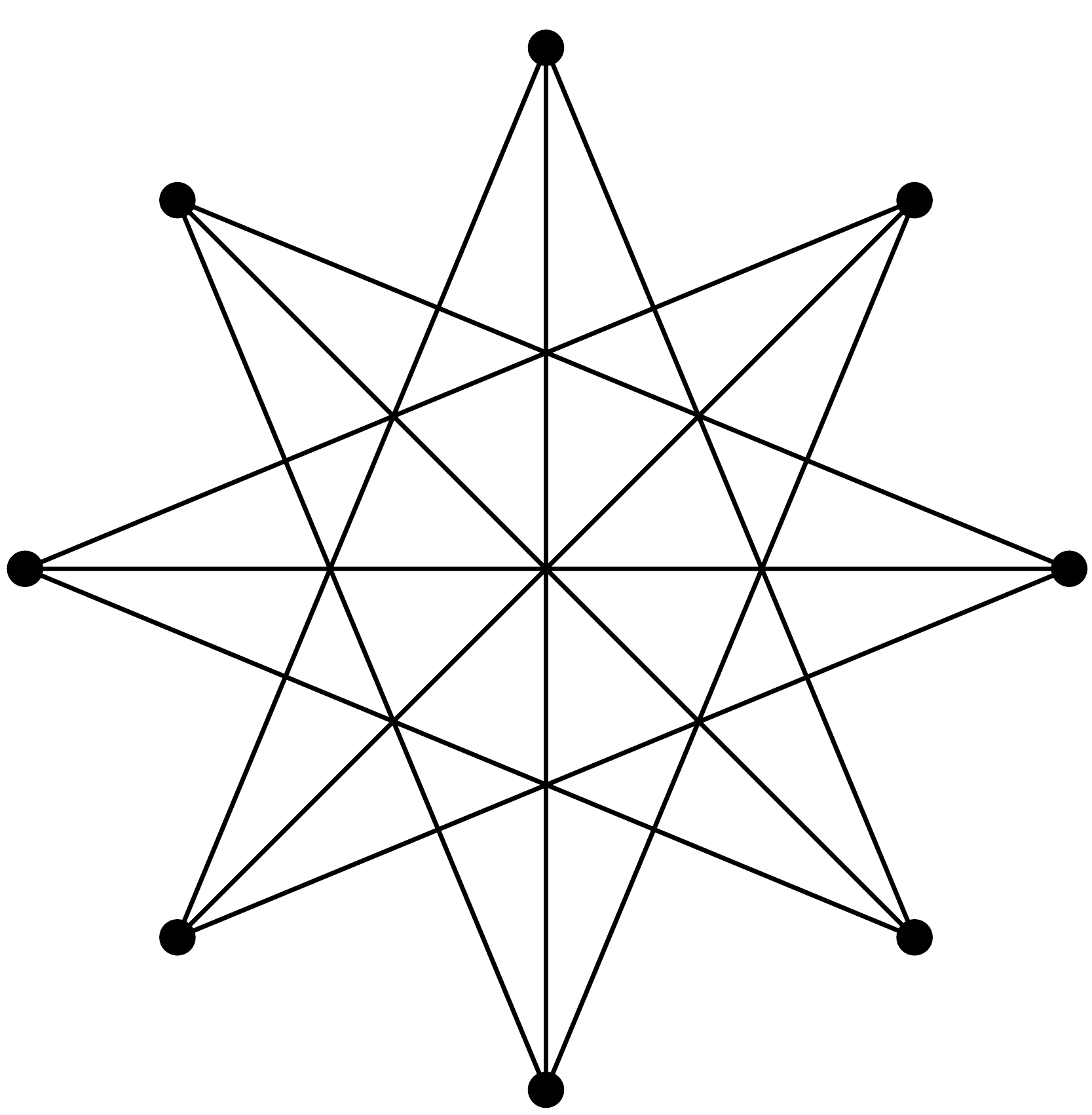}
		\caption*{$N_{8.7}$}
		\label{figure: N_8_7}
	\end{subfigure}
	\caption{The graphs $N_{8.i}, i = 1, ..., 7$}
	\label{figure: N_8_1 N_8_2 N_8_3 N_8_4 N_8_5 N_8_6 N_8_7}
\end{figure}

Also with the help of a computer, we obtain that the smallest graphs $N$ such that $K_3 \not\subset N$ and $\V(N)$ is a marked vertex set in $N$ have 8 vertices, and there are exactly 7 such graphs. Namely, they are the graphs $N_{8.i}, i = 1, ..., 7$ presented on Figure \ref{figure: N_8_1 N_8_2 N_8_3 N_8_4 N_8_5 N_8_6 N_8_7}. Among them, the minimal graphs are $N_{8.1}$, $N_{8.2}$ and $N_{8.3}$, and the remaining 4 graphs are their supergraphs. Thus, we derive the following
\begin{theorem}
\label{theorem: delta(G) geq 8}
Let $G$ be a minimal $(3, 3)$-Ramsey graph and $\omega(G) = 3$. Then, $\delta(G) \geq 8$. If $v \in \V(G)$ and $d(v) = 8$, then $G(v) = N_{8.i}$ for some $i \in \set{1, ..., 7}$ (see Figure \ref{figure: N_8_1 N_8_2 N_8_3 N_8_4 N_8_5 N_8_6 N_8_7}).
\end{theorem}


\vspace{14pt}

\centerline{ACKNOWLEDGEMENTS}

\vspace{12pt}

I would like to thank prof. Nedyalko Nenov, who read the manuscript and made some suggestions which led to the improvement of this work.


\clearpage

\centerline{APPENDICES}

\vspace{14pt}

\centerline{A. GRAPHS}

\vspace{12pt}

\begin{figure}[h]
	\captionsetup{justification=centering}
	\begin{subfigure}{.33\textwidth}
		\centering
		\includegraphics[trim={0 0 0 490},clip,height=100px,width=100px]{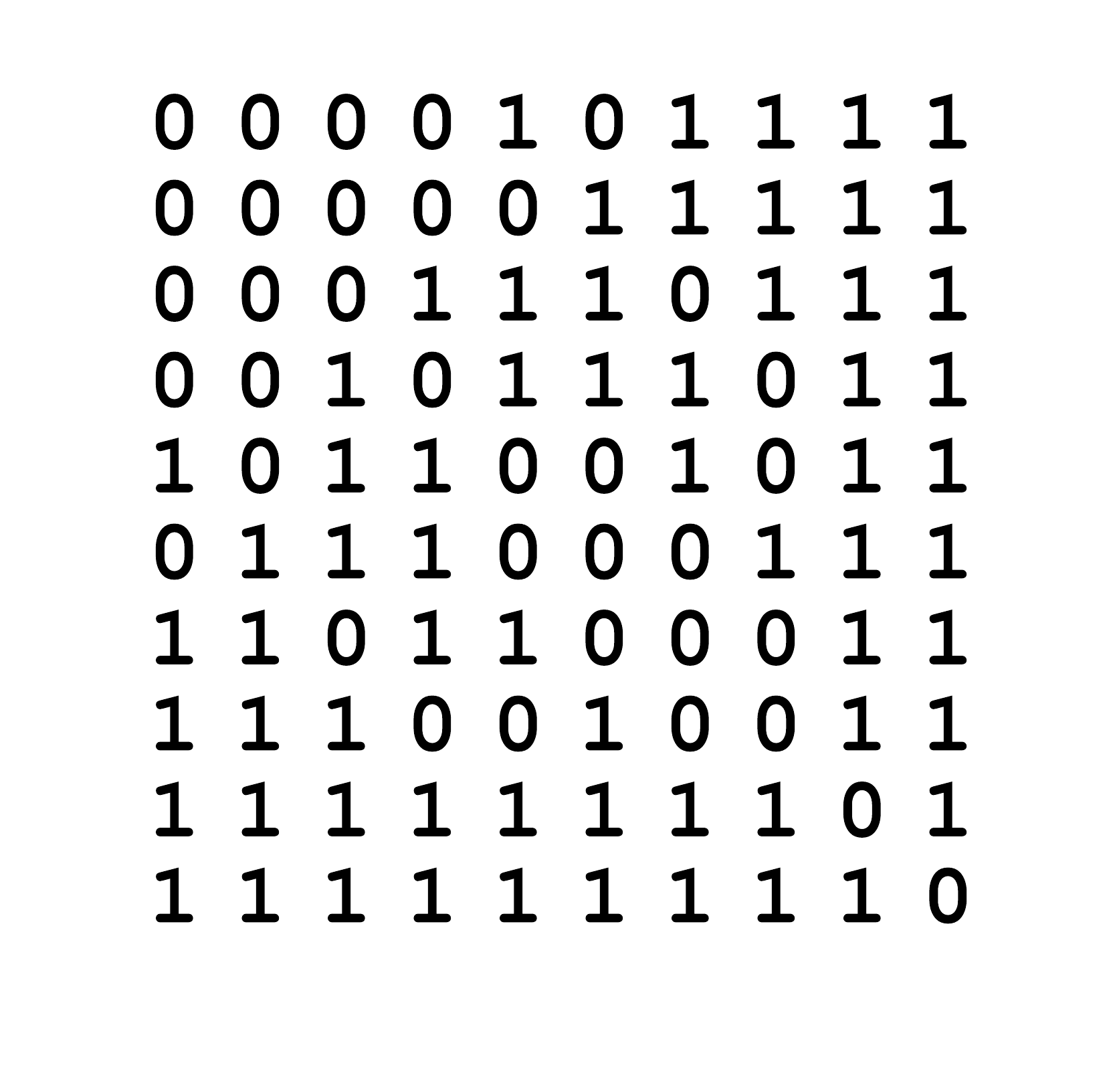}
		\vspace{-1.5em}
		\caption*{$G_{10.1}$}
		\label{figure: 10_1}
	\end{subfigure}\hfill
	\begin{subfigure}{.33\textwidth}
		\centering
		\includegraphics[trim={0 0 0 490},clip,height=100px,width=100px]{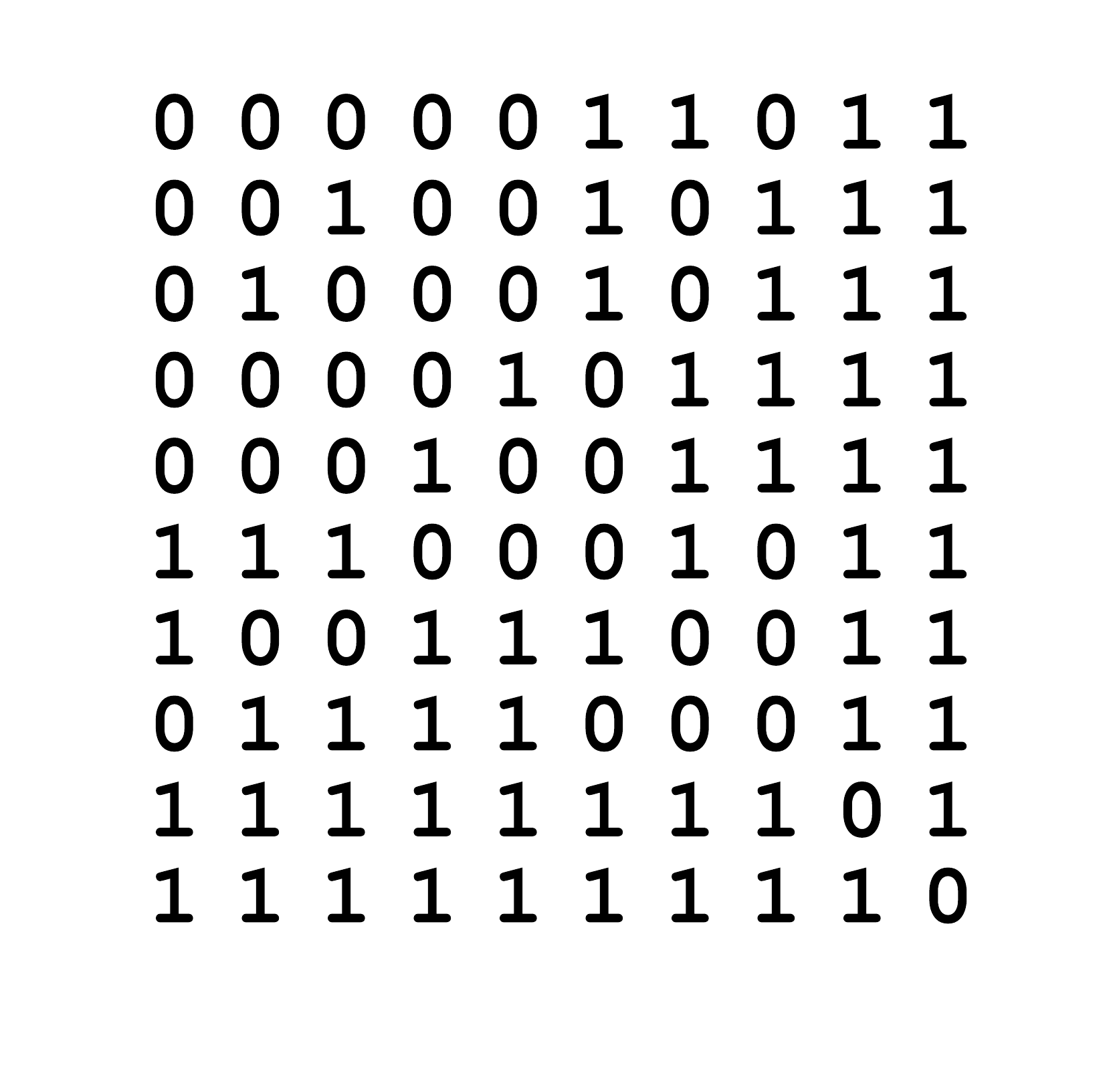}
		\vspace{-1.5em}
		\caption*{$G_{10.2}$}
		\label{figure: 10_2}
	\end{subfigure}\hfill
	\begin{subfigure}{.33\textwidth}
		\centering
		\includegraphics[trim={0 0 0 490},clip,height=100px,width=100px]{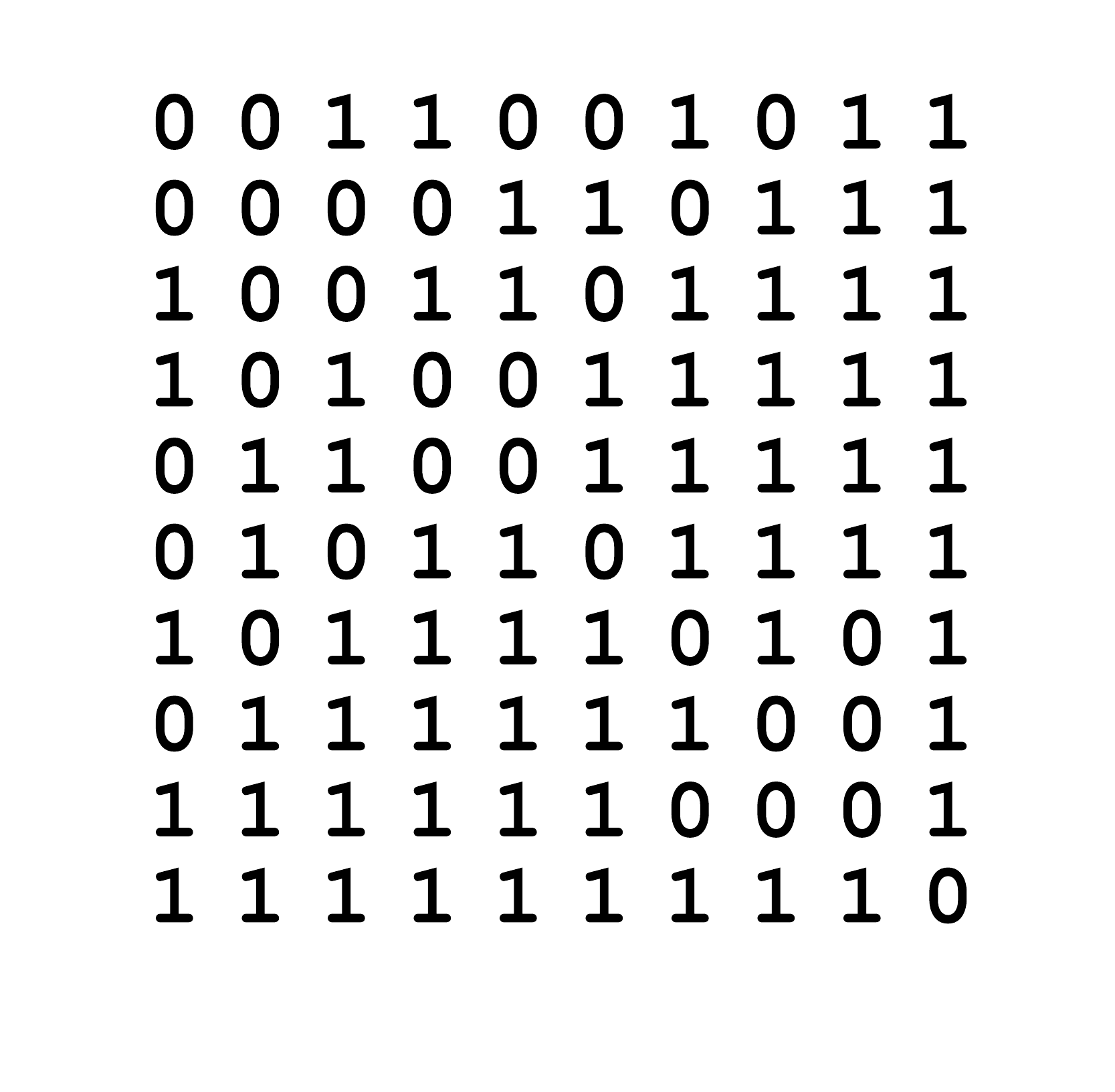}
		\vspace{-1.5em}
		\caption*{$G_{10.3}$}
		\label{figure: 10_3}
	\end{subfigure}
	
	\vspace{0.5em}
	\begin{subfigure}{.33\textwidth}
		\centering
		\includegraphics[trim={0 0 0 490},clip,height=100px,width=100px]{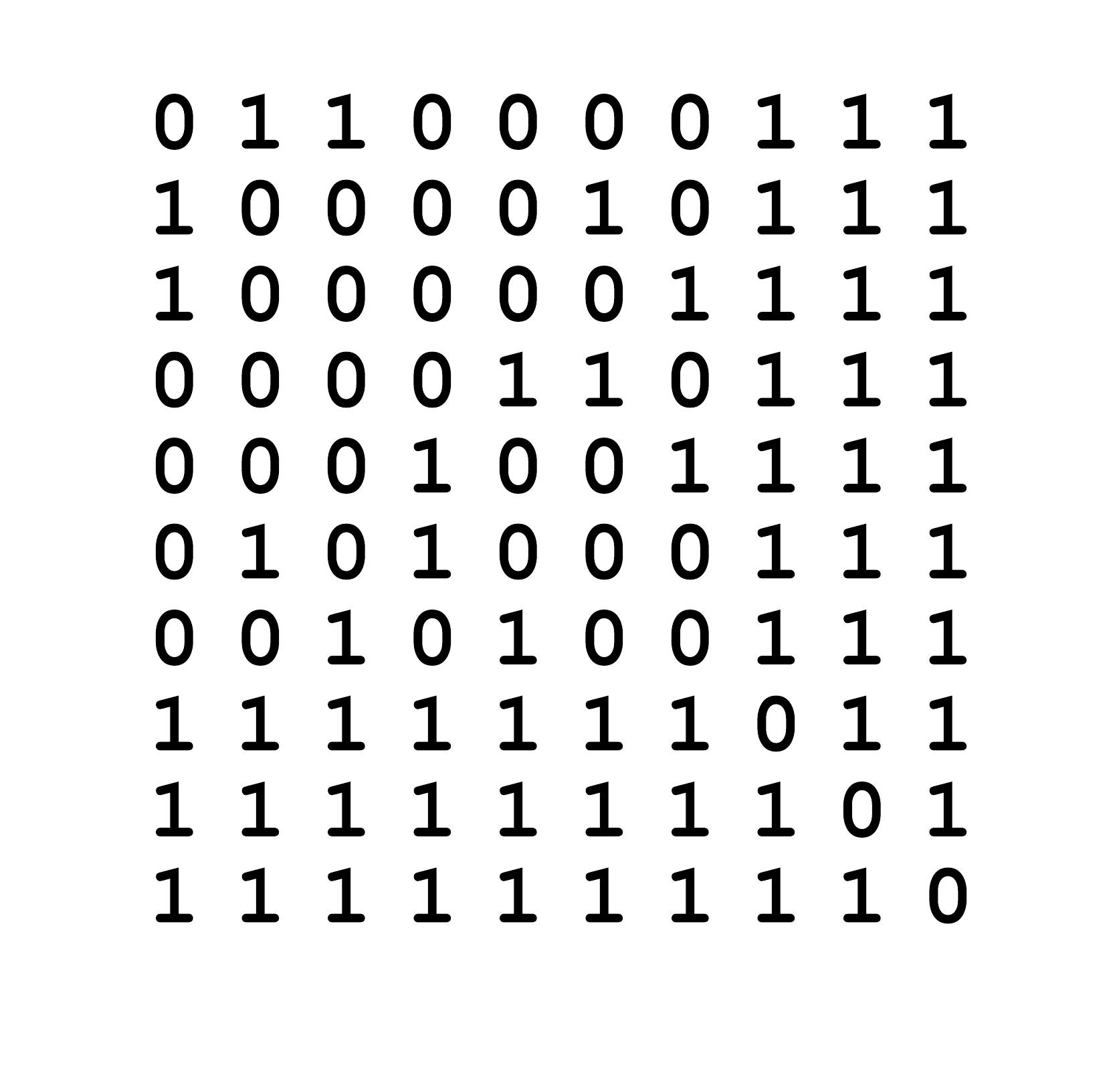}
		\vspace{-1.5em}
		\caption*{$G_{10.4}$}
		\label{figure: 10_4}
	\end{subfigure}\hfill
	\begin{subfigure}{.33\textwidth}
		\centering
		\includegraphics[trim={0 0 0 490},clip,height=100px,width=100px]{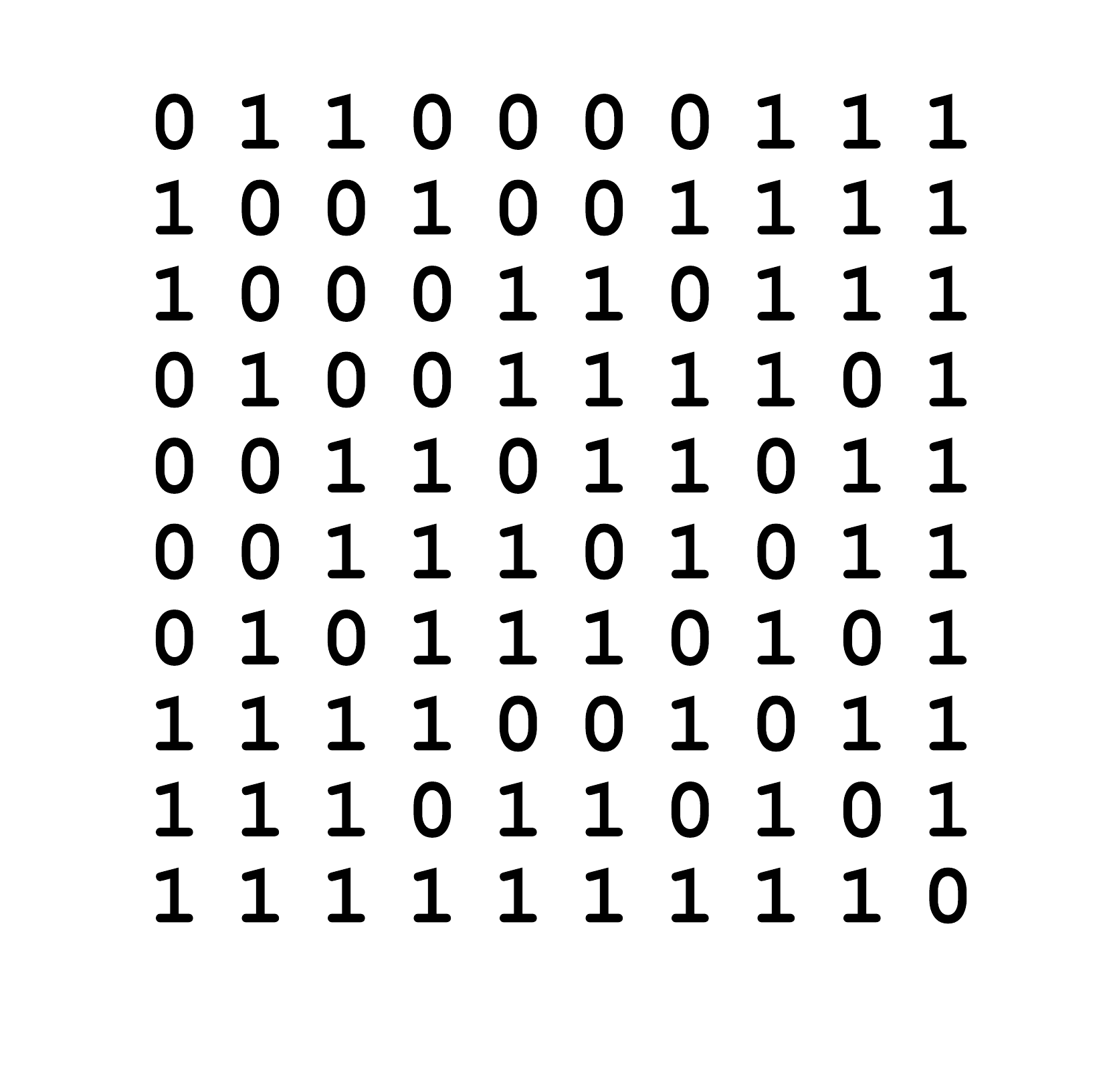}
		\vspace{-1.5em}
		\caption*{$G_{10.5}$}
		\label{figure: 10_5}
	\end{subfigure}\hfill
	\begin{subfigure}{.33\textwidth}
		\centering
		\includegraphics[trim={0 0 0 490},clip,height=100px,width=100px]{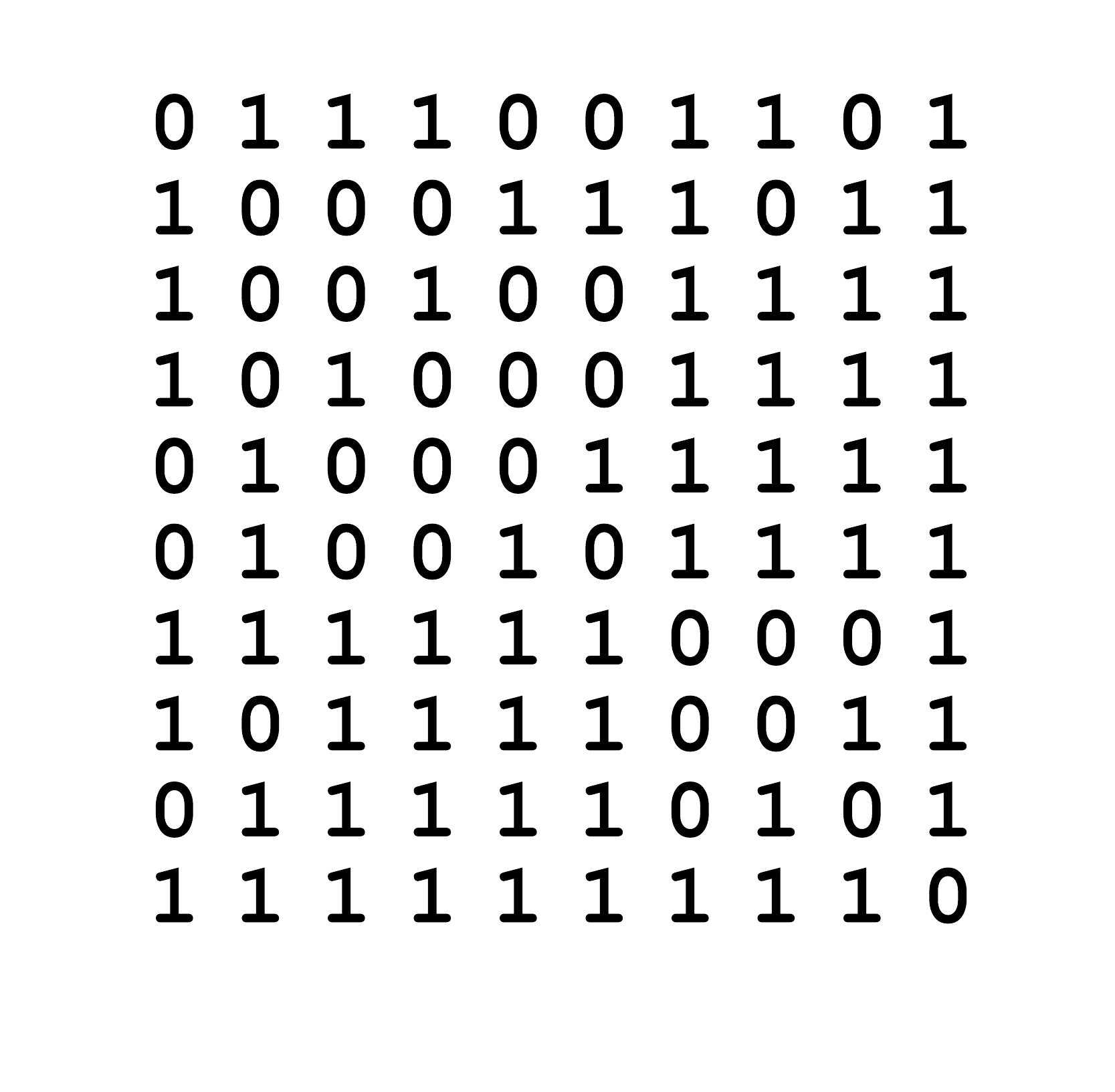}
		\vspace{-1.5em}
		\caption*{$G_{10.6}$}
		\label{figure: 10_6}
	\end{subfigure}
	
	\vspace{0.5em}
	\caption{10-vertex minimal $(3, 3)$-Ramsey graphs}
	\label{figure: 10}
\end{figure}

\begin{figure}[h]
	\captionsetup{justification=centering}
	\begin{subfigure}{.33\textwidth}
		\centering
		\includegraphics[trim={0 0 0 490},clip,height=110px,width=110px]{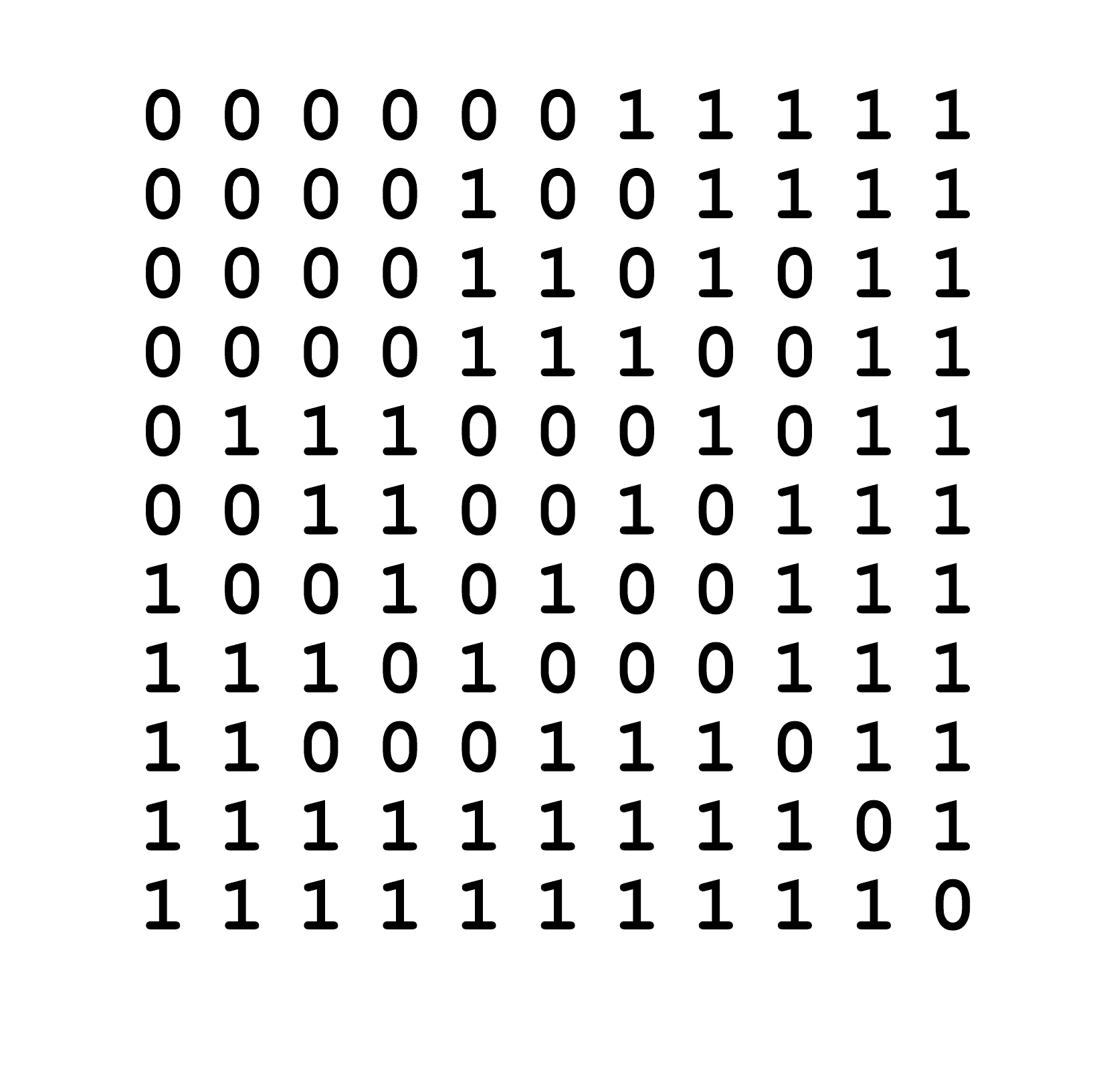}
		\vspace{-1.5em}
		\caption*{$G_{11.1}$}
		\label{figure: 11_1}
	\end{subfigure}\hfill
	\begin{subfigure}{.33\textwidth}
		\centering
		\includegraphics[trim={0 0 0 490},clip,height=110px,width=110px]{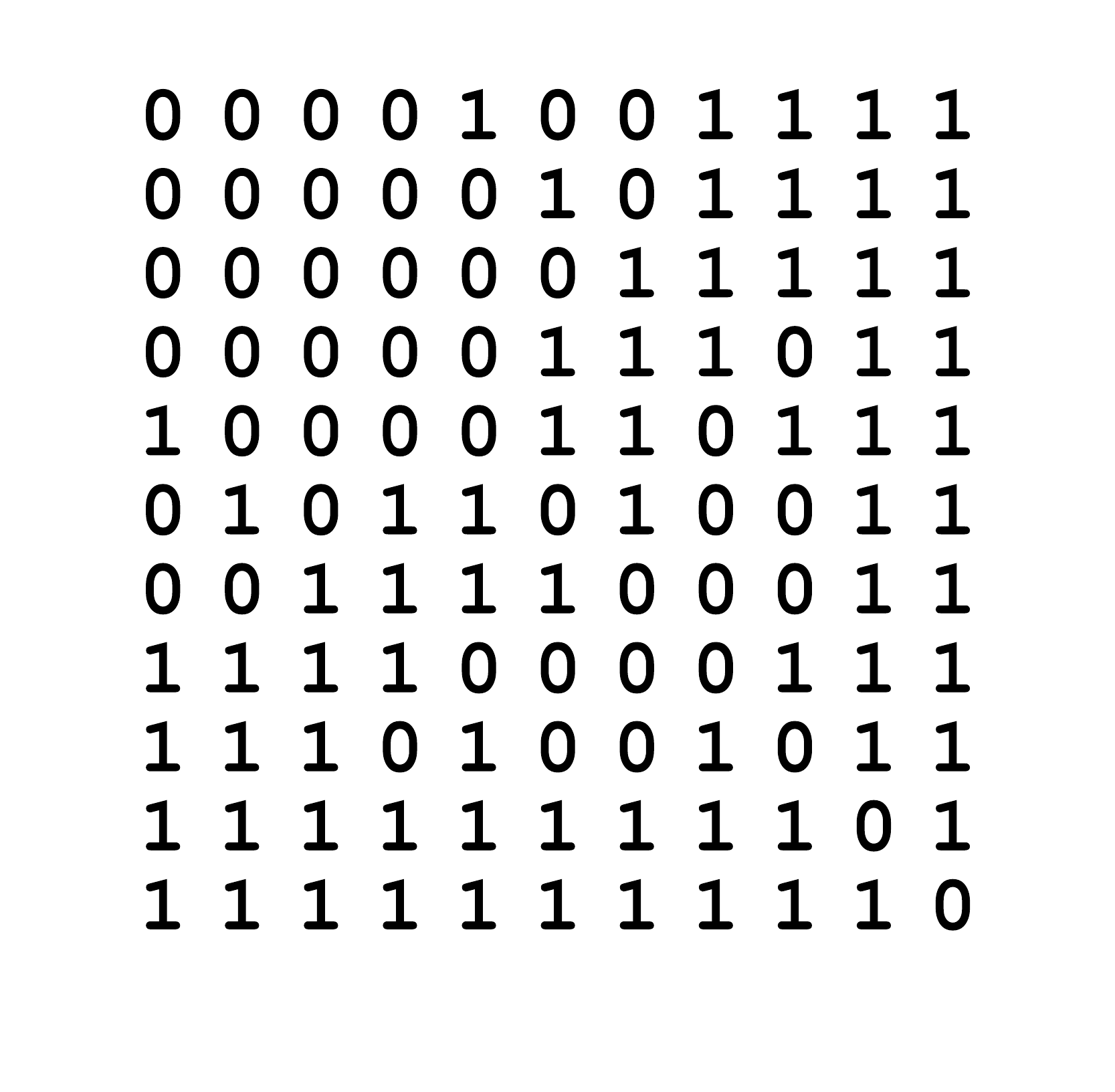}
		\vspace{-1.5em}
		\caption*{$G_{11.2}$}
		\label{figure: 11_2}
	\end{subfigure}\hfill
	\begin{subfigure}{.33\textwidth}
		\centering
		\includegraphics[trim={0 0 0 490},clip,height=110px,width=110px]{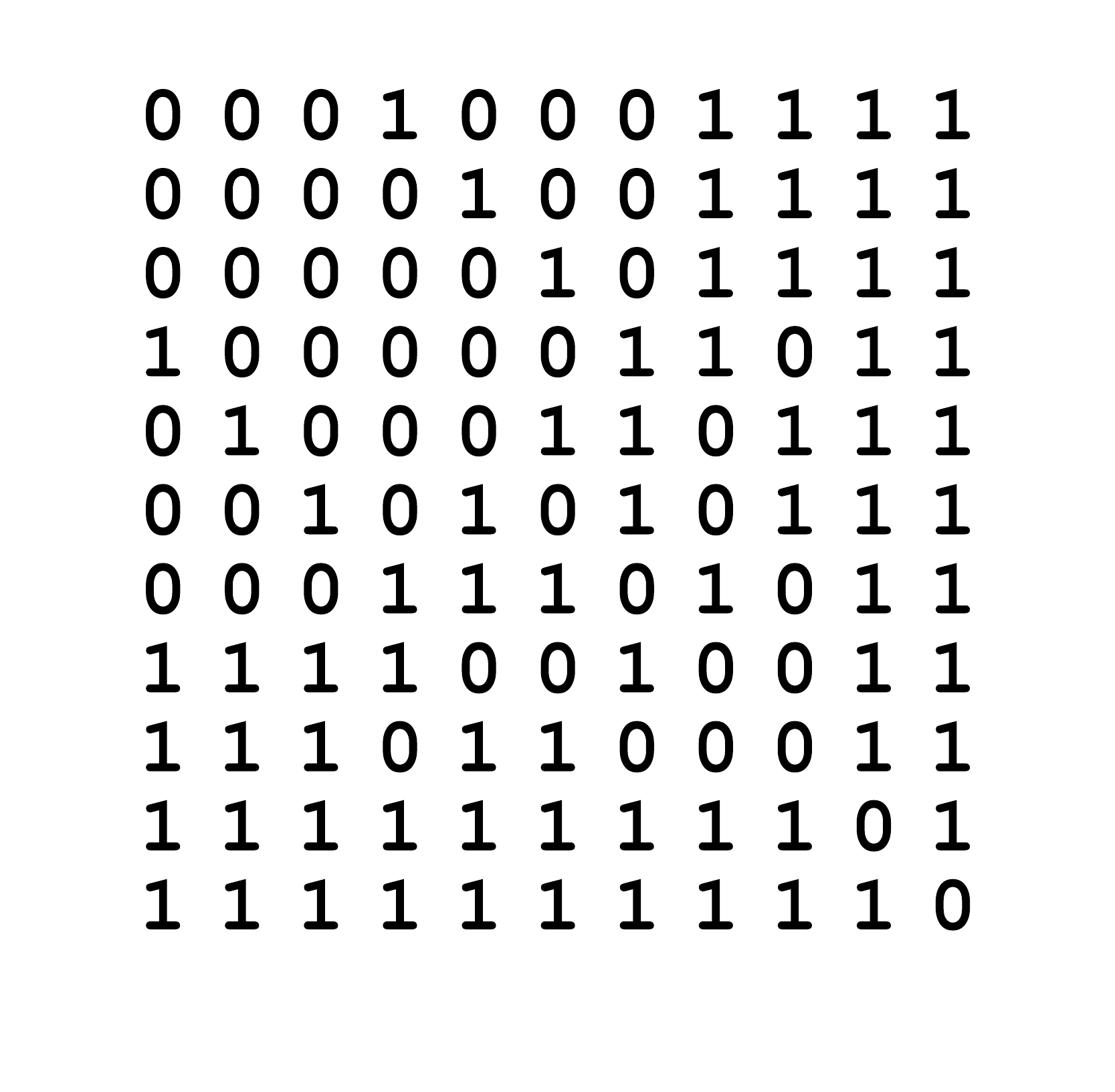}
		\vspace{-1.5em}
		\caption*{$G_{11.21}$}
		\label{figure: 11_21}
	\end{subfigure}
	
	\vspace{0.5em}
	\caption{11-vertex minimal $(3, 3)$-Ramsey graphs\\ with independence number 4}
	\label{figure: 11_a4}
\end{figure}

\begin{figure}[h]
	\captionsetup{justification=centering}
	\begin{subfigure}{.45\textwidth}
		\centering
		\includegraphics[trim={0 0 0 490},clip,height=110px,width=110px]{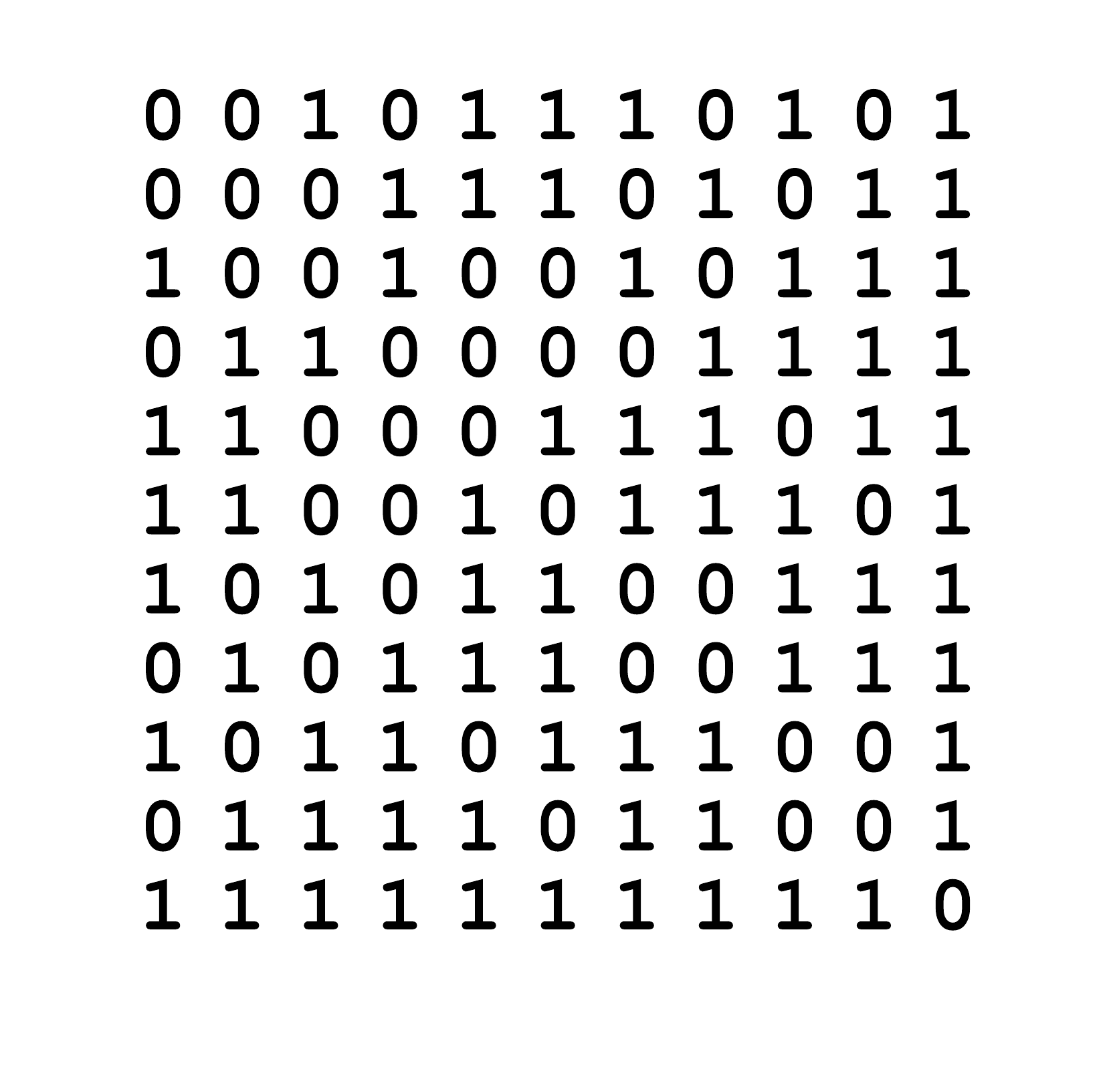}
		\vspace{-1.5em}
		\caption*{$G_{11.46}$}
		\label{figure: 11_46}
	\end{subfigure}\hfill
	\begin{subfigure}{.45\textwidth}
		\centering
		\includegraphics[trim={0 0 0 490},clip,height=110px,width=110px]{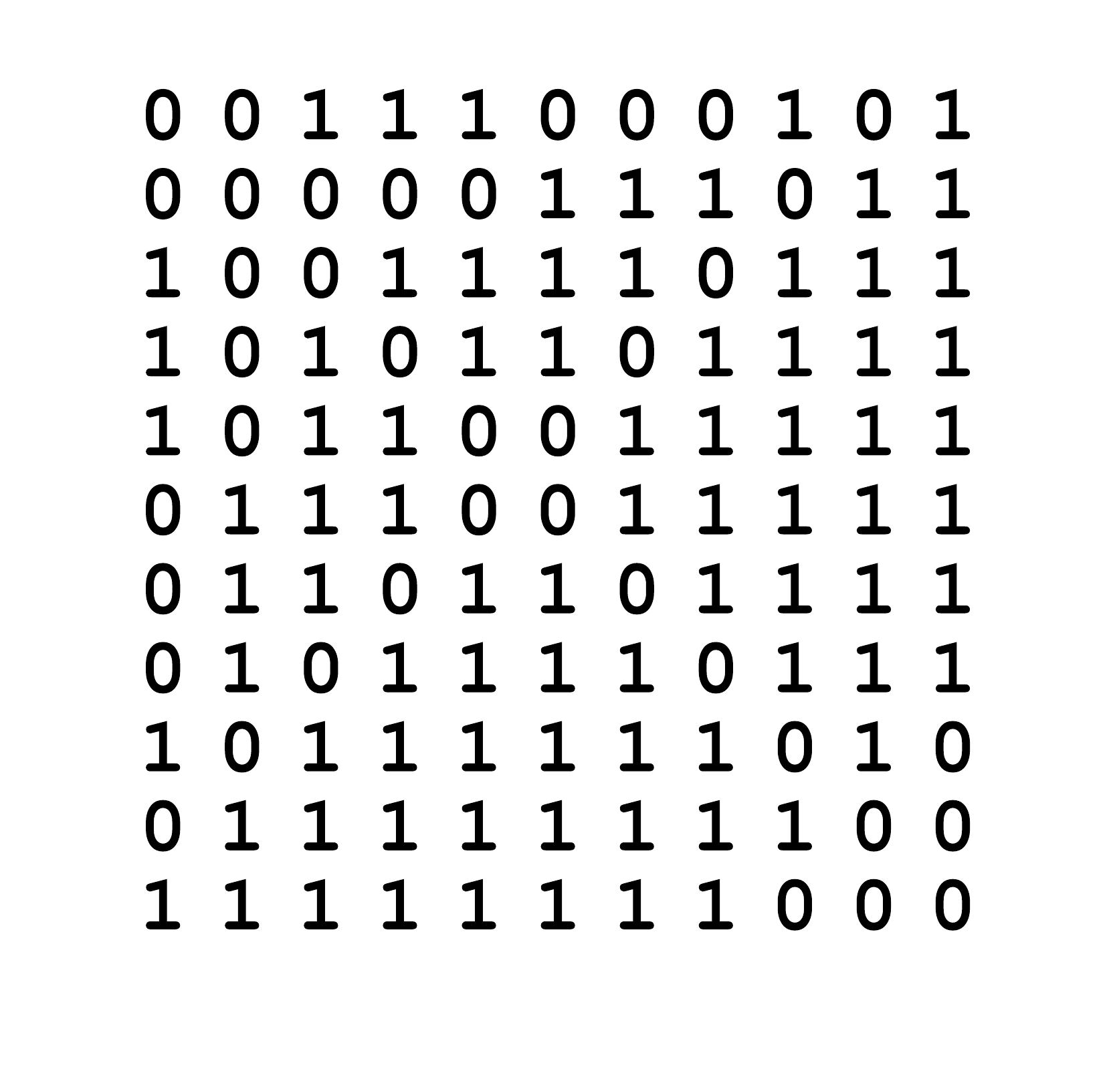}
		\vspace{-1.5em}
		\caption*{$G_{11.47}$}
		\label{figure: 11_47}
	\end{subfigure}
	
	\vspace{0.5em}
	\begin{subfigure}{.45\textwidth}
		\centering
		\includegraphics[trim={0 0 0 490},clip,height=110px,width=110px]{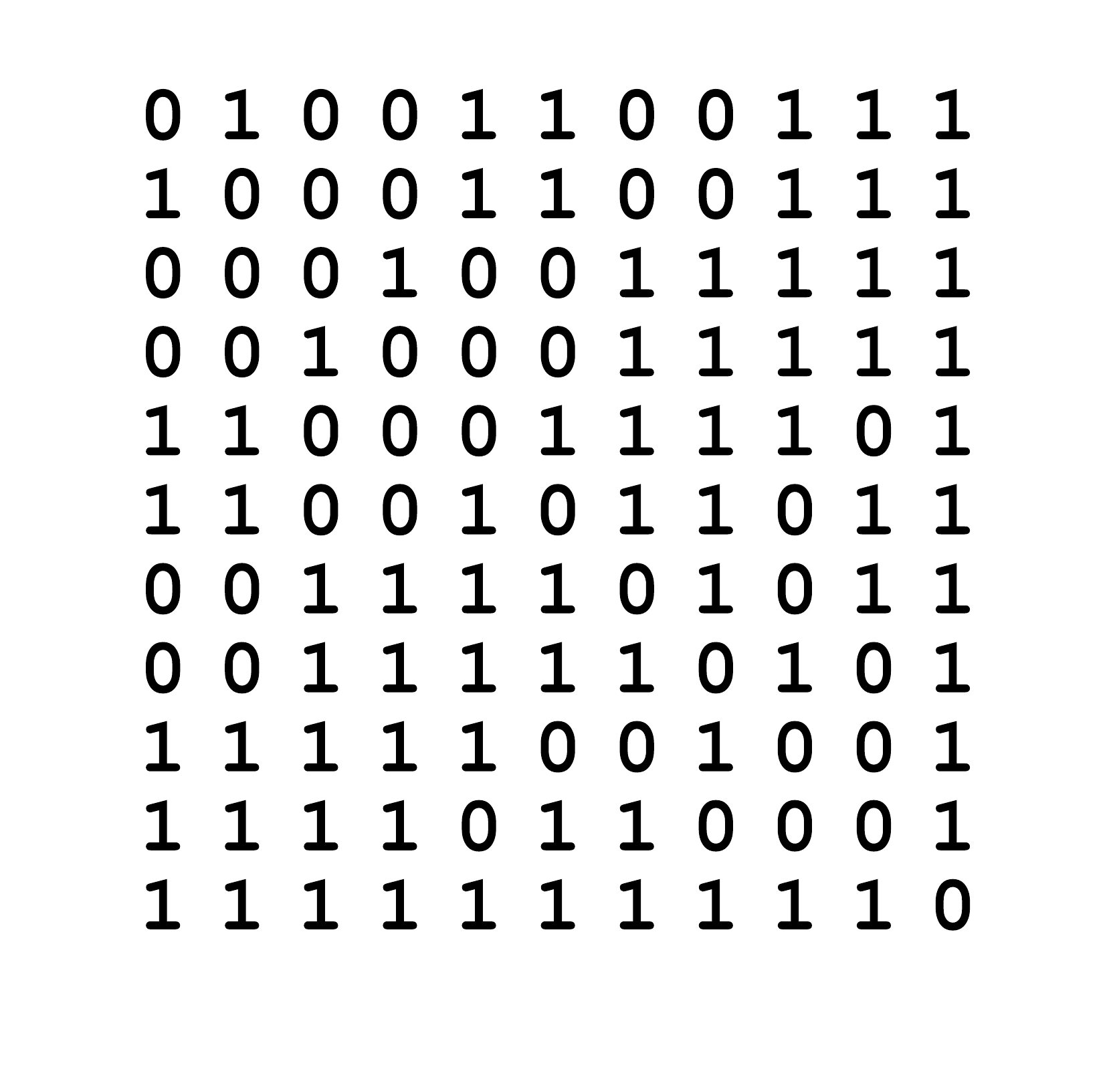}
		\vspace{-1.5em}
		\caption*{$G_{11.54}$}
		\label{figure: 11_54}
	\end{subfigure}\hfill
	\begin{subfigure}{.45\textwidth}
		\centering
		\includegraphics[trim={0 0 0 490},clip,height=110px,width=110px]{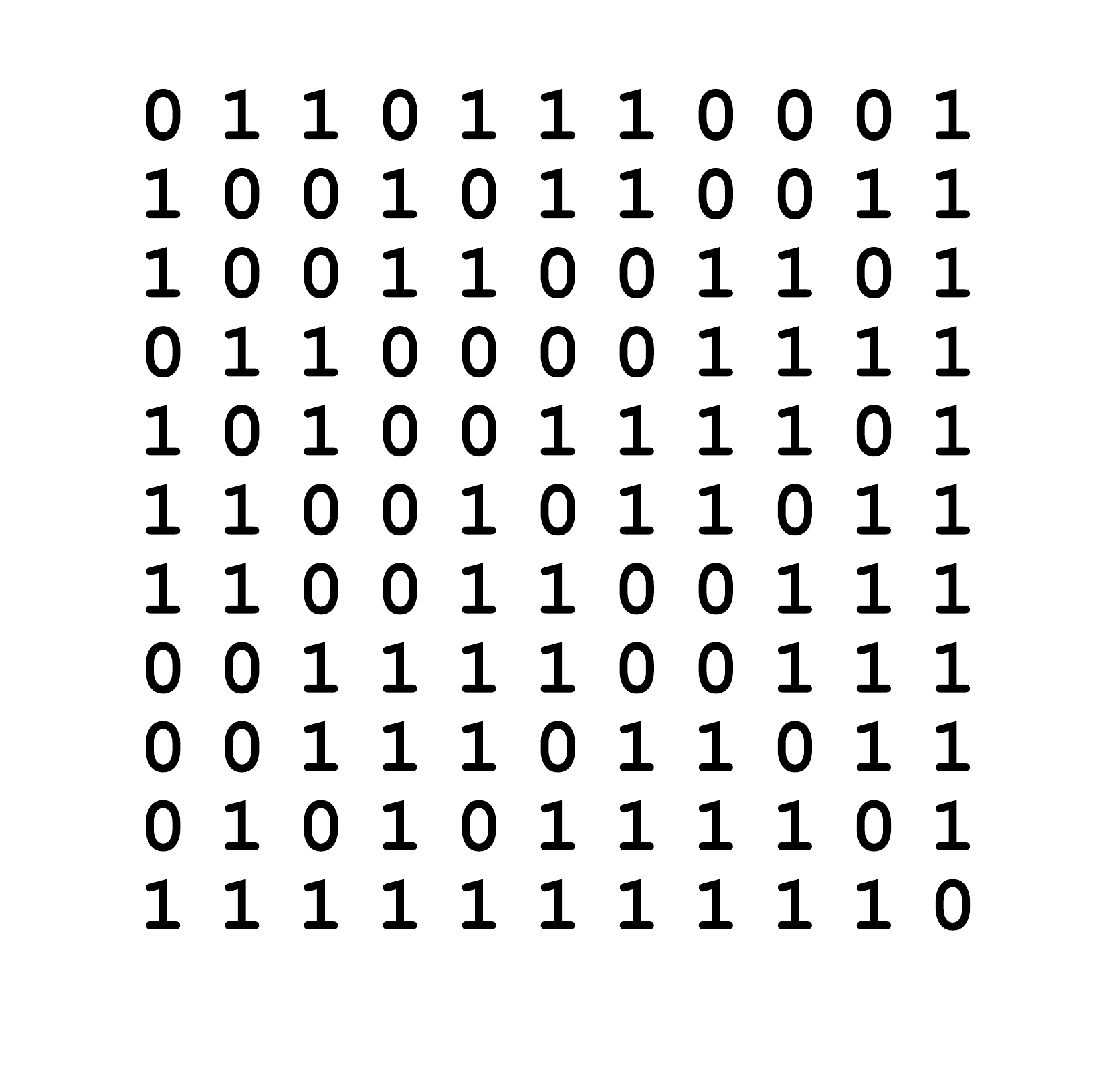}
		\vspace{-1.5em}
		\caption*{$G_{11.69}$}
		\label{figure: 11_69}
	\end{subfigure}
	
	\vspace{0.5em}
	\caption{11-vertex minimal $(3, 3)$-Ramsey graphs\\ with independence number 2}
	\label{figure: 11_a2}
\end{figure}

\begin{figure}[h]
	\captionsetup{justification=centering}
	\begin{minipage}{.45\textwidth}
		\centering
		\begin{subfigure}{\textwidth}
			\centering
			\includegraphics[trim={0 0 0 490},clip,height=120px,width=120px]{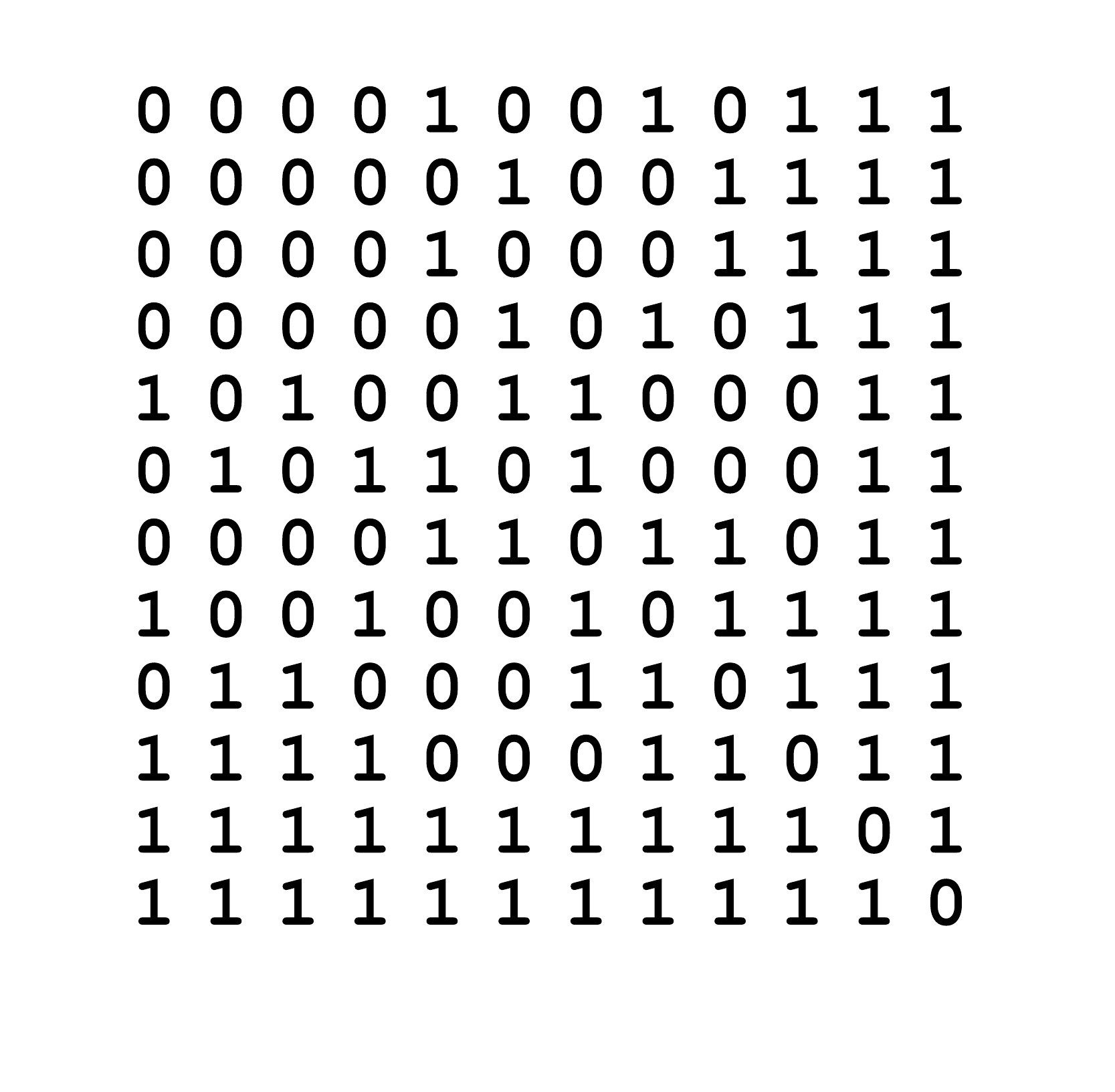}
			\vspace{-1.5em}
			\caption*{$G_{12.163}$}
			\label{figure: 12_163}
		\end{subfigure}
		
		\vspace{0.5em}
		\caption{12-vertex minimal $(3, 3)$-Ramsey graph\\ with independence number 5}
		\label{figure: 12_a5}
	\end{minipage}\hfill
	\begin{minipage}{.45\textwidth}
		\centering
		\begin{subfigure}{\textwidth}
			\centering
			\includegraphics[trim={0 0 0 490},clip,height=120px,width=120px]{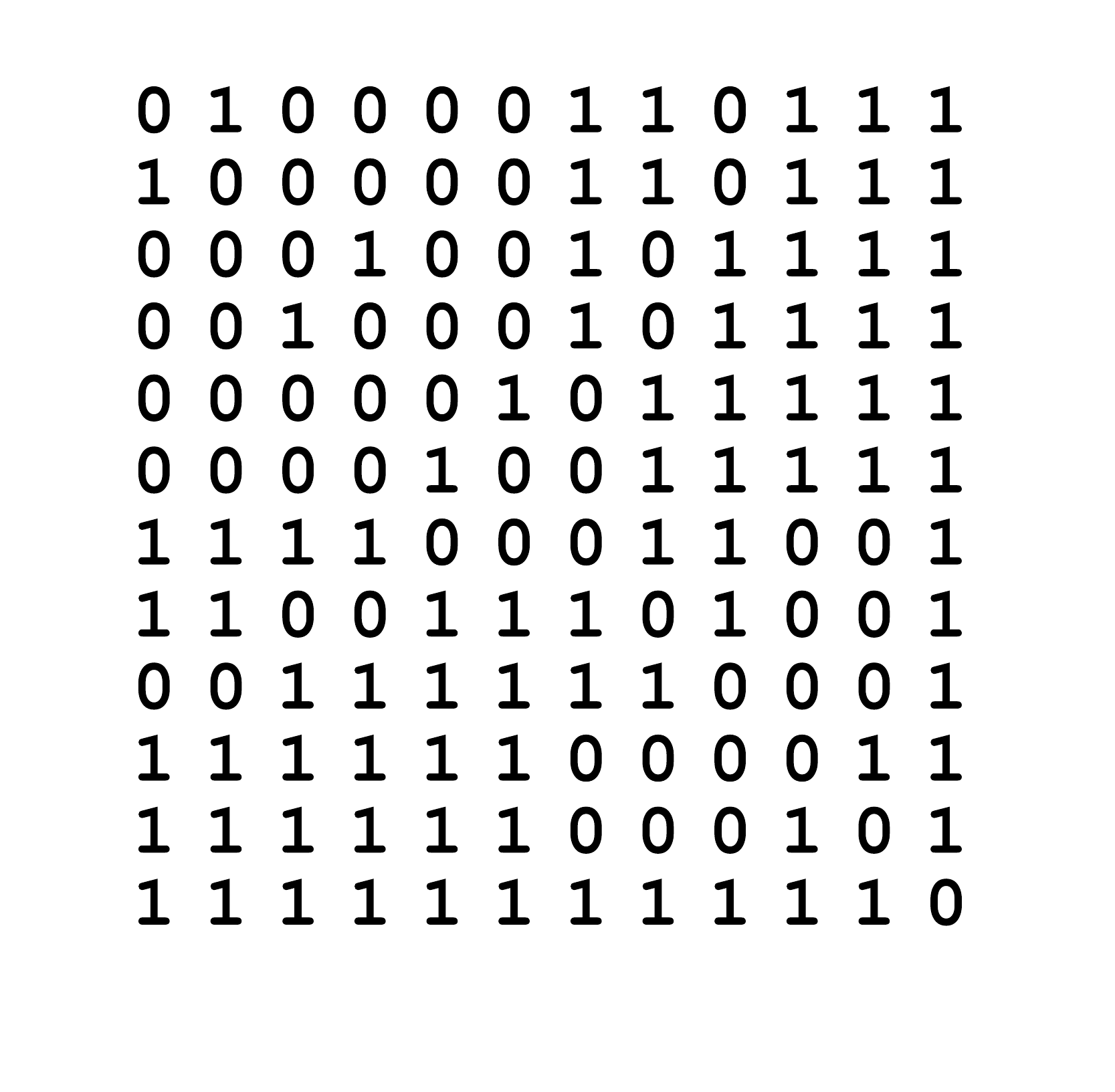}
			\vspace{-1.5em}
			\caption*{$G_{12.2240}$}
			\label{figure: 12_2240}
		\end{subfigure}
		
		\vspace{0.5em}
		\caption{12-vertex minimal $(3, 3)$-Ramsey graph\\ with 96 automorphisms}
		\label{figure: 12_aut}
	\end{minipage}\hfill
\end{figure}

\begin{figure}[h]
	\captionsetup{justification=centering}
	\begin{subfigure}{.45\textwidth}
		\centering
		\includegraphics[trim={0 0 0 490},clip,height=130px,width=130px]{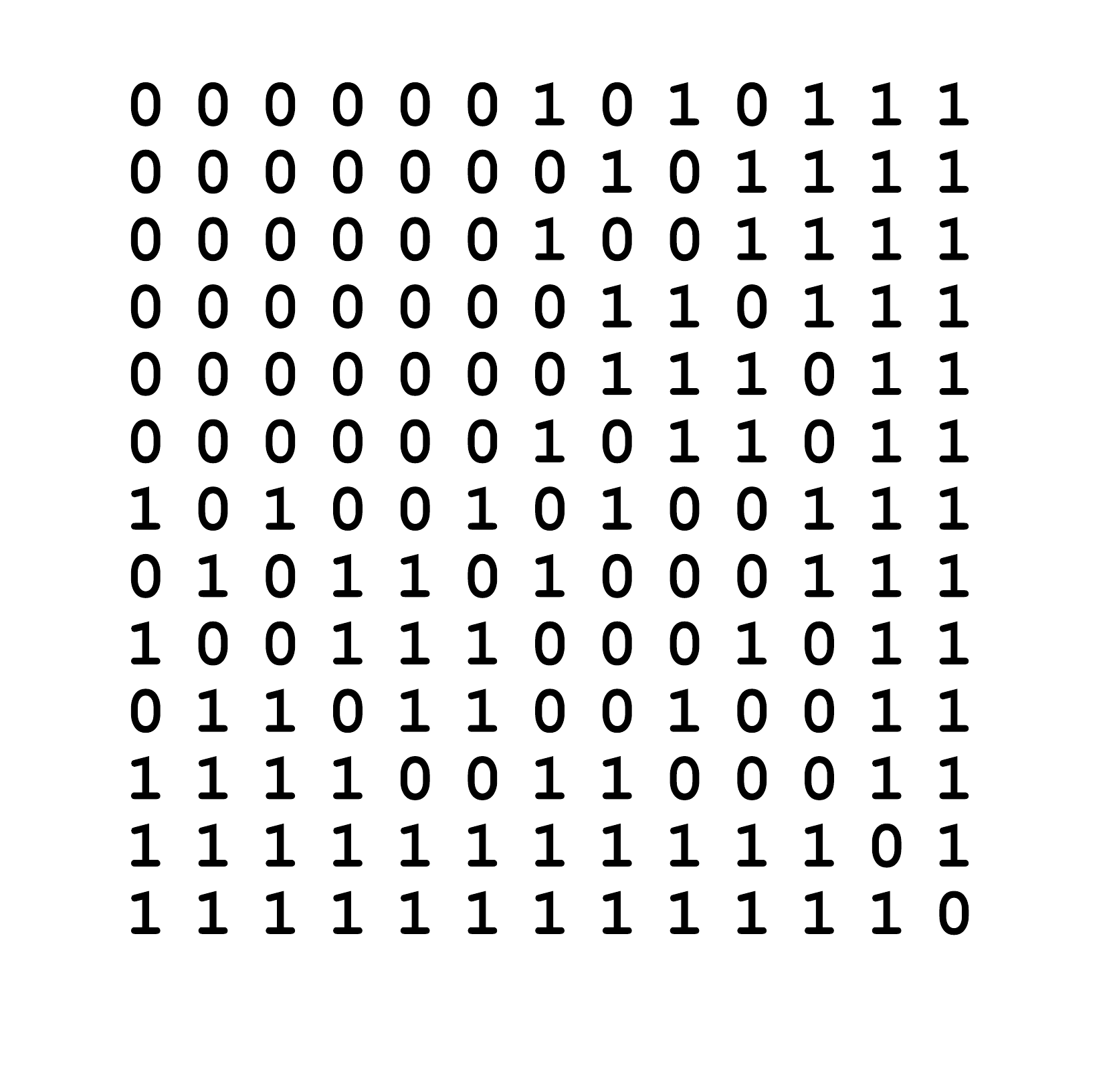}
		\vspace{-1.5em}
		\caption*{$G_{13.1}$}
		\label{figure: 13_1}
	\end{subfigure}\hfill
	\begin{subfigure}{.45\textwidth}
		\centering
		\includegraphics[trim={0 0 0 490},clip,height=130px,width=130px]{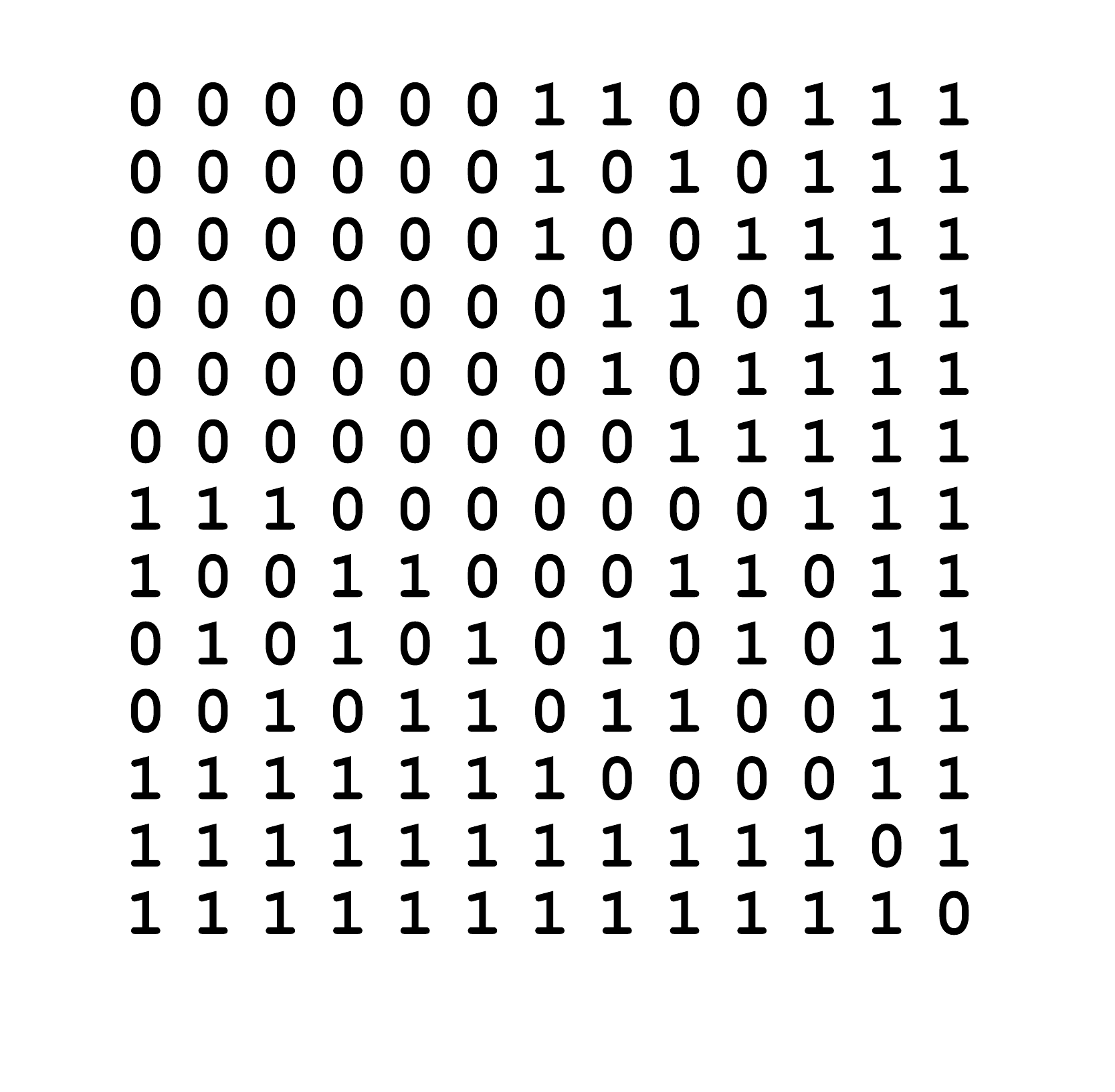}
		\vspace{-1.5em}
		\caption*{$G_{13.2}$}
		\label{figure: 13_2}
	\end{subfigure}
	
	\vspace{0.5em}
	\caption{13-vertex minimal $(3, 3)$-Ramsey graphs\\ with independence number 6}
	\label{figure: 13_a6}
\end{figure}

\begin{figure}[h]
	\captionsetup{justification=centering}
	\begin{subfigure}{.33\textwidth}
		\centering
		\includegraphics[trim={0 0 0 490},clip,height=130px,width=130px]{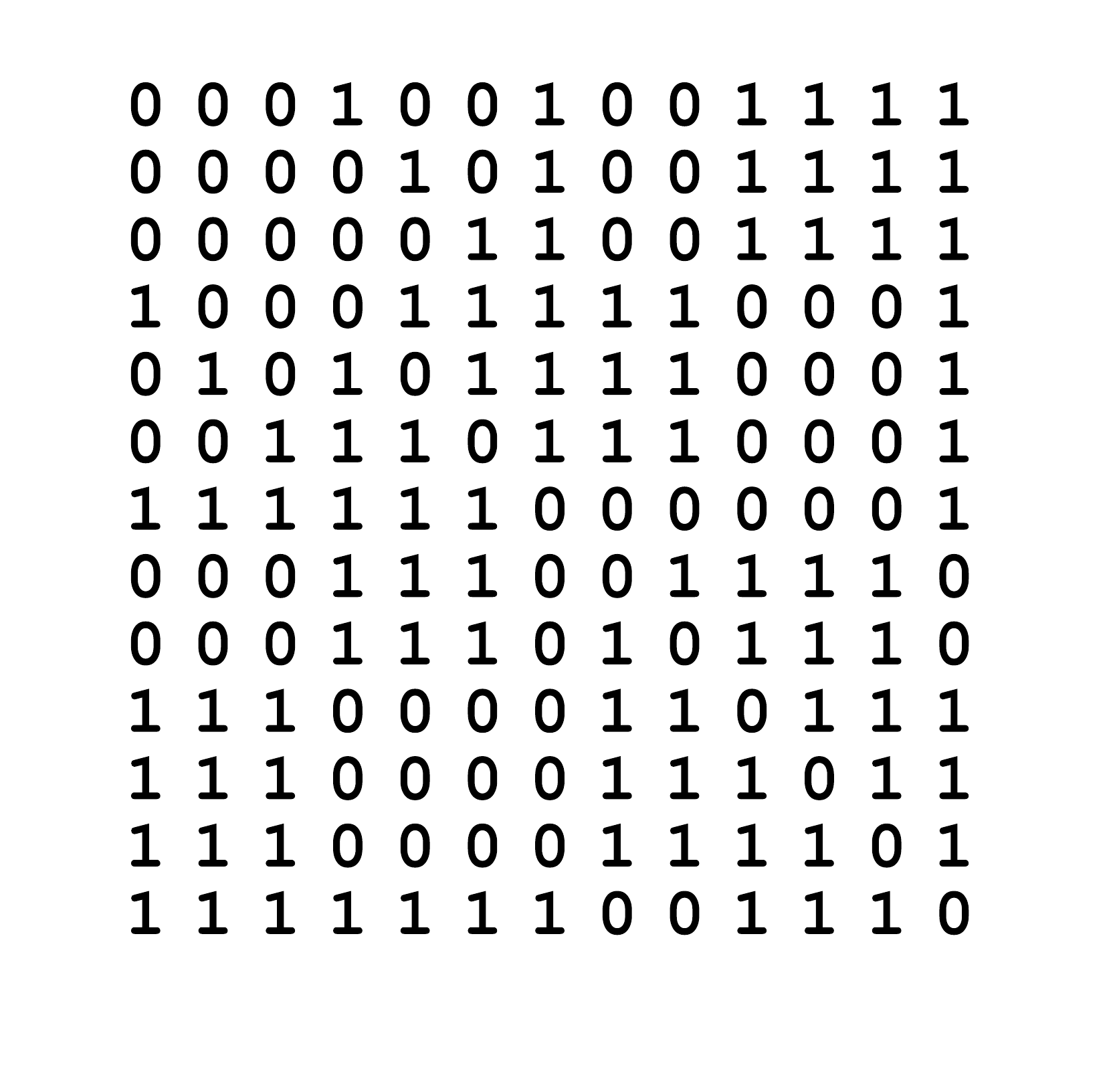}
		\vspace{-2em}
		\caption*{$G_{13.113198}$}
		\label{figure: 13_113198}
	\end{subfigure}\hfill
	\begin{subfigure}{.33\textwidth}
		\centering
		\includegraphics[trim={0 0 0 490},clip,height=130px,width=130px]{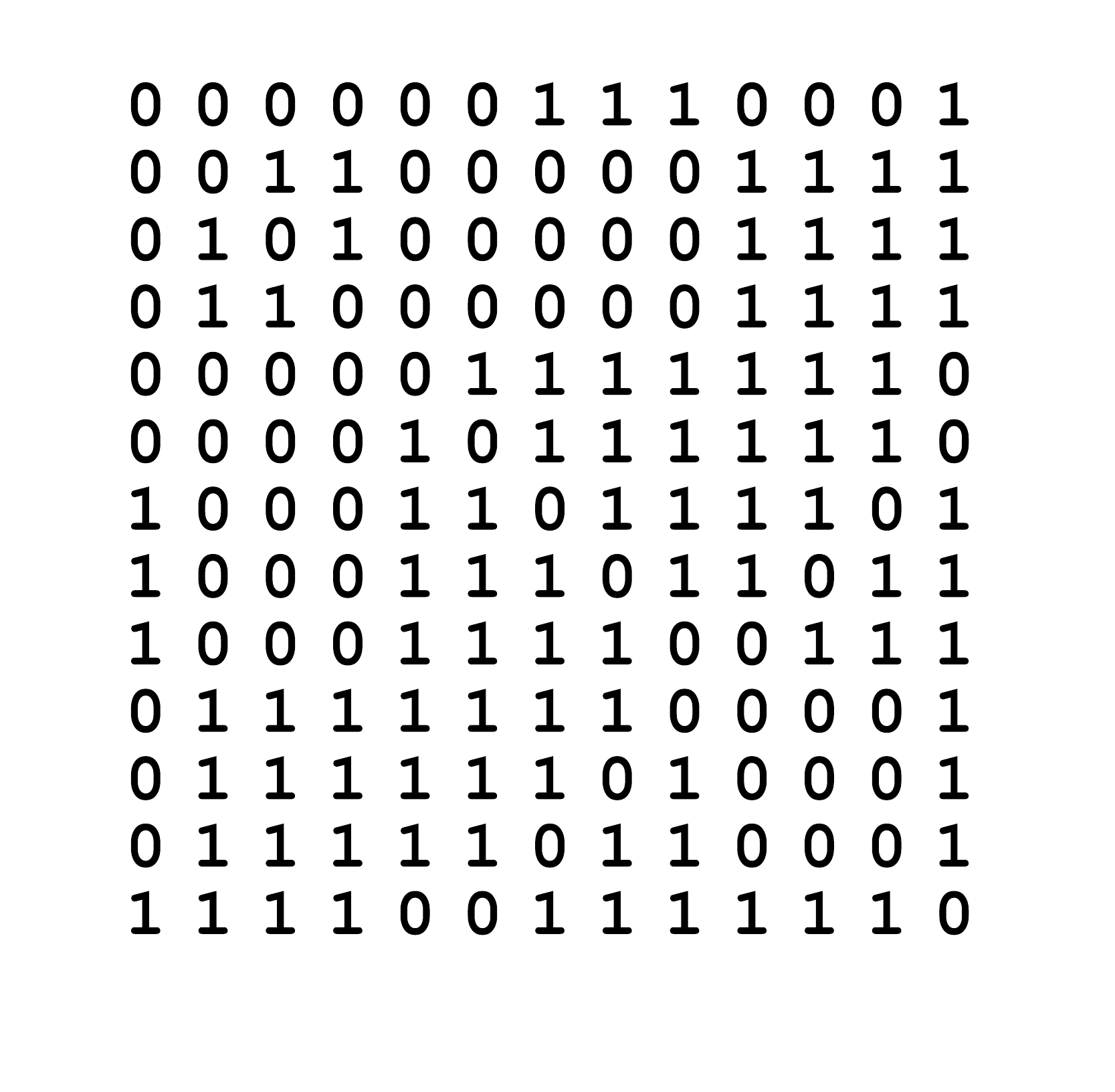}
		\vspace{-2em}
		\caption*{$G_{13.175639}$}
		\label{figure: 13_175639}
	\end{subfigure}\hfill
	\begin{subfigure}{.33\textwidth}
		\centering
		\includegraphics[trim={0 0 0 490},clip,height=130px,width=130px]{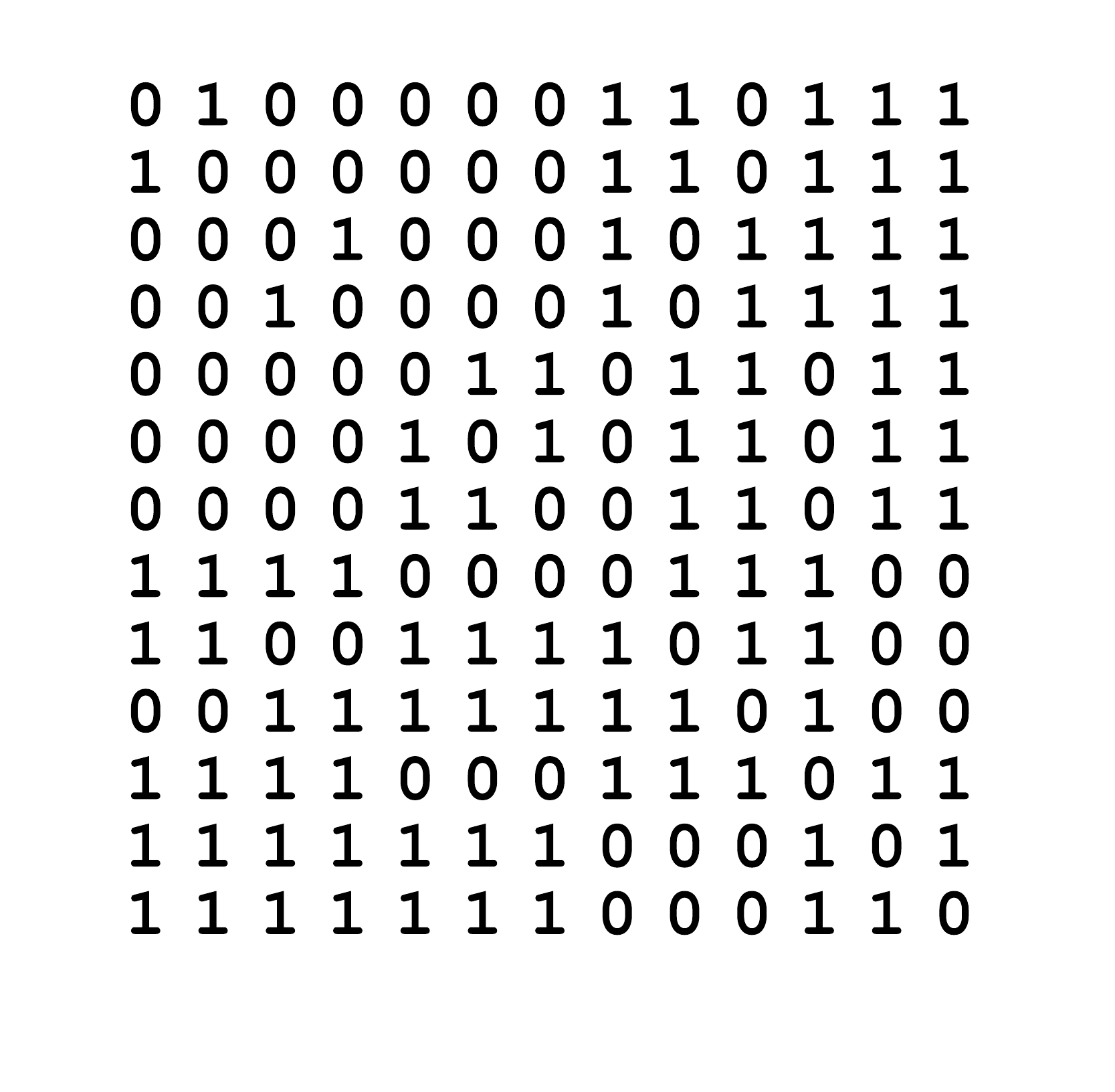}
		\vspace{-2em}
		\caption*{$G_{13.248305}$}
		\label{figure: 13_248305}
	\end{subfigure}
	
	\vspace{0.5em}
	\begin{subfigure}{.33\textwidth}
		\centering
		\includegraphics[trim={0 0 0 490},clip,height=130px,width=130px]{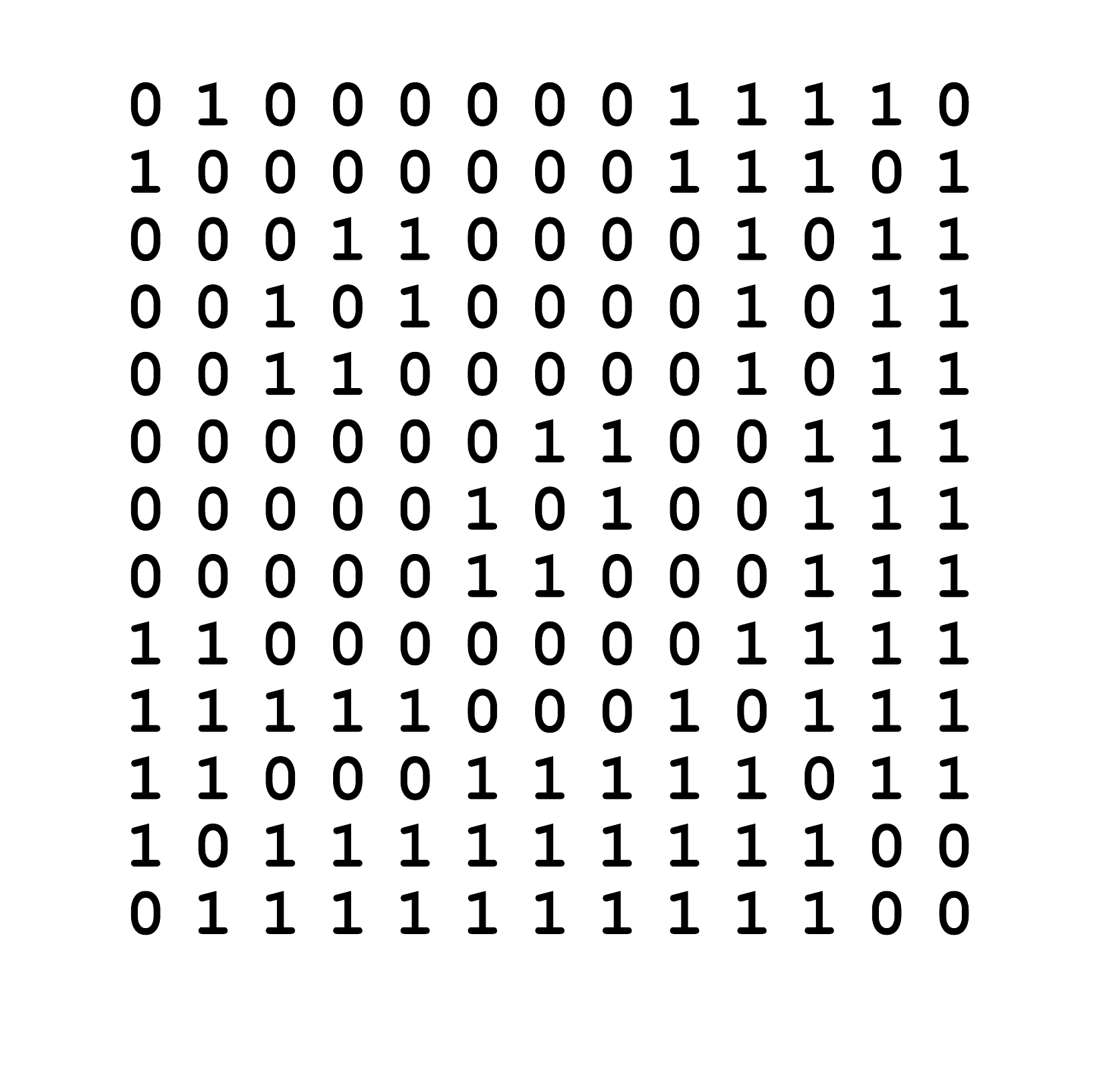}
		\vspace{-2em}
		\caption*{$G_{13.255653}$}
		\label{figure: 13_255653}
	\end{subfigure}\hfill
	\begin{subfigure}{.33\textwidth}
		\centering
		\includegraphics[trim={0 0 0 490},clip,height=130px,width=130px]{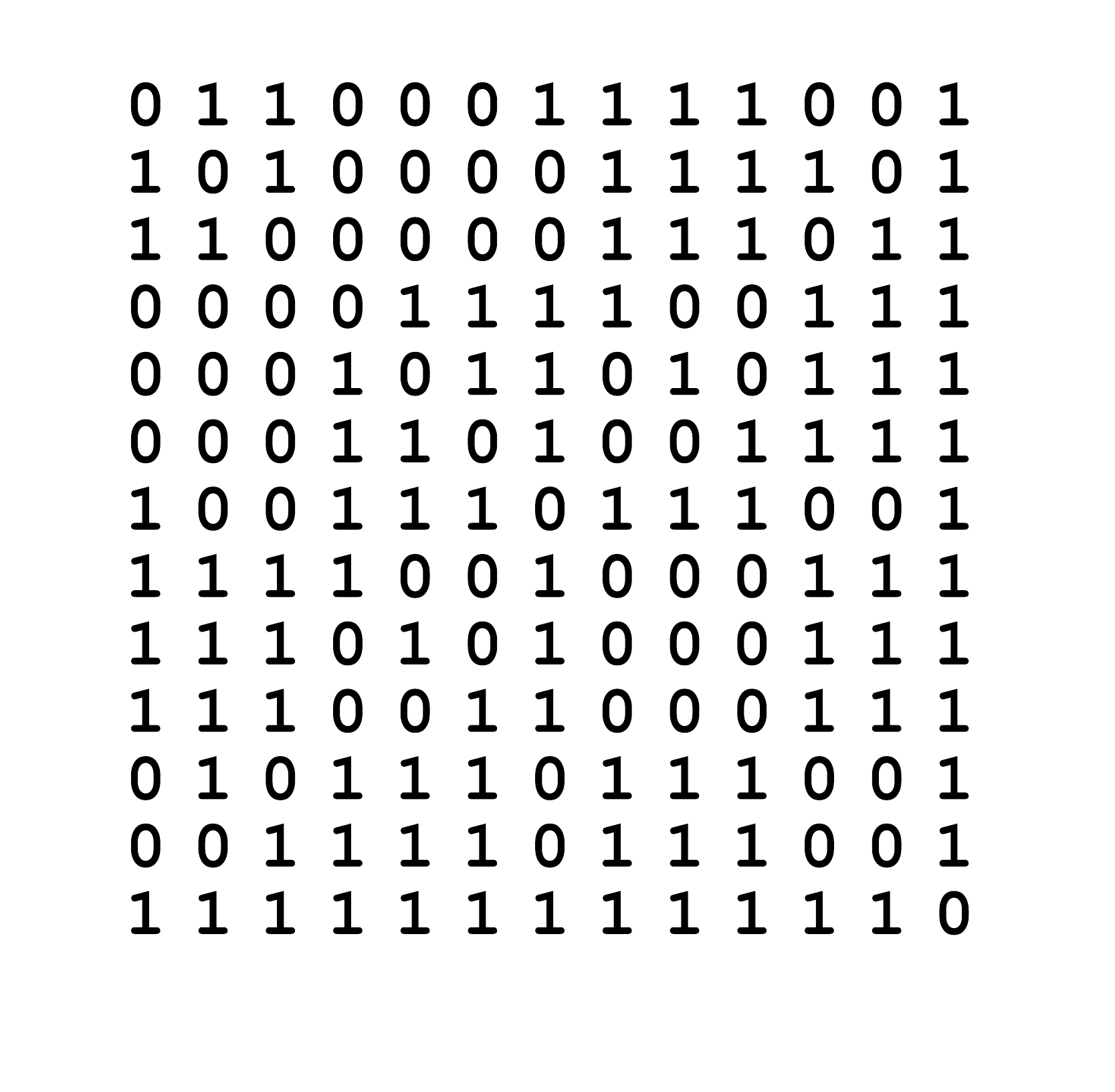}
		\vspace{-2em}
		\caption*{$G_{13.302168}$}
		\label{figure: 13_302168}
	\end{subfigure}\hfill
	\begin{subfigure}{.33\textwidth}
		\centering
		\includegraphics[trim={0 0 0 490},clip,height=130px,width=130px]{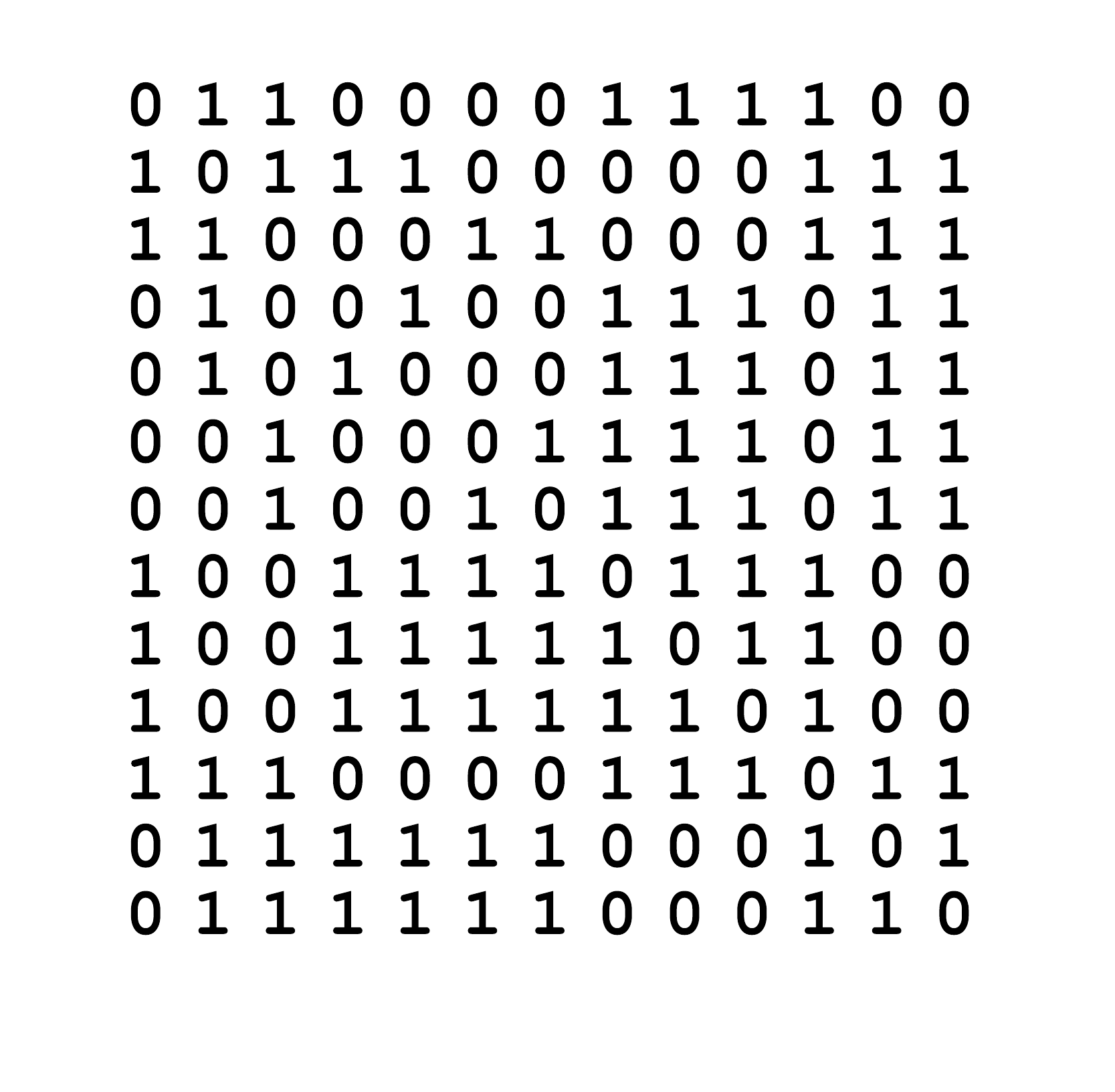}
		\vspace{-2em}
		\caption*{$G_{13.304826}$}
		\label{figure: 13_304826}
	\end{subfigure}
	
	\vspace{0.5em}
	\caption{13-vertex minimal $(3, 3)$-Ramsey graphs\\ with a large number of automorphisms}
	\label{figure: 13_aut}
\end{figure}

\begin{figure}[h]
	\captionsetup{justification=centering}
	\begin{subfigure}{.33\textwidth}
		\centering
		\includegraphics[trim={0 0 0 490},clip,height=100px,width=100px]{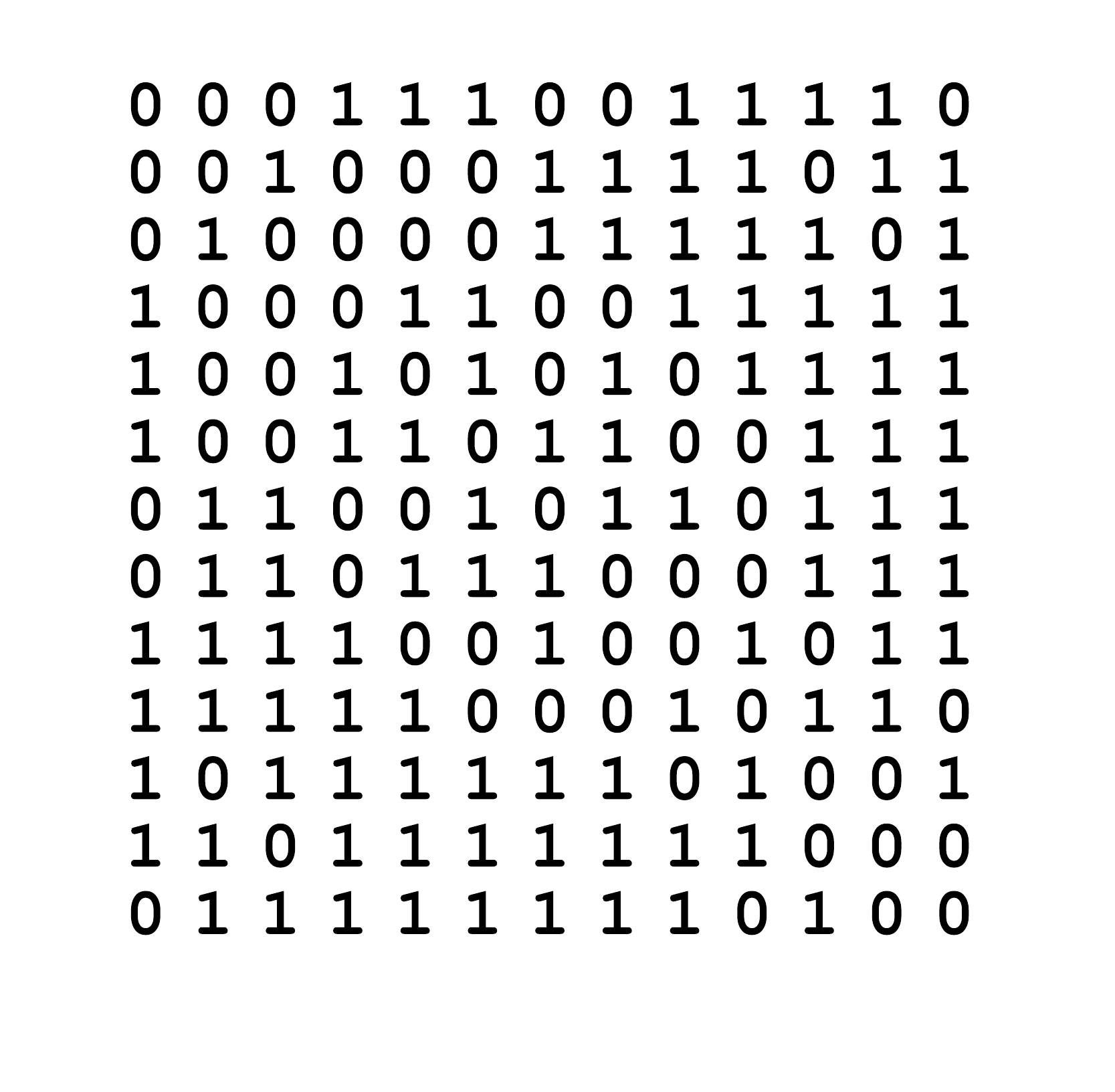}
		\vspace{-1.5em}
		\caption*{$G_{13.193684}$}
		\label{figure: 13_193684}
	\end{subfigure}\hfill
	\begin{subfigure}{.33\textwidth}
		\centering
		\includegraphics[trim={0 0 0 490},clip,height=100px,width=100px]{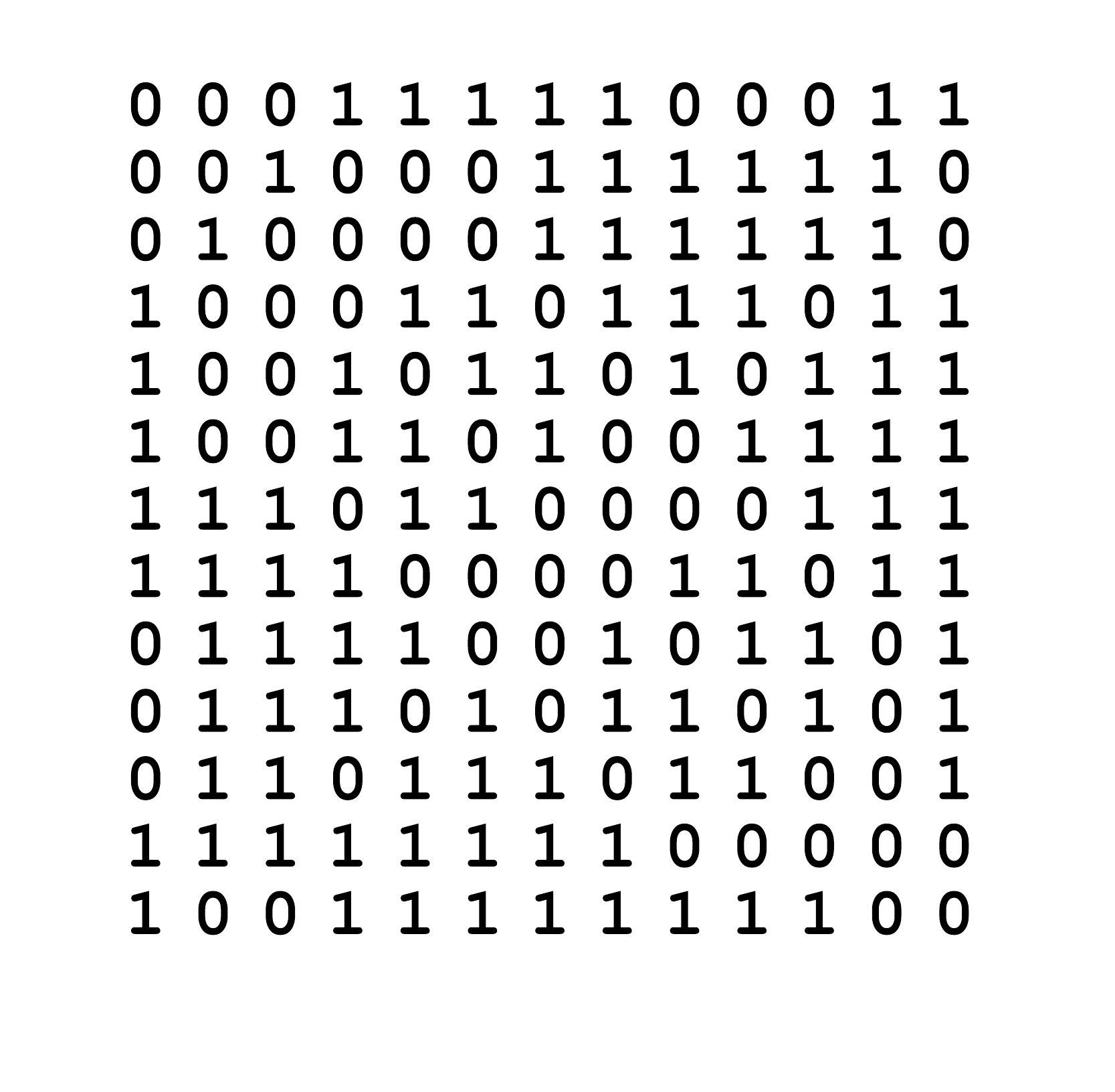}
		\vspace{-1.5em}
		\caption*{$G_{13.193760}$}
		\label{figure: 13_193760}
	\end{subfigure}\hfill
	\begin{subfigure}{.33\textwidth}
		\centering
		\includegraphics[trim={0 0 0 490},clip,height=100px,width=100px]{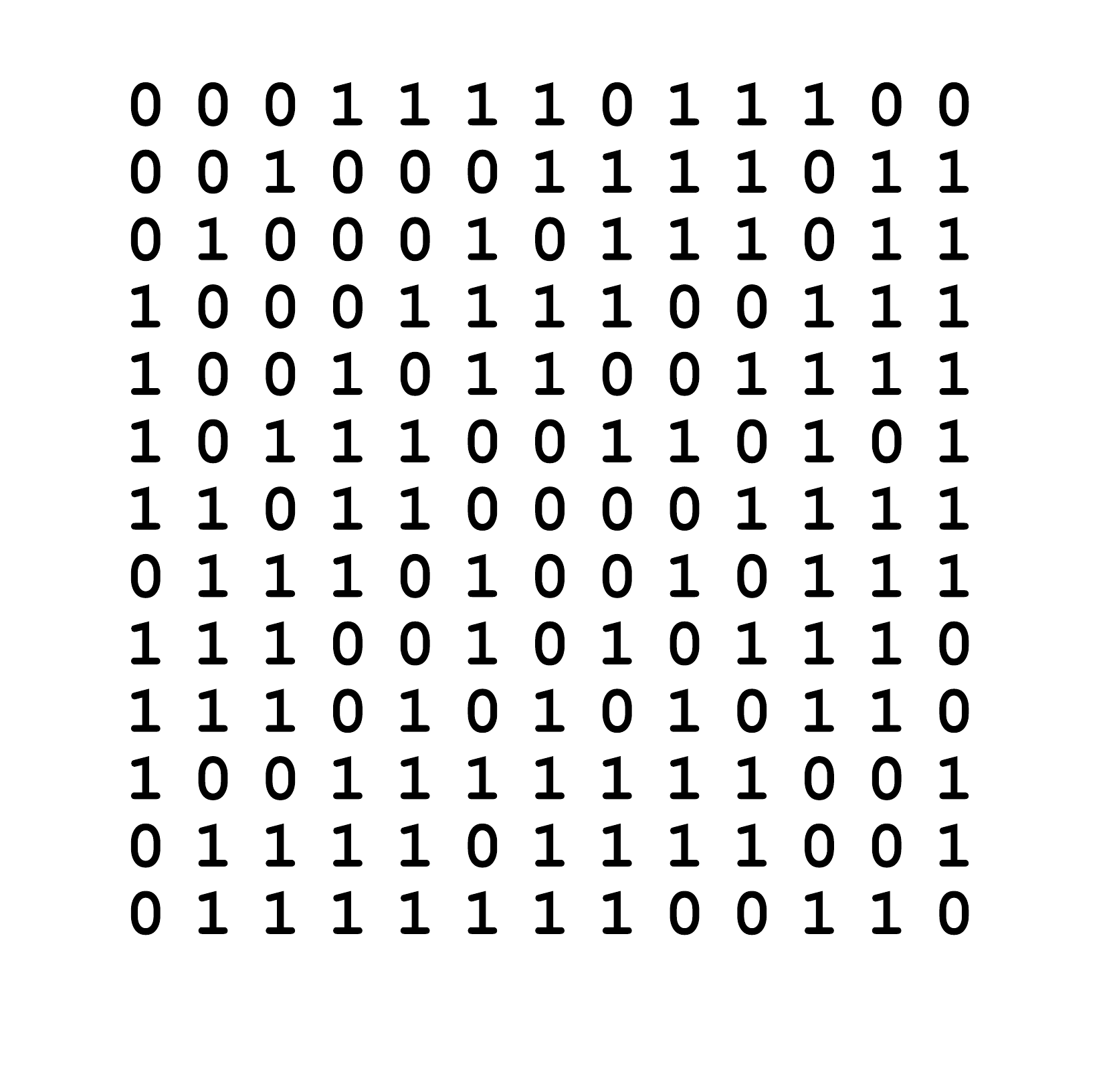}
		\vspace{-1.5em}
		\caption*{$G_{13.193988}$}
		\label{figure: 13_193988}
	\end{subfigure}
	
	\begin{subfigure}{.33\textwidth}
		\centering
		\includegraphics[trim={0 0 0 490},clip,height=100px,width=100px]{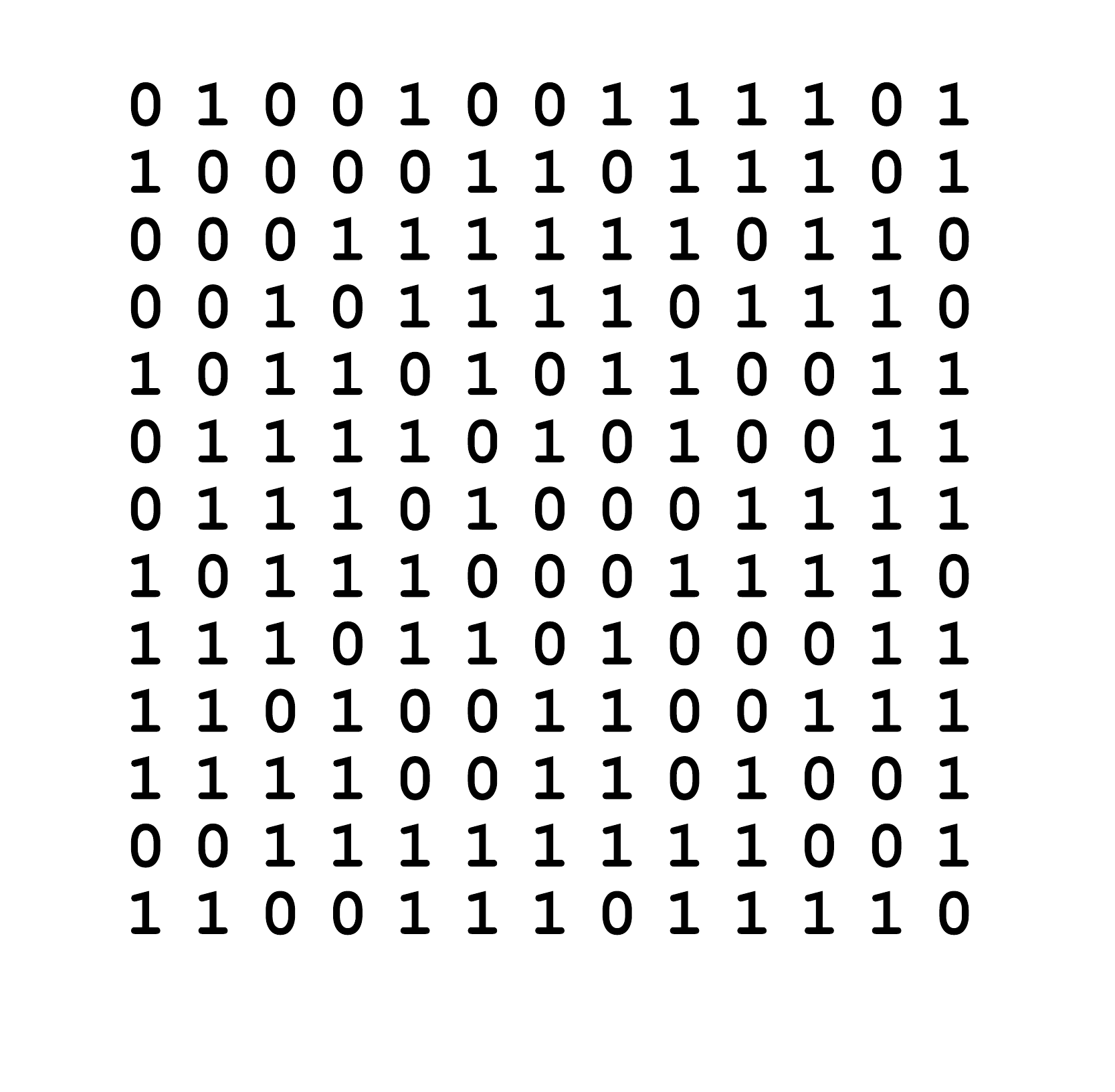}
		\vspace{-1.5em}
		\caption*{$G_{13.265221}$}
		\label{figure: 13_265221}
	\end{subfigure}\hfill
	\begin{subfigure}{.33\textwidth}
		\centering
		\includegraphics[trim={0 0 0 490},clip,height=100px,width=100px]{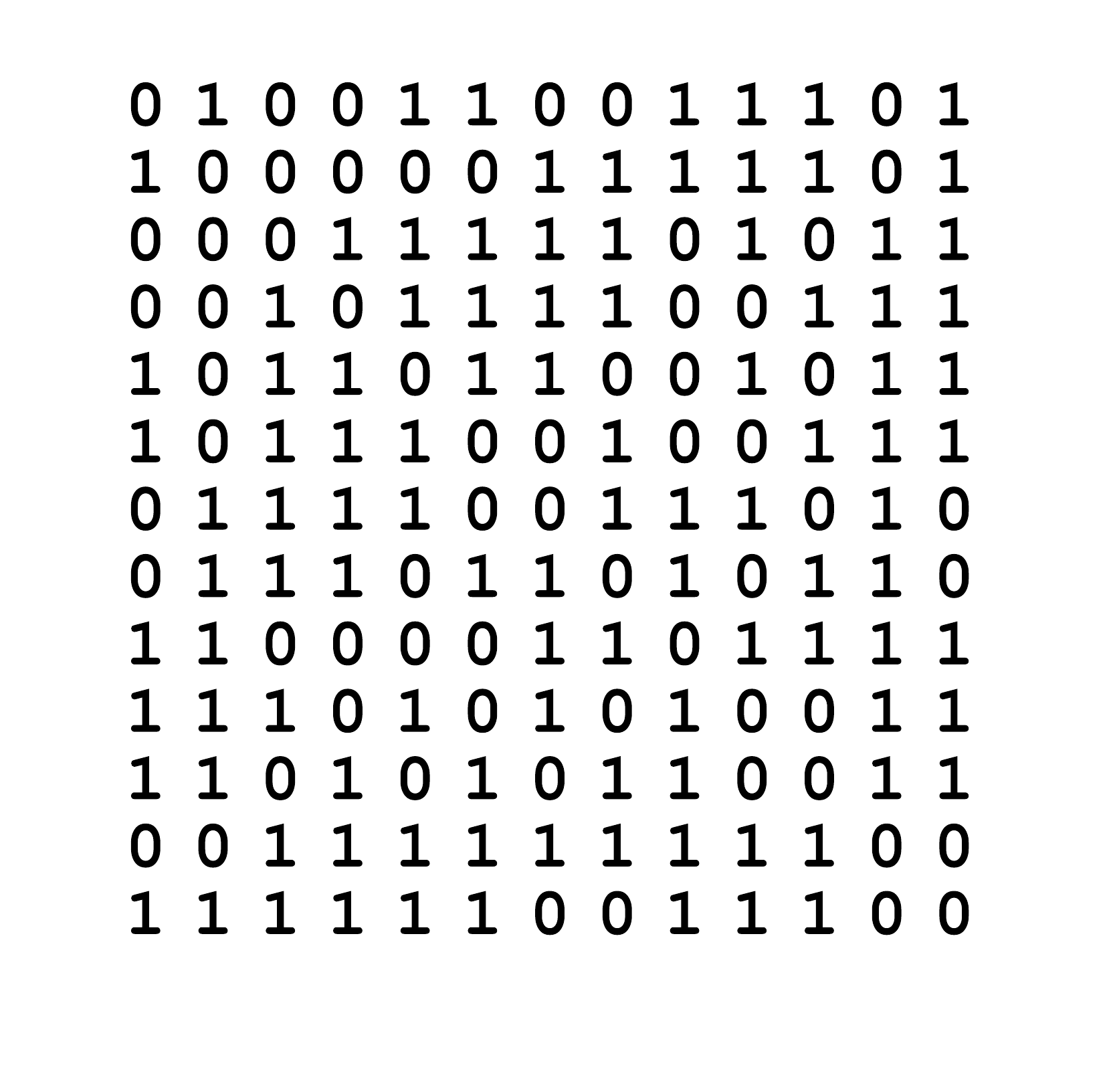}
		\vspace{-1.5em}
		\caption*{$G_{13.265299}$}
		\label{figure: 13_265299}
	\end{subfigure}\hfill
	\begin{subfigure}{.33\textwidth}
		\centering
		\includegraphics[trim={0 0 0 490},clip,height=100px,width=100px]{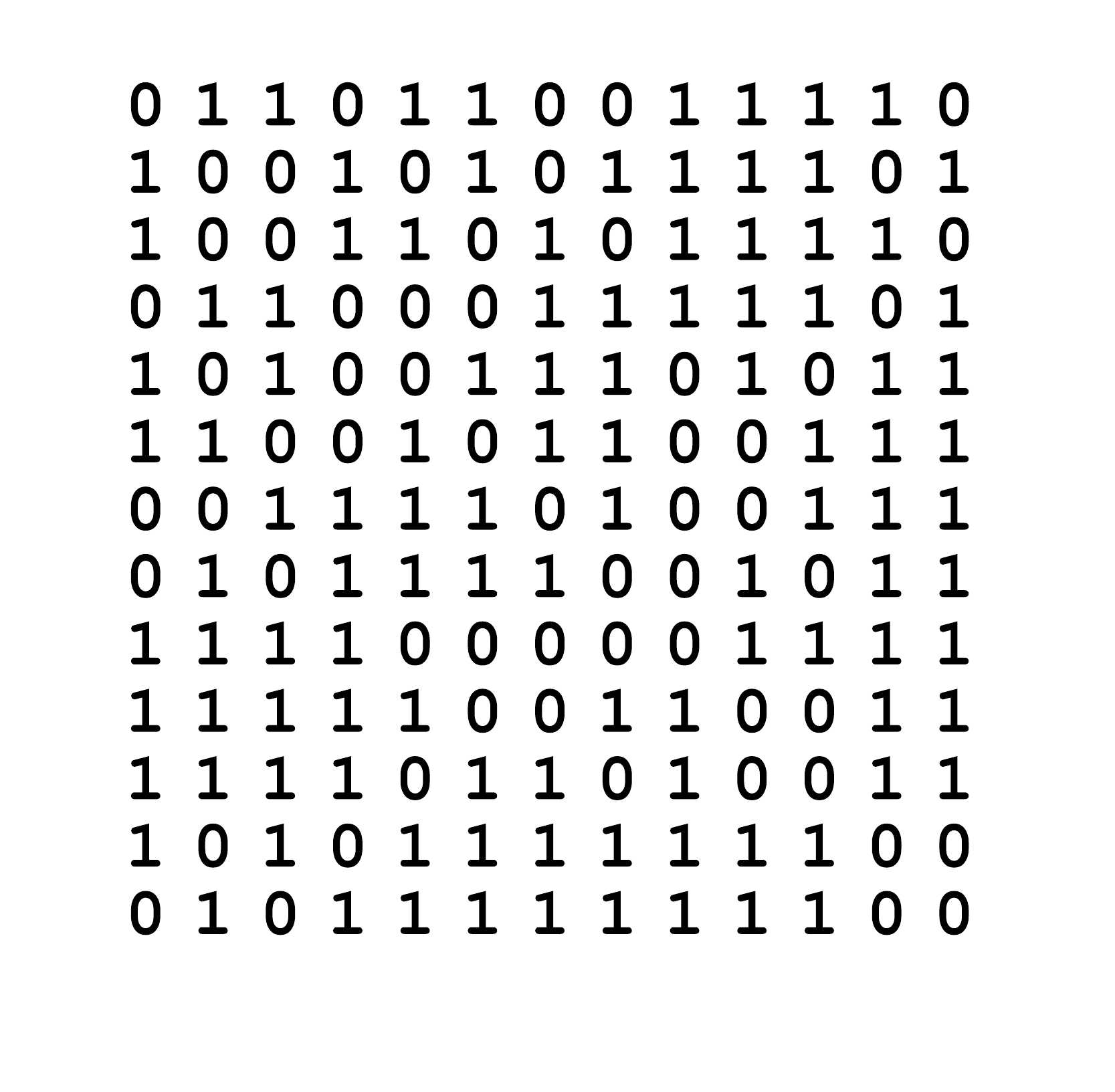}
		\vspace{-1.5em}
		\caption*{$G_{13.299797}$}
		\label{figure: 13_299797}
	\end{subfigure}
	
	\begin{subfigure}{.33\textwidth}
		\centering
		\includegraphics[trim={0 0 0 490},clip,height=100px,width=100px]{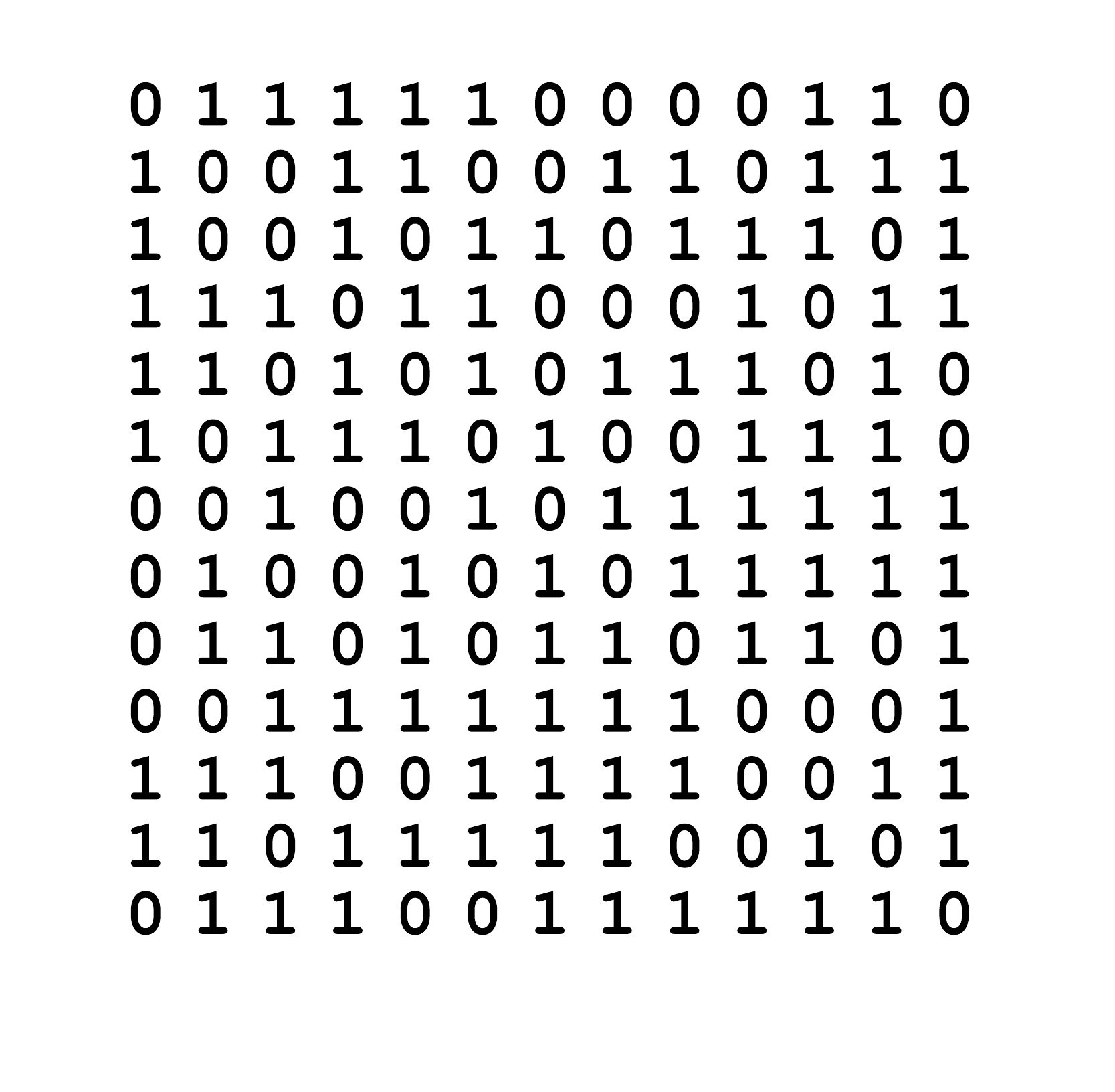}
		\vspace{-1.5em}
		\caption*{$G_{13.301368}$}
		\label{figure: 13_301368}
	\end{subfigure}\hfill
	\begin{subfigure}{.33\textwidth}
		\centering
		\includegraphics[trim={0 0 0 490},clip,height=100px,width=100px]{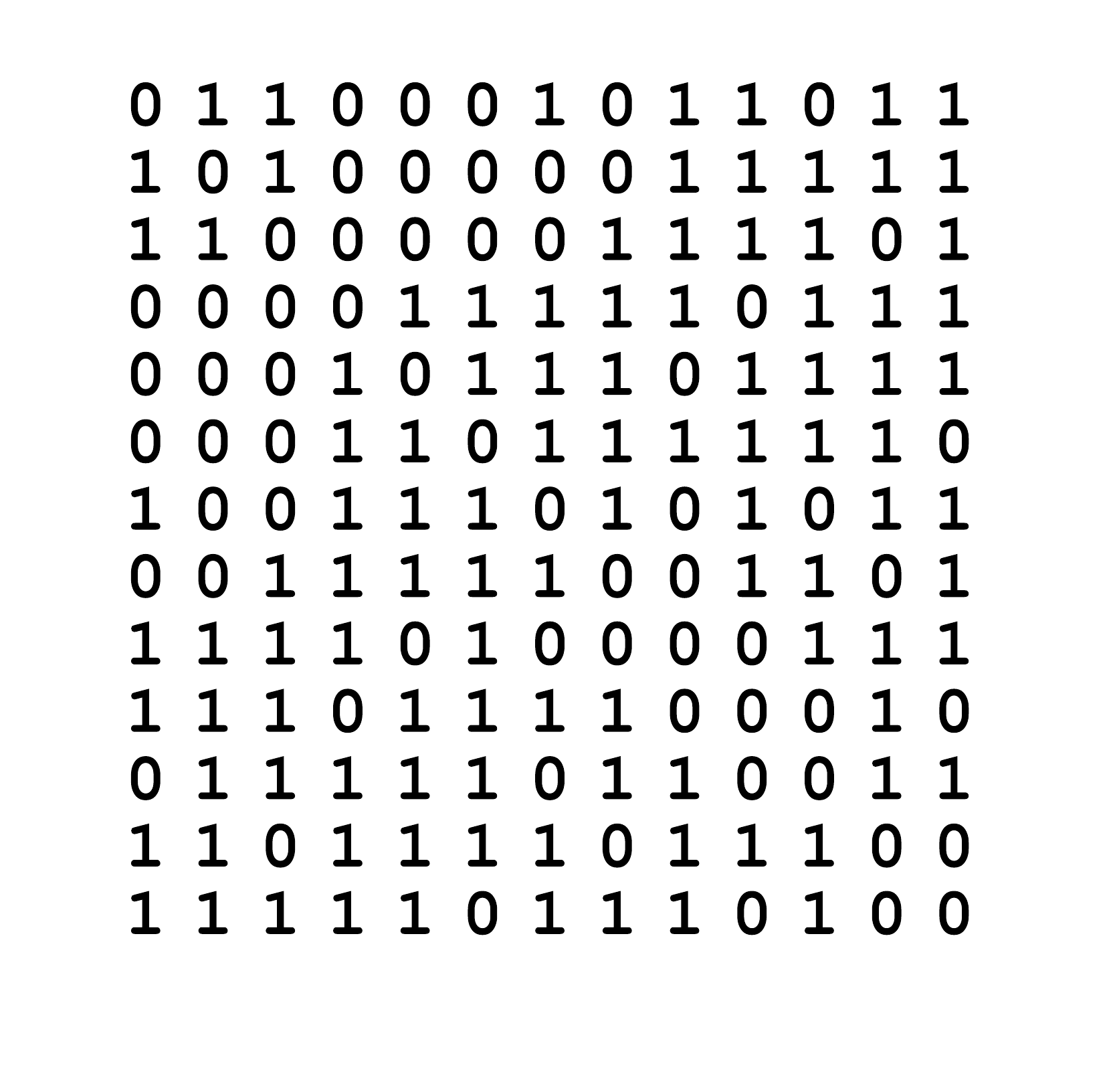}
		\vspace{-1.5em}
		\caption*{$G_{13.302151}$}
		\label{figure: 13_302151}
	\end{subfigure}\hfill
	\begin{subfigure}{.33\textwidth}
		\centering
		\includegraphics[trim={0 0 0 490},clip,height=100px,width=100px]{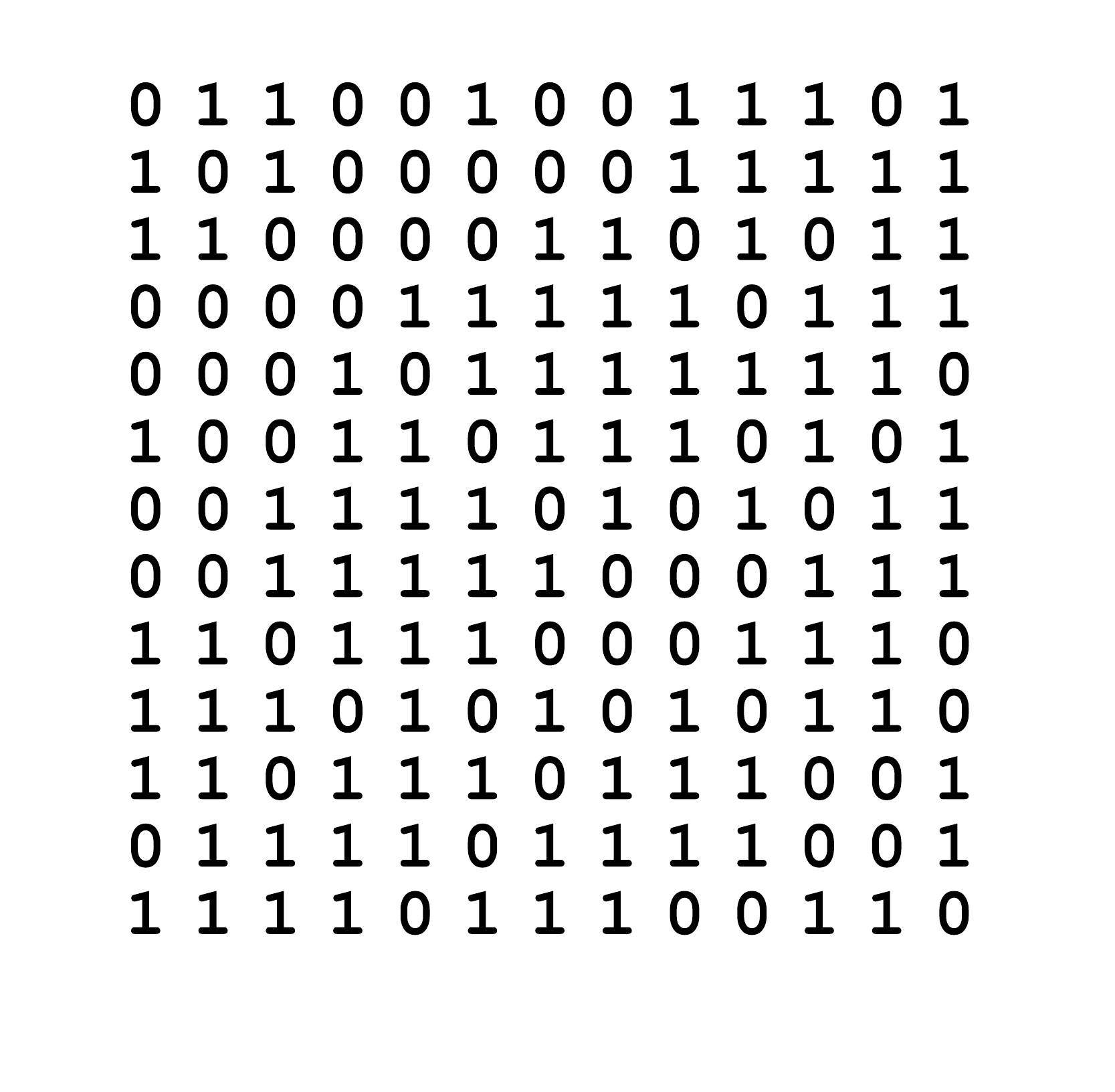}
		\vspace{-1.5em}
		\caption*{$G_{13.302764}$}
		\label{figure: 13_302764}
	\end{subfigure}
	
	\begin{subfigure}{.33\textwidth}
		\centering
		\includegraphics[trim={0 0 0 490},clip,height=100px,width=100px]{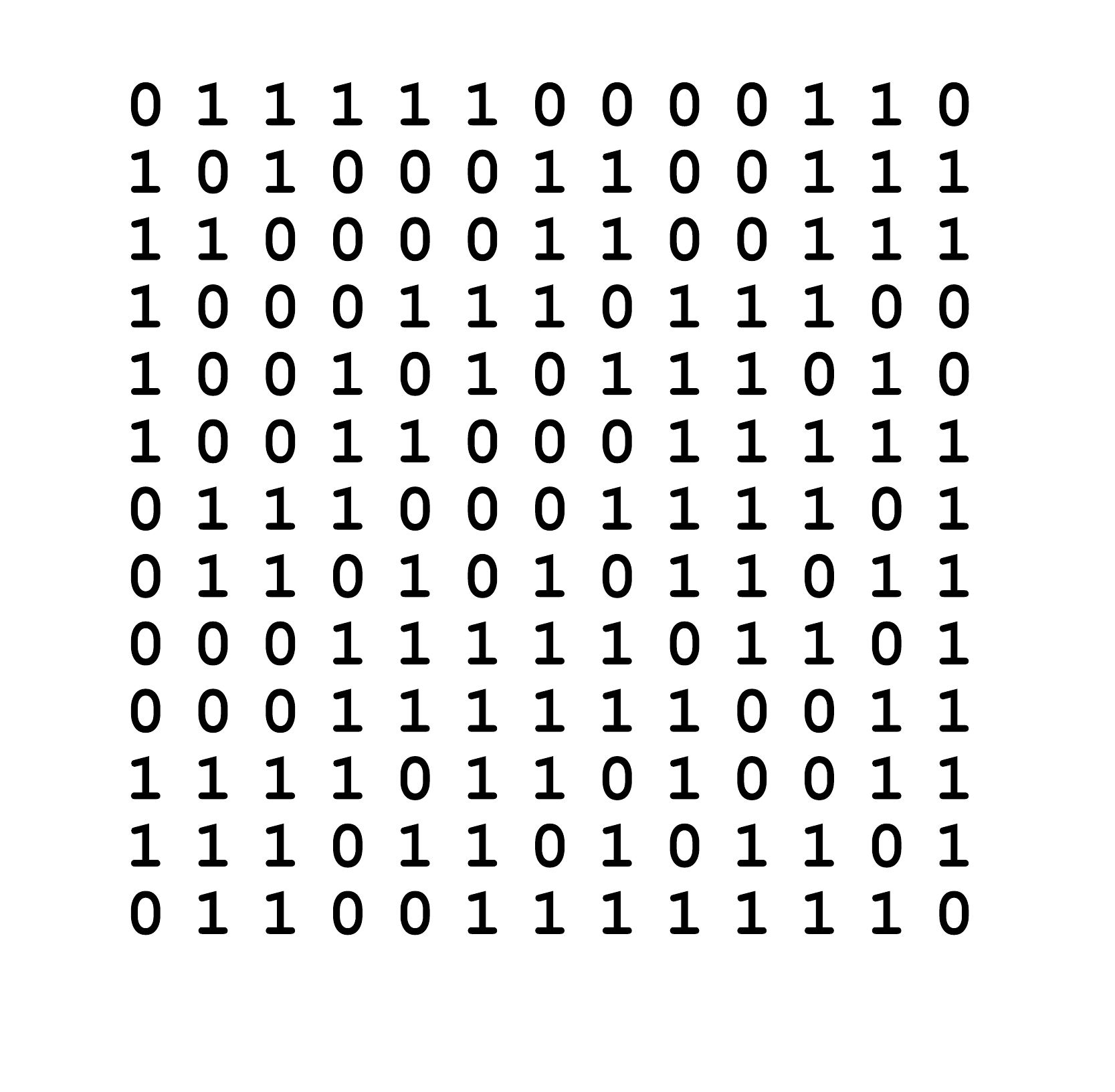}
		\vspace{-1.5em}
		\caption*{$G_{13.305857}$}
		\label{figure: 13_305857}
	\end{subfigure}\hfill
	\begin{subfigure}{.33\textwidth}
		\centering
		\includegraphics[trim={0 0 0 490},clip,height=100px,width=100px]{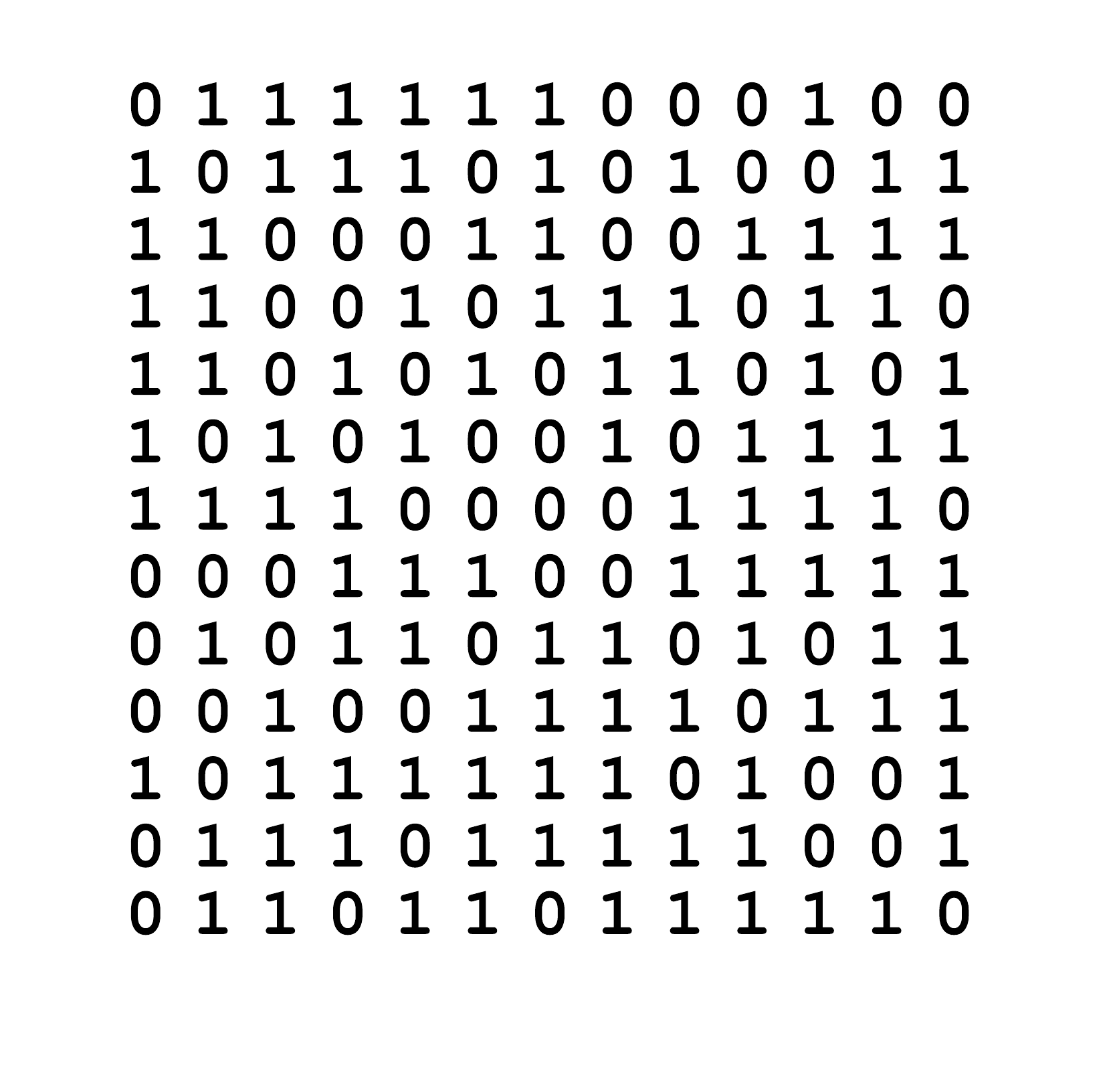}
		\vspace{-1.5em}
		\caption*{$G_{13.306448}$}
		\label{figure: 13_306448}
	\end{subfigure}\hfill
	\begin{subfigure}{.33\textwidth}
		\centering
		\includegraphics[trim={0 0 0 490},clip,height=100px,width=100px]{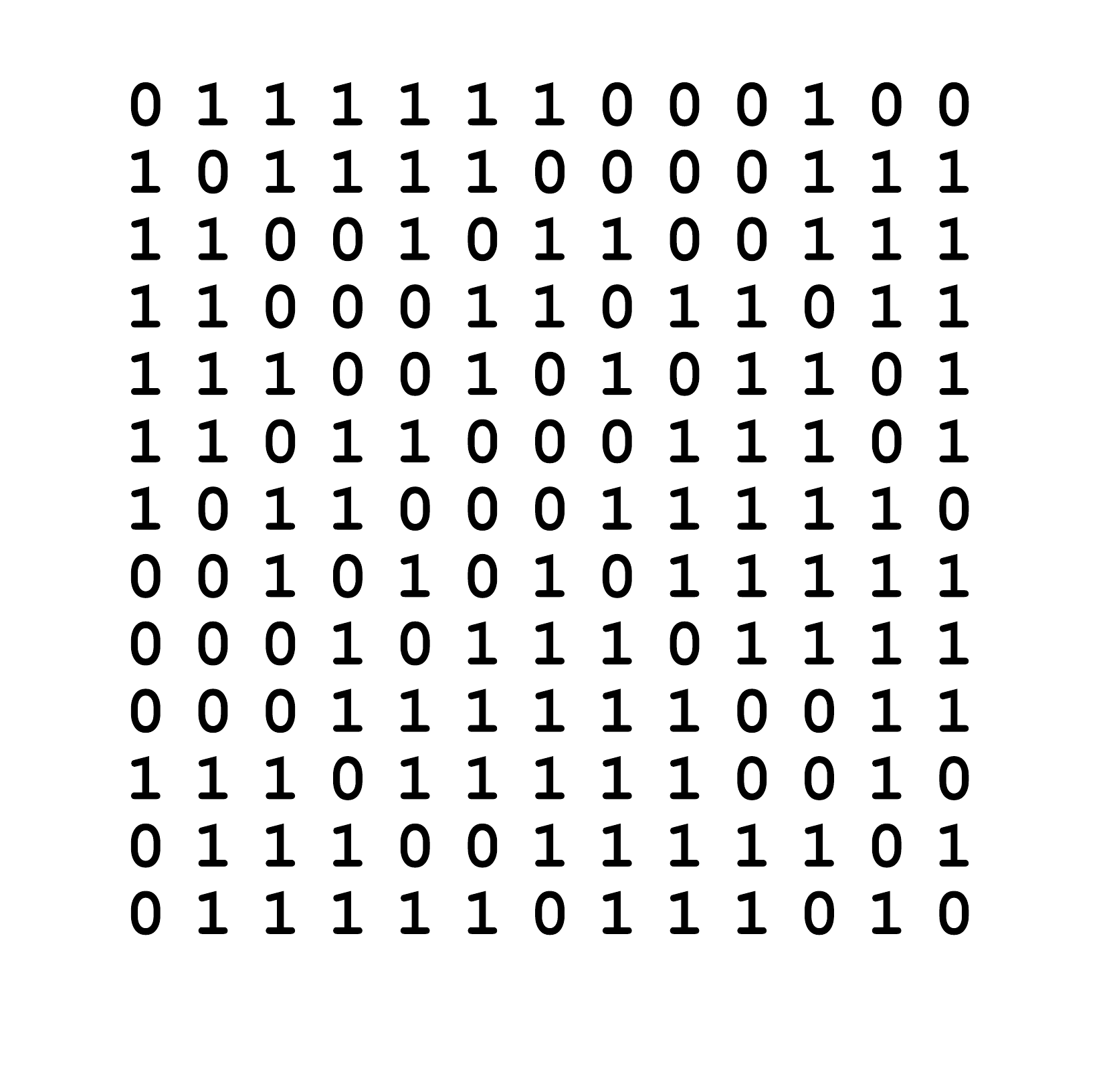}
		\vspace{-1.5em}
		\caption*{$G_{13.306460}$}
		\label{figure: 13_306460}
	\end{subfigure}
	
	\begin{subfigure}{\textwidth}
		\centering
		\includegraphics[trim={0 0 0 490},clip,height=100px,width=100px]{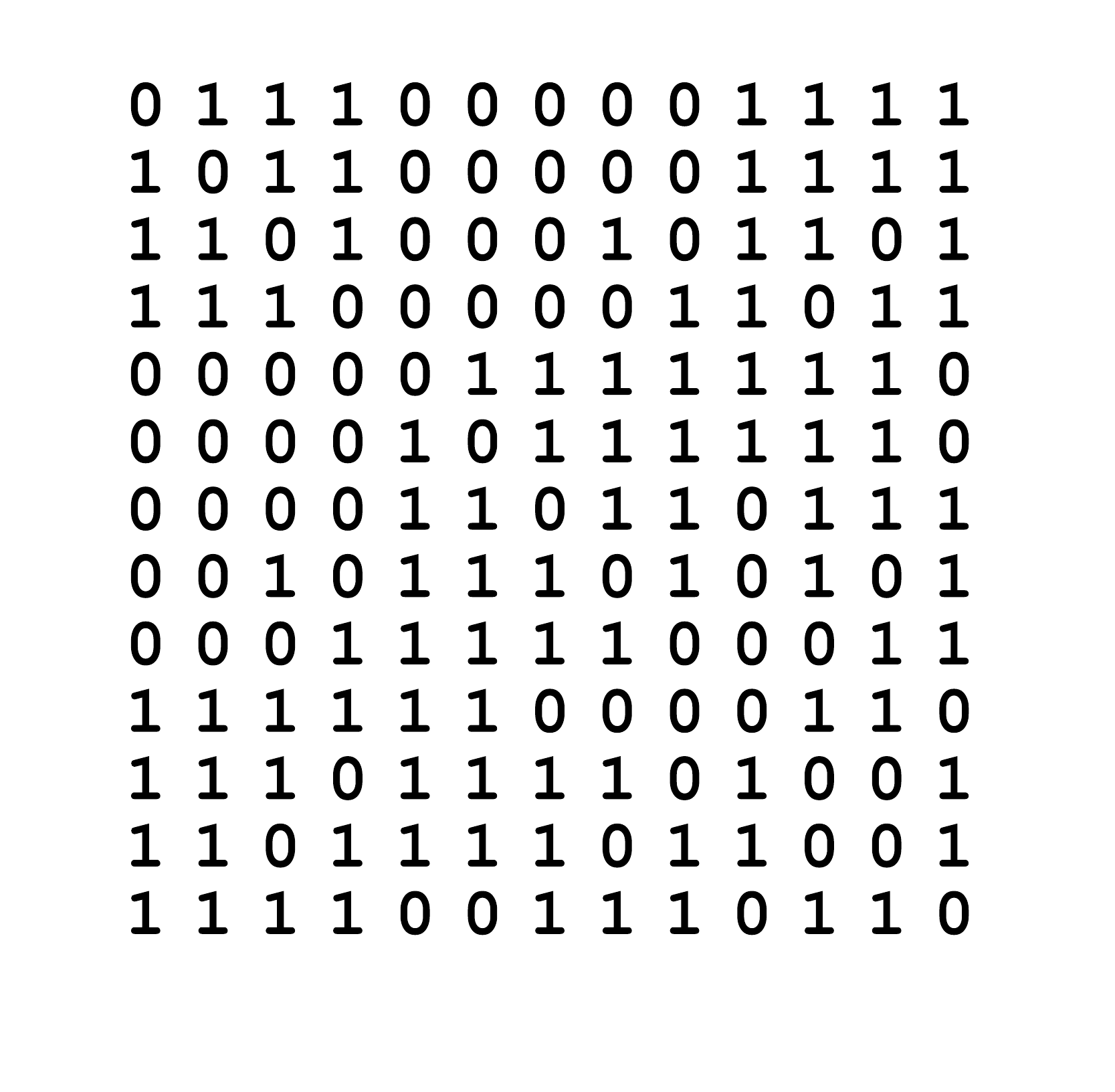}
		\vspace{-1.5em}
		\caption*{$G_{13.306470}$}
		\label{figure: 13_306470}
	\end{subfigure}
	
	\caption{13-vertex minimal $(3, 3)$-Ramsey graphs\\ with independence number 2}
	\label{figure: 13_a2}
\end{figure}


\clearpage

\centerline{REFERENCES}

\small

\bibliographystyle{plain}

\bibliography{main}


\vspace{12pt}
\baselineskip10pt

\vskip10pt
\begin{flushright}
\end{flushright}
\vskip20pt
\footnotesize
\begin{flushleft}
Faculty of Mathematics and Informatics\\
"St.~Kl.~Ohridski" University of Sofia\\
5 ~J.~Bourchier blvd., BG-1164 Sofia\\
BULGARIA \\
e-mail: asbikov@fmi.uni-sofia.bg
\end{flushleft}

\end{document}